\documentclass[12pt]{amsart}
\date{12 February  2016}
\usepackage{latexsym,amsmath,amsfonts,amscd,amssymb,amsthm}
\usepackage{lscape}
\usepackage{stmaryrd}
\usepackage{multirow}
\usepackage[english]{babel}
\usepackage[latin1]{inputenc}
\usepackage{array}
\usepackage[mathscr]{eucal} 
\usepackage{mathrsfs}
\usepackage{graphicx}
\usepackage{xypic}
\usepackage{calligra}
\usepackage[T1]{fontenc}
\usepackage[all]{xy}
\usepackage{enumerate}
\newcommand{\SortNoop}[1]{}

\newcommand{\plonge}{\hookrightarrow}

\newcommand{\projects}{\twoheadrightarrow}

\newcommand{\laction}{\curvearrowright}
\newcommand{\ol}[1]{\overline{#1}}

\newcommand{\mc}[1]{\mathcal{#1}}
\newcommand{\wh}[1]{\widehat{#1}}
\newcommand{\wt}[1]{\widetilde{#1}}

\newcommand{\lie}[1]{\mathfrak{#1}}
\newcommand{\ad}{\mathrm{ad}}
\newcommand{\Inn}{\mathrm{Int}}

\newcommand{\Aut}{\mathrm{Aut}}
\newcommand{\Ad}{\mathrm{Ad}}

 \newcommand{\rank}{\mathrm{rk}}
\newcommand{\Ima}{\mathrm{Im}}
\newcommand{\Lie}{\mathrm{Lie}}
\newcommand{\Char}{\mathrm{Char}}
\newcommand{\Split}{\mathrm{split}}

\newcommand{\benum}{\begin{enumerate}}
\newcommand{\eenum}{\end{enumerate}}

\newcommand{\R}{\mathbb{R}}

\newcommand{\C}{\mathbb{C}}
\newcommand{\Z}{\mathbb{Z}}
\newcommand{\E}{\mathbb{E}}

\renewcommand{\H}{\mathbb{H}}

\renewcommand{\sl}{\mathfrak{sl}}
\renewcommand{\sp}{\mathfrak{sp}}
\newcommand{\su}{\mathfrak{su}}
\renewcommand{\u}{\mathfrak{u}}
\newcommand{\gl}{\mathfrak{gl}}
\newcommand{\so}{\mathfrak{so}}

\newcommand{\Hom}{\mathrm{Hom}}
\newcommand{\Ker}{\mbox{Ker}}
\renewcommand{\Im}{\mathrm{Im}}
\newcommand{\SL}{\mathrm{SL}}
\newcommand{\SU}{\mathrm{SU}}
\newcommand{\U}{\mathrm{U}}
\newcommand{\GL}{\mathrm{GL}}
\newcommand{\SO}{\mathrm{SO}}
\newcommand{\Sp}{\mathrm{Sp}}
\newcommand{\Spin}{\mathrm{Spin}}
\newcommand{\s}{\mathfrak{s}}
\newcommand{\sll}{\mathfrak{sl}}
\newcommand{\g}{\mathfrak{g}}
\newcommand{\h}{\mathfrak{h}}

\newcommand{\lp}{\mathfrak{p}}
\newcommand{\la}{\mathfrak{a}}

\newcommand{\z}{\mathfrak{z}}
\newcommand{\m}{\mathfrak{m}}

\newcommand{\mr}{{\m^\C}_{reg}}

\newcommand{\lt}{\mathfrak{t}}
\newcommand{\lu}{\mathfrak{u}}
\newcommand{\lc}{\mathfrak{c}}
\newcommand{\ld}{\lie{d}}

\newcommand{\tg}{\widehat{\g}}
\newcommand{\TIG}{\widehat{G}}
\newcommand{\TIH}{\widehat{H}}
\newcommand{\tih}{\widehat{\h}}
\newcommand{\tm}{\widehat{\m}}
\newcommand{\tmr}{{\widehat{\m}^\C}_{reg}}
\newcommand{\TLambda}{\widehat{\Lambda}(\la)}
\newcommand{\TLambdac}{\widehat{\Lambda}(\la^\C)}
\newcommand{\tth}{\widehat{\theta}}

\newcommand{\moduli}[1]{\mc{M}(#1)}

\newcommand{\bweyl}{W(\la^\C)}

\newtheorem{thm}{Theorem}[section]
\newtheorem{prop}[thm]{Proposition}
\newtheorem{cor}[thm]{Corollary}
\newtheorem{rk}[thm]{Remark}
\newtheorem{lm}[thm]{Lemma}
\theoremstyle{definition}

\newtheorem{ex}[thm]{Example}
\newtheorem{defi}[thm]{Definition}

\newtheorem*{thm*}{Theorem}

\newcommand{\cM}{\mathcal{M}}

\newcommand{\PU}{\mathrm{PU}}

\newcommand{\Pic}{\operatorname{Pic}}

\renewcommand{\phi}{\varphi}

\newcommand{\liem}{\mathfrak{m}}

\newcommand{\liemc}{\mathfrak{m}^{\mathbb{C}}}
\newcommand{\lieh}{\mathfrak{h}}

\newcommand{\lieg}{\mathfrak{g}}

\newcommand{\liez}{\mathfrak{z}}

\newcommand{\liegc}{\mathfrak{g}^{\mathbb{C}}}

\renewcommand{\phi}{\varphi}

\begin{document}

\title[Higgs bundles for real groups  and the Hitchin--Kostant--Rallis section]
{Higgs bundles for real groups and the Hitchin--Kostant--Rallis section}
\author[Oscar Garc{\'\i}a-Prada]{Oscar Garc{\'\i}a-Prada}
\address{Instituto de Ciencias Matem\'aticas \\
  CSIC-UAM-UCM-UC3M \\ Nicol\'as Cabrera 13 \\ 28049 Madrid \\ Spain}
\email{oscar.garcia-prada@icmat.es}
\author[Ana Pe{\'o}n-Nieto]{Ana Pe{\'o}n-Nieto}
\thanks{The second author was supported by a FPU grant from Ministerio de Educaci{\'o}n}
\address{Ruprecht-Karls-Universit\"at, Heidelberg
Mathematisches Institut\\Im Neuenheimer Feld 288\\
69120 Heidelberg\\Germany }
\email{apeonnieto@mathi.uni-heidelberg.de}
\author[S. Ramanan]{S. Ramanan}
\address{Chennai Mathematical Institute\\
H1, SIPCOT IT Park, Siruseri\\
Kelambakkam 603103\\
India}
\email{sramanan@cmi.ac.in}



\begin{abstract}
We consider the moduli space of polystable $L$-twisted $G$-Higgs bundles over 
a compact Riemann surface $X$, where $G$ is a real reductive Lie group, and $L$
is a holomorphic line bundle over $X$. Evaluating the Higgs field at 
a basis of the ring of  polynomial invariants of the isotropy representation, 
one defines the  Hitchin map. 
This is a map  to an affine space,  whose dimension is determined by $L$ and 
the degrees of the polynomials in  the  basis. Building up on the work of
Kostant--Rallis and Hitchin, in this paper,  we construct a section of this map.
This generalizes the section  constructed by Hitchin  when $L$ is the 
canonical line bundle of $X$ and $G$ is complex. 
In this case the image of the section is related to the
Hitchin--Teichm\"uller  components of the moduli space of 
representations of the fundamental
group of $X$ in $G_{\Split}$, a split real form of $G$. In fact, 
our construction is very natural in that  we can
start with the moduli space for  $G_\Split$, instead of  
$G$, and construct the section for the Hitchin map for $G_\Split$ directly.
The construction involves the notion of maximal split subgroup of a 
real reductive Lie group.  
\end{abstract}

\maketitle
\tableofcontents

\section{Introduction}
Let $G$ be a \textbf{real reductive Lie group}. Following Knapp \cite{K}, by 
this we mean a tuple
$(G,H,\theta,B)$, where $H \subset G$ is a maximal compact subgroup, 
$\theta\colon \lieg
\to \lieg$ is a Cartan involution and $B$ is a
non-degenerate bilinear form on $\lieg$, which is $\Ad(G)$-
and $\theta$-invariant, satisfying natural compatibility conditions. We will also need the notion of a
real strongly reductive Lie group (see
Definition \ref{K-reductive} for details). The Cartan involution $\theta$ gives a
decomposition  
(the Cartan decomposition)
\begin{displaymath}
\lie{g} = \lieh \oplus \liem
\end{displaymath}
into its $\pm1$-eigenspaces, where  $\lieh$ is the Lie
algebras of $H$. The group $H$ acts linearly on $\liem$ through the adjoint
representation of $G$ --- this is the isotropy representation that we
complexify to obtain a representation (also referred as isotropy representation)
\begin{math}
  \iota\colon H^\C \to \GL(\liemc).    
\end{math}

Let $X$ be a compact Riemann surface and $L$ be a holomorphic line bundle 
over $X$. A $L$-twisted $G$-Higgs bundle on $X$ is a pair
$(E,\varphi)$, where $E$ is a holomorphic principal $H^\C$-bundle
over $X$ and $\varphi$ is a holomorphic section of $E(\liemc)\otimes L$,  
where $E(\liemc)= E \times_{H^{\C}}\liemc$ is the $\liemc$-bundle associated to 
$E$ via the isotropy representation.  
The section $\varphi$ is called the  Higgs field.  Two $L$-twisted  
$G$-Higgs bundles $(E,\varphi)$ and $(E',\varphi')$ are isomorphic 
if there is an isomorphism $f\colon E \to  E'$ such that 
$\varphi = f^*\varphi'$ where $f^*$ is the obvious induced map.
When $L$ is the canonical line bundle $K$ of $X$ we obtain the  familiar
theory of $G$-Higgs bundles. When $G$ is compact the Higgs field is
identically zero and a $L$-twisted $G$-Higgs bundle is simply a
principal $G^\C$-bundle. When $G$ is complex $G=H^\C$ and the isotropy
representation coincides with the adjoint representation of $G$.
This is the situation originally considered by Hitchin in \cite{SDE,Duke},
for $L=K$. It is worth point out that considering the theory for an arbitrary
line bundle $L$ is indeed relevant, as illustrated for example in 
the works \cite{BNR,Ngo}. In fact, even in the study of $G$-Higgs bundles 
for $L=K$ one may end up with a different twisting, like in the case of 
maximal Toledo invariant $G$-Higgs bundles (see \cite{BGR}).

There is a notion of stability which depends on an element $\alpha$ of the 
centre of  $\lieh$. This element is fixed by the topology of the bundle,
except in the case in which  $G/H$ is a Hermitian symmetric space. In this 
situation 
$\alpha$  is a  continuous parameter, which varies in a way governed by 
the Milnor--Wood inequality 
(see \cite{BGR}).
Let $\mc{M}_L^\alpha(G)$ the  moduli space of isomorphism classes of
$\alpha$-polystable $L$-twisted  $G$-Higgs bundles. We will omit in the
notation the subindex $L$ when $L=K$. We will also omit the superindex 
$\alpha$ when  $\alpha=0$.

In a similar way to that done by  Hitchin when $G$ is complex, to study 
this moduli space one considers the Hitchin map 
$$
h_L:\mc{M}_L^\alpha(G)\to B_L(G)
$$
defined by evaluating the Higgs field at a basis of the ring of 
polynomial $H^\C$-invariants of the isotropy representation, and 
$B_L(G)\cong H^0(X,\oplus_{i=1}^aL^{m_i})$ is the Hitchin base, where $a$ 
is the real rank of the group and $m_i-1$ are the exponents of $G$
(see Section \ref{section construction HKR section} for a more intrinsic definition of this map, and
the definition of exponents). Again we will omit the subindex $L$ in 
$h_L$ and $B_L(G)$ when $L=K$. As a first step to analyse the Hitchin map, 
in this paper,  we construct a section under certain conditions. 
This generalizes the construction given by Hitchin, when $G$ is complex 
and $L=K$ \cite{Teich}. In this case the image of the section is related to the
Hitchin components of the moduli space of representations of the fundamental
group of $X$ in $G_{\Split}$, a split real form of the complex group $G$. In fact, 
in relation to this, our construction is indeed very natural since we can
start directly with the moduli space $\cM(G_\Split)$ instead of  
$\cM(G)$ and construct the section for the Hitchin map for $G_\Split$
instead of that for $G$, which by construction lies in $\cM(G_\Split)$.  
It is important to point out that $B_L(G)=B_L(G_\Split)$.

Sections \ref{liealg} and \ref{Reductive groups} establish the Lie theoretical results necessary for the sequel. 
Section \ref{liealg}
is essentially introductory: we recall the Cartan theory for reductive complex
Lie algebras in Section \ref{section rforms algebras}. 
 Section \ref{section max split} reviews 
the construction of the maximal split subalgebra $\tg$ of any real reductive Lie algebra $\g$, due to 
Kostant--Rallis \cite{KR71}.

In Section \ref{Reductive groups} we study real reductive Lie groups following
Knapp's  definition (\cite[Chap. VII]{K}). 
We extend classical  structural results in Lie 
theory, such as closedness of reductivity by involutions (Proposition \ref{prop red closed by inv}), or
basic results used in the Cartan theory of groups (Proposition \ref{maximal 
compact and theta enough}). All of this is done in Section \ref{subsec real
  red gps}. 
Let $(G,H,\theta, B)$ be a real reductive Lie group  in the sense of Definition 
 \ref{K-reductive}. The main aim of Section \ref{subsect real forms} is to study the interplay between involutions $\iota$ of $G$
 and the fixed point subgroup $G^\iota$, as well as the relations with adjoint groups and normalising subgroups. The main result in this direction
 is Proposition \ref{prop Giota}, which specialises to real forms of complex
 reductive Lie groups in  Corollary \ref{prop Gsigma reductive}. All of
 these results are essential for Sections \ref{section Higgs pairs} and \ref{section construction HKR section}. 
 Section \ref{max split subgp} deals with the construction of a maximal split subgroup 
 $$(\TIG_0,\TIH_0,\tth,\widehat{B})\leq(G,H,\theta, B)$$ 
 (see Propositions 
\ref{defi maximal connected split} and \ref{maximal split reductive}). We use results by Borel and Tits \cite{BTGroupes, BTComplements} to study the 
connections between the topology of both groups (Corollary \ref{maximal split as finite cover}), which will be used in Section \ref{section topo type}.

Section \ref{section kostant-rallis section} generalizes part of the work of Kostant and Rallis \cite{KR71} to our context. More precisely, given $\g$ the 
reductive Lie 
algebra of a reductive Lie group $G$, consider its Cartan decomposition $\g=\h\oplus\m$, where $\h=\Lie(H)$ for some maximal compact subgroup $H\leq G$.
We study the Chevalley morphism $\chi: \m^\C\to\m^\C\sslash H^\C$ and in particular the existence of a section of this morphism (see Theorem \ref{theorem KR section}).
We hereby note the prominent role
of real forms of quasi-split type in the whole theory (see Lemma \ref{lemma centraliser tds}\textit{2.}).  

We recall the basics on moduli spaces of Higgs bundles in Section \ref{section
  Higgs pairs}, following  \cite{GGMHitchinKobayashi}. 
 The results in this section are not original with the exception  perhaps
of Proposition 
\ref{prop morphism moduli}.

The main result of this paper is in Section \ref{section construction HKR section}, where we generalize Hitchin's construction
of a section
of the Hitchin map \cite{Teich}. This yields Theorem \ref{thm HKR Hermitian}, which reads as follows.
\begin{thm*}
Let $(G,H,\theta, B)$ be a strongly reductive Lie group, and let  $(\TIG_0,\TIH_0,\tth,\widehat{B})$ be its maximal connected split subgroup. 
Let $L\to X$ be a line bundle with degree $d_L\geq 2g-2$.
Let  $\alpha\in i\z(\so(2))$ be such that $\rho'(\alpha)\in\z(\h)$, where
$\rho':\so(2)\to \lieh$ is given by (\ref{eq rho' bis}). 
 Then, the choice of a square root of $L$ determines $N$ non equivalent sections of the map
$$
h_L: \mc{M}_L^{\rho'(\alpha)}(G)\to B_{L}(G).
$$
Here, $N$ is the number of cosets in $\Ad(G)^\theta/\Ad(H)$.

Each such section $s_G$ satisfies the following:
\benum
\item[1.] If $G$ is quasi-split, $s_G(B_L(G))$ is contained in the stable locus of $\mc{M}_L^{\rho'(\alpha)}(G)$, and in the smooth locus
if $Z(G)=Z_G(\g)$ and $d_L\geq 2g-2$.
\item[2.] If $G$ is not quasi-split, the image of the section is contained in the strictly polystable locus.
\item[3.] For arbitrary groups, the Higgs field is everywhere regular.
\item[4.] If $\rho'(\alpha)\in i\z\left({\tih}\right)$, the section factors through $\mc{M}_L^{\rho'(\alpha)}(\TIG_0)$. This is in particular
the case if $\alpha=0$.
\item[5.] If $G_{\Split}<G^\C$ is the split real form of a complex reductive Lie group, 
$K=L$ and $\alpha=0$, $s_G$ is the factorization of the Hitchin section through $\moduli{G_{\Split}}$.
\eenum
\end{thm*}
We will refer to a  section defined as  above as  a \textbf{Hitchin--Kostant--Rallis}, abbreviate HKR section for short.

Due to the degree of generality
in which we have chosen to work, we need to develop the theory with new tools. A remarkable fact is that
the the image of the section need not be smooth, even when the group is connected, of adjoint type, and the twisting is the canonical bundle. 
This differs from the complex group case
studied by Hitchin in \cite{Teich}, and is due to the fact that split groups are quasi-split (see Propositions \ref{basic pair stable}
and Corollary \ref{cor basic pair smooth}). After some analysis in Section \ref{section reminder} of the representation theory involved 
(note the differences with the complex case pointed out in Corollary \ref{cor TDSgps}), 
we move on in Section \ref{sect SL2 Higgs pair} to study the basic case: the HKR section for
$\SL(2,\R)$-Higgs bundles. The latter is then used in Section \ref{sect basic pair} to produce a $G$-Higgs bundle, which will be deformed
to yield a section of the Hitchin map, analysis done in Section \ref{section non Hermitian}. We use the results in this section to prove
  in Proposition \ref{prop component HKR} that for quasi-split groups $G$,
the image of the section covers  a connected component of the moduli space if and only if the real group is split.
We include in Section \ref{section regularity} a geometric 
interpretation of the algebraic notion of regularity.

The topological type  of the elements in the  image of the HKR section is 
studied in Section \ref{section topo type}. 
We study the Hermitian and non Hermitian cases separately. 
In the first case, an answer is given in  Proposition \ref{prop toledo}. In
the second case, however, the answer depends on the topological type of elements
of the Hitchin section for the maximal split subgroup. We deal with this in Proposition \ref{prop top invariant non hermitian}.


\section{Reductive Lie algebras and maximal split subalgebras}\label{liealg}

A \textbf{reductive Lie algebra} over a field $k$ is a Lie algebra $\g$ over $k$ whose adjoint representation
is completely reducible. Semisimple Lie algebras are reductive. It is well known that any reductive Lie algebra 
decomposes as a direct sum 
$$
\g=\g_{ss}\oplus\z(\g)
$$
where $\g_{ss}=[\g,\g]$ is a semisimple Lie subalgebra (the 
semisimple part of $\g$) and $\z(\g)$ is the centre of $\g$, thus an abelian subalgebra.

We will focus on Lie algebras over the real and complex numbers and the relation between them. As a
first example, note
that any complex reductive Lie algebra ${\g^\C}$ with its underlying real structure $\left({\g^\C}\right)_\R$ is a real
reductive Lie
algebra.
On the other hand, given a real reductive Lie algebra $\g$, its complexification
$\g^\C:=\g\otimes_\R\C$ is a
complex reductive Lie algebra.
\subsection{Real forms of complex Lie algebras}\label{section rforms algebras}
A \textbf{real form} $\g\subset{\g^\C}$ of a complex Lie algebra  ${\g^\C}$ is the subalgebra of fixed
points of  an antilinear
involution $\sigma\in \Aut_{2}\left(\left({\g^\C}\right)_\R\right)$, where $\Aut_{2}\left(\left({\g^\C}\right)_\R\right)$ denotes the subset of order two
automorphisms of the real Lie algebra underlying $\g^\C$. Equivalently, it is a real subalgebra $\g\subset\g^\C$ such that
the natural homomorphism of $\C$-algebras $\g\otimes\C\to\g^\C$ is an isomorphism.

Any real Lie algebra $\g$ is a real form of its complexification ${\g^\C}:=\g\otimes_\R\C$ with associated involution
 $\g^\C\cong_\R \g\oplus\g\ni(X,Y)\mapsto (X,-Y)$. Also, given a complex reductive Lie algebra ${\g^\C}$, one can obtain it as a real form of ${\g^\C}\otimes\C$
 by choosing a
 maximal compact subalgebra $\lie{u}\subset{\g^\C}$ (i.e., a real subalgebra whose adjoint group is compact). Let
 $\tau\in \Aut_\R\left(\left({\g^\C}\right)_\R\right)$ be the  antilinear involution 
 defining $\u$. Then, considering ${\g^\C}\otimes\C\cong{\g^\C}\oplus{\g^\C}$, define on it the antilinear involution 
 $$
 \tau^\C(x,y):=(\tau(x),-\tau(y)),
 $$
whose subalgebra of fixed points is isomorphic to $\lie{u}\oplus i\lie{u}\cong\left({\g^\C}\right)_\R$.

Two real forms $\g$ and $\g'$ of ${\g^\C}$ 
(defined respectively by antilinear involutions 
$\sigma,\ \sigma'\in \Aut_\R(\left({\g^\C}\right)_\R)$) are  Cartan
isomorphic, denoted by $\sigma \sim_c \sigma'$, 
if  there exists 
$\phi\in\Aut_\C({\g^\C})$ making the following diagram commute
$$
\xymatrix{
{\g^\C}\ar[r]^\phi\ar[d]_\sigma&{\g^\C}\ar[d]^{\sigma'}\\
{\g^\C}\ar[r]_\phi&{\g^\C}.
}
$$
We will consider the stronger equivalence condition, that we will denote 
by  $\sigma\sim_i\sigma'$ if furthermore $\phi$ can be chosen inside the 
group of inner automorphisms of the Lie algebra $\Inn_\C({\g^\C})$.


It is well known (see for example  \cite[Sec. 3]{O}) that there 
exists a correspondence between  isomorphism classes (under equivalence
$\sim_c$ or $\sim_i$) of real forms of a complex semisimple Lie algebra
$\g^\C$ and orbits of $\C$-linear involutions (under $\Inn(\g^\C)$, resp. $\Aut(\g^\C)$) of ${\g^\C}$. This correspondence is obtained by
composing the involution defining the real form with a commuting involution defining a compact form. Both forms are then said to be compatible.
\begin{prop}\label{correspondence reductive}
Given a complex reductive Lie algebra ${\g^\C}$, and a compact real form  $\u$ of $\g^\C$, there is a 1-1 correspondence between conjugacy
classes under $\sim_i$ 
of real forms compatible with $\u$ and  conjugacy classes under $\sim_i$ of 
linear automorphisms $\theta:{\g^\C}\to{\g^\C}$.
\end{prop}
\begin{proof}
We note first that involutions  of a Lie algebra leave the semisimple part and the centre invariant. This, together with  Theorem 3.2 in \cite{O} implies that
it is enough to prove the proposition for abelian Lie algebras, that is, vector spaces. 

Let $\g^\C$ be an abelian Lie algebra of dimension $n$. A choice of basis allows to identify it with $\C^n$. A real form  $\g$ is a real 
subspace of dimension $n$, which is 
the set of fixed points of the reflection with respect to $\g$. Note that the only compact real form is $(i\R)^n\subset\C^n$, 
as if $v_1,\dots, v_n$
are the real vectors expanding the subspaces, exponentiation of any vector that is not purely imaginary contains a spiral which is non compact (as
real forms of $\C$ are in correspondence with real vectorial lines in $\C\cong \R^2$ which exponentiate to $\U(1)$ or spirals--the case of $\R$ corresponds
to the degenerate spiral).

Now, the only real form compatible with $(i\R)^n$ is a direct sum of copies of $\R$ and $i\R$. On the other hand, compatible involutions with 
$\sigma:(z_1,\dots, z_n)\mapsto -(\ol{z_1},\dots, \ol{z_n})$ are combinations of complex conjugation and mutiplication by $\pm1$ on the factors and transpositions, which composed with 
$\sigma$ yield all possible linear involutions of $\C^n$, that is, transpositions and multiplication by $\pm1$.
\end{proof}

\begin{rk}
Proposition \ref{correspondence reductive} classifies real forms of an 
abelian Lie algebra up to $\sim_i$ equivalence. Note that the result does not depend on the choice of a compact form, as neither does the result for semisimple
algebras, and the  compact form of the centre 
is unique, but we are forced to consider compatible real forms. If we considered real forms up to outer isomorphism, then the compact form and the split one would be identified.
\end{rk}
An involution of a real reductive Lie algebra $\g$ defining a maximal compact form is called a
 \textbf{Cartan involution}. The decomposition of $\g$ into $(+1)$ and $(-1)$-eigenspaces is
 a \textbf{Cartan decomposition}. Any such has the form 
\begin{equation}\label{eq cartan dec}
\g=\h\oplus\m 
\end{equation}
satisfying the relations
$$
[\h,\h]\subseteq\h,\qquad [\m,\m]\subseteq\h,\qquad 
[\h,\m]\subseteq\m.
$$
In particular, we have an action $\iota:\h\to\gl(\m)$ induced by the adjoint action of $\g$ on itself, which
is called the infinitesimal \textbf{isotropy representation}.

Involutions produce new Lie algebras.
\begin{prop}\label{prop red closed by inv}
The class of reductive Lie algebras is closed by taking fixed points of involutions. 
\end{prop}
\begin{proof}
By the preceeding discussion, it is enough to prove the statement for simple Lie algebras, as any extension
of a simple Lie algebra by a central subalgebra is reductive, and all reductive Lie algebras are a direct sum
of algebras of this kind. Now, any Lie algebra $\g$ it is a real form of
its complexification $\g^\C$. Given $\iota$ and involution of $\g$, we may extend it to a $\C$-linear 
involution of $\g^\C$. Then, the Cartan theory for semisimple Lie algebras and Theorem 2.1  
imply that $(\g^\C)^\iota=\h^\C$ for some compact Lie subalgebra $\h\subset\g$. But
\cite[I.11]{O} implies that
$\h$ is reductive.
\end{proof}
\begin{rk}
 The above proves that fixed points of involutions of simple Lie algebras are reductive, but not necessarily semisimple. 
 For example, the maximal compact subalgebra $\u(2)\subset\sp(4,\R)$ is fixed by the Cartan involution and is reductive, but not simple or semisimple.
\end{rk}
\subsection{Maximal split subalgebras and restricted root systems}\label{section max split}
Let $\g$ be a real reductive Lie algebra with a Cartan involution $\theta$ decomposing
$\g$ as $\g={\h}\oplus\m.$
Given a maximal subalgebra 
$\la\subset\m$ it follows from the definitions that it must be abelian, and one can easily 
prove that its elements are semisimple and diagonalizable over the real numbers (cf. \cite[Chap.VI]{K}, 
note that
Knapp proves it for semisimple Lie algebras, but for reductive Lie algebras it suffices to use invariance of the
centre
and the semisimple part of $[\liegc,\liegc]$) under the Cartan involution.
Any such subalgebra is called a \textbf{maximal anisotropic} Cartan subalgebra of $\g$. By extension,
its complexification $\la^\C$ is called a \textbf{maximal anisotropic} Cartan subalgebra of $\g^\C$ (with respect
to $\g$). A
maximal anisotropic Cartan subalgebra $\la$ can be completed to a $\theta$-equivariant \textbf{Cartan subalgebra} of $\g$,
namely, a subalgebra whose complexification is a Cartan subalgebra of $\g^\C$. Indeed, define
\begin{equation}\label{eq max split cartan}
\ld=\lt\oplus\la 
\end{equation}
where $\lt\subset\lc_\h(\la):=\{x\in\h\ :\ [x,\la]=0\}$ is a maximal abelian subalgebra 
(\cite{K}, Proposition 6.47). Cartan subalgebras of this kind
(and their complexifications) are called \textbf{maximally split}.

The dimension of maximal anisotropic Cartan subalgebras of a real reductive Lie algebra $\g$ is called the \textbf{the
real (or split) rank of }
$\g$. This number measures the degree of compactness of real forms: indeed,  a real form 
is compact (that is, its adjoint group is compact)
 if and only if $\mathrm{rk}_\R(\g)=0$. On the other hand, a real form is defined to
be \textbf{split} 
if $\mathrm{rk}_\R(\g)=\mathrm{rk}{\g^\C}$. Note that the split rank depends on the involution $\theta$ associated with the real form, when $\g$ is not semisimple.

The restriction to $\la$ of the adjoint representation of $\g$ yields a decomposition
of $\g$ into 
$\la$-eigenspaces
$$\g=\bigoplus_{\lambda\in\Lambda(\la)}{\g}_{\lambda},$$
where $\Lambda(\la)\subset\la^{*}$ is called the set of \textbf{restricted roots} of $\g$ with
respect
to $\la$.
The set $\Lambda({\la})$ forms a root system (see \cite[Chap. II, Sec. 5]{K}), which may not
be reduced (that is, there may be roots whose double is also a root).
The name restricted roots is due to the following fact:  extending restricted roots by $\C$-linearity, we obtain $\Lambda({\la^\C})\subset \left(\la^\C\right)^*$, also called 
restricted roots.
Now, take a maximally split $\theta$-invariant
Cartan subalgebra $\ld\subset\g$  as in (\ref{eq max split cartan}), and let $\Delta(\g^\C,\ld^\C)$ be the corresponding
set of roots; then, restricted roots are 
restrictions of roots. In fact, a root $\gamma\in\Delta(\g^\C,\ld^\C)$ decomposes as 
\begin{equation}\label{eq dec root}
\gamma=\lambda+i\beta 
\end{equation}
where $\lambda$ is the extension by complex linearity of an element in $\la^*$ and $\beta$
 is the extension by complex linearity of an element $\lt^*$. This implies $\gamma|_{\la^\C}=\lambda|_{\la^\C}$.
We can decompose $\Delta(\g^\C,\ld^\C)=\Delta_i\cup\Delta_r\cup\Delta_c$ where 
\begin{eqnarray}\label{eq types roots}
 \Delta_i=\{\gamma\in\Delta\ :\ \gamma|_{\la^\C}\equiv 0\},\\\nonumber 
 \Delta_r=\{\gamma\in\Delta\ :\ \gamma|_{\lt^\C}\equiv 0\},\\\nonumber
 \Delta_c=\Delta\setminus(\Delta_i\cup\Delta_r)
 \end{eqnarray}
 are respectively called \textbf{imaginary, real } and \textbf{complex }roots.
 
 In \cite{KR71}, Kostant and Rallis give a procedure to construct a $\theta$-invariant subalgebra 
$\tg\subset\g$ such $\tg\subset(\tg)^\C$ is a split real form, whose Cartan subalgebra is 
${\la}$ and such that $\z(\tg)=\z(\g)\cap\m$. Their construction relies on the following notion.
\begin{defi}\label{defi TDS} A \textbf{three dimensional subalgebra} (TDS) $\lie{s}^\C\subset{{\g^\C}}$ 
is the image of 
an injective morphism $\sll(2,\C) \to{\g^\C}$.
 A TDS is called \textbf{normal} if  $\dim \lie{s}^\C\cap{\h^\C}=1$ and $\dim \lie{s}^\C\cap{\m^\C}=2$.
 It is called \textbf{principal} if it is generated by elements $\{e,f,x\}$, 
where $e$ and $f$ are nilpotent regular elements in $\m^\C$ (cf. Definition \ref{def
  regular}),  and $\ {x}\in\lie{h}^\C$ is semisimple. 
A set of generators satisfying such relations is called a
\textbf{normal basis} or \textbf{normal triple}. 
\end{defi}
\begin{defi}\label{defi maximal split}
A subalgebra $\tg\subset\g$ generated by $\la$ and $\lie{s}^\C\cap \g$, where $\lie{s}^\C$ is a 
principal normal TDS invariant by the involution defining $\g$ inside of $\g^\C$
 is called a \textbf{maximal split subalgebra}.
\end{defi}
Maximal split subalgebras can be constructed very explicitely; for this, consider the following reduced system 
of roots
\begin{equation}\label{eq TLambda}
 \TLambda=\{\lambda\in\Lambda({\la})\ |\ \lambda/2\notin\Lambda({\la}) \}.
\end{equation}
Let $\{\lambda_1,\dots,\lambda_a\}=\Sigma({\la})\subset\Lambda({\la})$ be a system of simple
restricted roots (cf. \cite[Chap. VI]{K}),
which is also a system of simple roots for $\TLambda$. Let $h_i\in\la$ be
the dual to
$\lambda_i$ with respect to 
some $\theta$ and $\Ad(\exp({\g}))$-invariant bilinear form $B$ satisfying that $B$ is negative definite
on
${\h}$ and positive definite on $\m$. Strictly speaking, in \cite{KR71} they take $B$ 
to be the Cartan-Killing form on $\g$; however, the above
 assumptions are enough to obtain the necessary results hereby quoted.
 Now, for each $\lambda_i\in\Sigma({\la})$ choose $y_i\in{\g}_{\lambda_i}$. We have 
$$
[y_i,\theta y_i]=b_ih_i,
$$
where $b_i=B(y_i,\theta y_i)$.
Indeed, $[y_i,\theta y_i]\in{\la}\cap[{\g},{\g}]$, so it is enough 
to prove that $B([y_i,\theta y_i],x)=B(y_i,\theta y_i)\lambda_i(x)$ for all $x\in{\la}$, which is a
simple calculation.

Consider 
$$z_i=\frac{2}{\lambda_i(h_i)b_i}\theta y_i,\qquad w_i=[y_i,z_i]=
\frac{2}{\lambda_i(h_i)}h_i.
$$
 We have the following (Proposition 23 in \cite{KR71}).
\begin{prop}\label{redsub}
Let $\g\subset{\g^\C}$ be a real form, and let $\sigma$ be the antilinear involution of ${\g^\C}$ defining
$\g$.
Let $\tg$ be the subalgebra generated by all the $y_i,\ z_i,\ w_i$'s as above, and
$\lc_{\m}(\la)$, the centraliser of $\la$ in $\m$. Let $\tg^\C=\tg\otimes\C$. Then 
\benum
 \item[1.]\label{tg reductive} $\tg^\C$ is a $\sigma$- and $\theta$-invariant reductive subalgebra of 
 $\g^\C$. We thus have $\tg^\C=\tih^\C\oplus\widehat{\m}^\C$ where $\tih^\C=\h^\C\cap\tg^\C$, 
 $\widehat{\m}^\C=\m^\C\cap\tg^\C$.
\item[2.]\label{tg split} $\tg\subset \g$ is a maximal split subalgebra as in Definition \ref{defi maximal split}.
  Moreover, the subsystem $\widehat{\Lambda}({\la^\C})\subset\Lambda({\la^\C})$ as defined in (\ref{eq TLambda})
  is the root system of $\tg^\C$ with respect to ${\la^\C}$.
 \eenum
\end{prop}
Since $\TLambdac$  is a reduced root system, we can uniquely 
assign to it a complex semisimple
Lie algebra
$\tg^\C$. In
\cite{Araki} Araki gives the details necessary to obtain $\tg^\C$ (or its Dynkin 
diagram) from the Satake diagram of $\g$ whenever the latter is a simple Lie algebra. 
The advantage of Araki's procedure is that it allows identifying the isomorphism class of $\tg$
easily. However, unlike Kostant and Rallis' method, it does not provide the embedding 
$\widehat{\g}\plonge \g$. 
See \cite{Araki} for details. 
\begin{rk}\label{split form complex}
Let ${\g^\C}$ be a complex reductive Lie algebra, and let $\left({\g^\C}\right)_\R$ be its underlying real reductive
algebra. Then,
the maximal split subalgebra of $\left({\g^\C}\right)_\R$ is isomorphic to the split real form ${\g}_{split}$ of ${\g^\C}$. 
It is clearly split within its complexification and it is maximal within $\left({\g^\C}\right)_\R$ with this property,
which can be easily checked by identifying $\left({\g^\C}\right)_\R\cong\g_{split}\oplus i\g_{split}$.
\end{rk}

\begin{table}
\caption {Maximal split subalgebras} \label{tab} 
\begin{tabular}{|l|l|l|}
  \hline
Type&${\g}$&$\widehat{\g}$\\
\hline
AI&$\lie{sl}(n,\R))$&$\sll(n,\R)$\\
\hline
AII&$\lie{su}^*(2n)$&$\sll(n,\R)$\\
\hline
\multirow{2}{*}{AIII}& $\lie{su}(p,q),\ p<q$& $\lie{so}(p,p+1)$\\\cline{2-3}
 &$ \lie{su}(p,p)$ & $\lie{sp}(2p,\R)$\\
\hline
BI&$\lie{so}(2p,2q+1),\ p\leq q$&$\lie{so}(2p,2p+1)$\\
\hline
CI&$\lie{sp}(2n,\R)$&$\lie{sp}(2n,\R)$\\
\hline
\multirow{2}{*}{CII}& $\lie{sp}(p,q)\ p<q$& $\lie{so}(p,p+1)$\\\cline{2-3}
 &$\lie{sp}(p,p)$ & $\lie{sp}(2p,\R)$\\
\hline
BDI&$\lie{so}(p,q)\ p+q=2n, p<q$&$\lie{so}(p,p+1)$\\
\hline
DI&$\lie{so}(p,p)$&$\lie{so}(p,p)$\\
\hline
\multirow{2}{*}{DII}& $\lie{so}^*(4p+2)\ p<q$& $\lie{so}(p,p+1)$\\\cline{2-3}
 &$\lie{so}^*(4p)$ & $\lie{sp}(2p,\R)$\\
\hline
EI&$\lie{e}_{6(6)}$&$\lie{e}_{6(6)}$\\
\hline
EII&$\lie{e}_{6(2)}$&$\lie{f}_{4(4)}$\\
\hline
EIII&$\lie{e}_{6(-14)}$&$\lie{so}(3,2)$\\
\hline
EIV&$\lie{e}_{6(-26)}$&$\lie{sl}(3,\R)$\\
\hline
EV&$\lie{e}_{7(7)}$&$\lie{e}_{7(7)}$\\
\hline
EVI&$\lie{e}_{7(-5)}$&$\lie{f}_{4(4)}$\\
\hline
EVII&$\lie{e}_{7(-25)}$&$\lie{sp}(6,\R)$\\
\hline
EVIII&$\lie{e}_{8(8)}$&$\lie{e}_{8(8)}$\\
\hline
EIX&$\lie{e}_{8(-24)}$&$\lie{f}_{4(4)}$\\
\hline
FI&$\lie{f}_{4(4)}$&$\lie{f}_{4(4)}$\\
\hline
FII&$\lie{f}_{4(-20)}$&$\lie{sl}(2,\R)$\\
\hline
G&$\lie{g}_{2(2)}$&$\lie{g}_{2(2)}$\\
\hline
\end{tabular}
\end{table}
\section{Reductive Lie groups and maximal split subgroups}\label{Reductive groups}
\subsection{Real reductive Lie groups}\label{subsec real red gps}
Following Knapp \cite[VII.2]{K}, we define reductivity of a Lie group as follows.
\begin{defi}\label{K-reductive}
A \textbf{real reductive group} is a $4$-tuple $(G,H,\theta,B)$ where
\benum
\item $G$ is a real Lie group with reductive Lie algebra $\g$.
\item $H<G$ is a maximal compact subgroup.
\item $\theta$ is a Lie algebra involution of $\g$ inducing an eigenspace decomposition
$$\g=\h\oplus\m$$
where $\h=\mathrm{Lie}(H)$ is the $(+1)$-eigenspace for the action of $\theta$, and $\m$ is the 
$(-1)$-eigenspace.
\item\label{bilinear form} $B$ is a $\theta$- and $\Ad(G)$-invariant non-degenerate bilinear form,
with
respect to which $\h\perp_{B}\m$
and $B$ is negative definite  on $\h$ and positive definite on $\m$.
\item\label{diffeo} The multiplication map $H\times \exp(\m)\to G$ is a diffeomorphism.
\eenum
If furthermore $(G,H,\theta, B)$ satisfies
\begin{enumerate}
\item[(SR)]  $G$ acts by  inner automorphisms on the complexification $\g^\C$ of its Lie algebra via the adjoint representation 
\end{enumerate}
then the group will be called \textbf{strongly reductive}.
\end{defi}
\begin{rk}\label{rk defi k vs ours}
  Note that the definition of Knapp \cite[VII.2]{K} differs from ours in two ways: on the one hand, he assumes (SR) in the definition of reductivity.
  Since we will cite his results, we will need to pay attention to  which of them really use this hypothesis. On the other hand, he does not assume $H$ to be maximal, just compact.
  Maximality in fact results from the polar decomposition.
\end{rk}

\begin{rk}
 If $G^\C$ satisfies condition (SR) in Definition \ref{K-reductive}, then, by definition, $\Ad(G^\C)$ is equal to $\Ad(\g^\C)$, the connected component of $\Aut(\g^\C)$.
\end{rk}
Given a Lie group $G$ with reductive Lie algebra $\g$, the extra data $(H,\theta,B)$ defining a reductive
structure will be refered to as \textbf{Cartan data} for $G$. 

A morphism of reductive Lie groups $(G',H',\theta',B')\to (G,H,\theta,B)$ 
is a morphism of Lie groups $G'\to G$ which respects the corresponding Cartan data in the obvious way.
In particular, a reductive Lie subgroup of a reductive Lie group $(G,H,\theta, B)$ is a reductive 
Lie group $(G',H',\theta',B')$ such that $G'\leq G$ is a Lie subgroup and the Cartan data
$(H',\theta',B')$
is obtained by intersection and restriction.

\begin{rk}\label{rk notation ss gps}
 When the group $G$ is semisimple, letting $B$ be the Killing form, the rest of the Cartan data is fully determined by the
 choice of a maximal compact subgroup $H$. In this case, we omit the Cartan
 data from the notation.
\end{rk}
\begin{lm}\label{centre ss}
 Let $G$ be a semisimple Lie group with maximal compact subgroup $H\leq G$. Then, $Z(G)\leq Z(H)$,
 and equality holds if $G$ is complex.
\end{lm}
\begin{proof}
 From Corollary 7.26 (2) in \cite{K}, we have that $Z(G)=Z_{H}(G)e^{i\z_\m(\g)}$, as the quoted result does notause (SR) in Definition
 \ref{K-reductive}, but semisimplicity implies $\z_\m(\g)=0$, so $Z(G)\leq Z(H)$. Now, if $G$ is a  complex group, 
 given that  $G=He^{i\h}$, and that $Z(H)\subset Z_H(\h)=Z_H(i\h)$, we have that $Z(H)$ 
  centralises the identity component $G_0$. Since any connected component of 
  $G$ is of the form $hG_0$ for some $h\in H$, it follows that $Z(H)$ centralises all connected components,
  and so also $G$.
\end{proof}
\subsection{Real forms of complex reductive Lie groups.}\label{subsect real forms}
A great variety of examples of real reductive Lie groups is provided by real forms of complex reductive
Lie groups. Recall that a real form $G$ of a complex Lie group $G^\C$ is the group of fixed points of an antiholomorphic
involution $\sigma:G^\C\to G^\C$.

Some of the results in this section are common knowledge, but due to the lack of known references covering the general case we include them in this section. Similar results
are also proved in \cite{GRInvolutions}.

The following proposition proves real forms of some complex reductive Lie groups  
inherit a reductive group structure from their complexification.
\begin{prop}\label{maximal compact and theta enough}
Let $(G^\C,U,\tau,B)$ be a connected complex reductive Lie group, and
let $\sigma$ be an antilinear involution of $G^\C$ defining $G=\left(G^\C\right)^\sigma$.  Then, on
$G^\C$, there exists an 
involution conjugate by an inner element $\sigma'=\Ad_g\circ\sigma\circ\Ad_{ g^{-1}}$ such that 
$G'=gGg^{-1}$ can be endowed with Cartan data $(H',\theta',B')$ making it a 
 reductive subgroup of $(G^\C,U,\tau,B)$ in the sense of Definition \ref{K-reductive}.
\end{prop}
\begin{proof}
By Proposition \ref{correspondence reductive} and the fact that all maximal compact Lie subalgebras are conjugate, at the level of the Lie algebras there is 
an inner conjugate of $d\sigma$ that
commutes with $\tau$, say
$(d\sigma)'=\Ad_{ g}\circ d\sigma\circ \Ad_{ g^{-1}}$. We notice that $(d\sigma)'=d\sigma'$ where 
$\sigma'=\Ad_{g}\circ \sigma\circ \Ad_{ g^{-1}}$. So $U_0=\exp(\lu)$ is $\sigma'$-invariant. 

All of this implies that the polar decomposition of $G^\C$ for a choice of Cartan data $(U,\tau, B)$ induces one for $G'=\Ad(g)(G)$. Indeed,  $G'$ is diffeomorphic to $H'\times\exp{\m'}={G^\C}^{\sigma'}$,
where $H'=U^{\sigma'}$, $\exp{\m'}=\exp{\lu}^{\sigma'}$, as any $g\in G'$ can be written as $g=ue^V$ for $u\in U$, $V\in i\lu$, and it must be
$$
u^{\sigma} e^{\sigma' V}=ue^V\iff u^{-1}u^{\sigma}=e^{-\sigma V}e^{V}\in U\cap\exp i\lu=\{1\}.
$$
So $G'\cong H'\times\exp{\m'}$. 
  
Non degeneracy of $B|_\g$ follows easily: for any element  $X\in \g$ there exists $Y=Y_1+iY_2\in \g^\C$
such that $0\neq B(X,Y)=B(X,Y_1)+iB(X,Y_2)$. In particular $B(X,Y_1)\neq 0$. Clearly $\h'\perp_B\m'$, and
all the other properties of Definition \ref{K-reductive} are staightforward to check.
\end{proof}
\begin{rk}\label{rk invols ss}
Proposition \ref{maximal compact and theta enough} is well known for semisimple Lie groups (see for example Theorem 4.3.2
in \cite{GOV}).
\end{rk}
\begin{cor}\label{complex group as real form}
Let $G^\C$ be a connected complex reductive Lie group. Then, there exists a correspondence between
$G^\C$-conjugacy classes of real forms $G<G^\C$ and holomorphic involutions of $G^\C$ up to conjugation by $\Ad(G)$.
\end{cor}
\begin{proof}
 It follows from Proposition \ref{maximal compact and theta enough} by noticing that a choice of Cartan data is determined up to conjugation (except for the metric $B$, which plays no role, so we
 can ignore it), 
 and the indeterminacy
 in the choice of the antiholomorphic involution yielding a given real form too. To see the latter, assume $\sigma$ and $\sigma'$ are two different involutions
 of $G^\C$ with the same fixed point subgroup $G$. Then, since $\g^\C\cong\g\oplus i\g$, the differentials are the same $d\sigma=d\sigma'$. This means that $\sigma$ and $\sigma'$
 act the same way on the identity component $(G^\C)_0$, which is the group itself. 
\end{proof}
The following important fact is a consequence of Proposition \ref{maximal compact and theta enough}
\begin{prop}\label{prop lift}
Let $G'$ and $G^\C$ be as in Proposition \ref{maximal compact and theta enough}. We abuse notation by calling
$\sigma'$ and $\tau$  both the involutions defining $G'$ and $U$ and their differentials. Then
the composition $\sigma'\tau=\tau\sigma'$ defines a holomorphic involution of $G^\C$ which lifts the extension of $\theta$ to $\g^\C$
by complex linearity, and so we will abuse notation and denote 
$\theta:=\tau\sigma$ for the holomorphic involution of $G^\C$. Note that in particular, this holomorphic involution lifts $\theta$ to $G$.
\end{prop}
Proposition \ref{prop lift} is relevant at a conceptual level: it tells us that antilinear involution of a connected complex reductive Lie group can be chosen to
respect the Cartan data. This motivates the following definition, covering also the case of non compact groups.
\begin{defi}
 Let $(G^\C,U,\tau,B)$ be a complex reductive Lie group. We define a real form 
 $(G,H,\theta,B)<(G^\C,U,\tau,B)$ to be a real reductive subgroup such that $G<G^\C$ is a real form. This implies in particular
 that the involution $\sigma$ defining $G$ commutes with $\tau$.
\end{defi}
There are more reductive real subgroups of a complex reductive Lie group than real forms; some of these are related to real forms, as in the following example.
\begin{ex}\label{real forms and groups}
Consider $\SL(2,\R)<\SL(2,\C)$, which is a real form with associated involution $\sigma$ given by complex conjugation. But its normaliser inside 
$\SL(2,\C)$, say $N:=N_{\SL(2,\C)}(\SL(2,\R))$, is not. Reductivity of this group is shown in Corollary \ref{prop Gsigma reductive}. We just recall here some basic facts. 

The group $N$ is generated by $\SL(2,\R)$ and the element
$$
\left(\begin{array}{cc}
 0&i\\
 i&0
\end{array}
\right),
$$ 
so that it fits into an exact sequence
$$
1\to \SL(2,\R)\to N\to\Z/2\Z\to 1.
$$
The importance of these normalising subgroups will be made clear in Section
\ref{section Higgs pairs}.
\end{ex}
More generally, one may produce a real subgroup from a real form $G<G^\C$ defined by $\sigma$ as follows.
\begin{defi}\label{defi Giota}
 Given a complex or real Lie group $G$ and an involution $\iota: G\to G$ 
 (holomorphic or antiholomorphic), we 
 define
 $$
 G_\iota=\{g\in G\ :\ g^{-1}g^\iota\in Z(G)\}.
 $$
\end{defi}
\begin{rk}\label{rk Kiota vs NGK}
 Note that $G_\iota\subset N_G(G^\iota)$, as $Z(G)\subset Z_G(G^\iota)$.
\end{rk}
With the above definition, $(G^\C)_\sigma$ is a subgroup which is not necessarily a real form.
\begin{ex}
With the notation of Example \ref{real forms and groups}, for $G=\SL(2,\R)$, 
we have that $(G^\C)_\sigma=N$, which is not a real form.
\end{ex}
The above example generalises to all semisimple Lie groups.
\begin{lm}\label{cor centres}
 Let $G<G^\C$ be a real form of a complex semisimple Lie group defined by the involution $\sigma$. Then:
 
 1. $Z(G^\C){=}Z_{G^\C}(G)$.
 
 2.. $Z(G){=}Z(G^\C)^\sigma$
\end{lm}
\begin{proof}
By Corollary IV.4.22 in \cite{K}, $G^\C\plonge\Aut(V_\C)\subset \mbox{End}(V_\C)$ is a matrix group, so $G^\C$ is contained in the complex subspace spanned by $G$ inside of 
$\mbox{End}(V_\C)$. This implies that $Z_{G^\C}(G)\subset Z(G^\C)$. The other inclusion is trivial, which proves \textit{1.} 

As for \textit{2.}, by \textit{1.}, $Z(G^\C)^\sigma=Z_{G^\C}(G)^\sigma=Z_G(G)=Z(G)$.
\end{proof}
\begin{lm}\label{lm Gcs=NGcG ss}
If $G<G^\C$ is a real form of a complex semisimple Lie group, then $(G^\C)_\sigma=N_{G^\C}(G)$. 
\end{lm}
\begin{proof}
We easily see that
$
N_{G^ \C}(G)=\{g\in G^\C\ : \ g^{-1}g^\sigma\in Z_{G^\C}(G)\}$, as $g\in N_{G^ \C}(G)$ is equivalent to $ \sigma(gfg^{-1})= gfg^{-1}$ for all $f\in G$, which is in turn 
equivalent to $g^{-1}g^\sigma f(g^\sigma)^{-1}g=f$, i.e., $g^{-1}g^\sigma\in Z_{G^ \C}(G)$.

Now, by \textit{(1)} in Lemma \ref{cor centres} above $Z_{G^\C}(G)=Z(G^\C)$. Substituting this in the expression for $N$ we see the equality we wanted. 
\end{proof}
We next study the existence of a reductive structure of $G_\iota$, and apply it to the case $(G^\C)_\sigma$ which we then 
compare with $N_{G^\C}(G)$. 
\begin{prop}\label{prop Giota}
 Let $(G,H,\theta, B)$ be a reductive Lie group.
 
 1. Assume $G$ is connected, and let $\iota$ be an involution of $G$. Then, a conjugate $H':=\Ad(g)(H)$ of $H$ and
 its corresponding involution $\theta'$ provide Cartan data that induces Cartan data on $G_\iota$ by restriction and 
 intersection.
 
 2. When $G$ is not necessarily connected, if $\iota$ is an involution of $(G,H,\theta, B)$ (namely, $\iota$ leaves
 each component of the Cartan data invariant) then
 $G_\iota$ is $\theta$ stable and 
  $(G_\iota, (G_\iota)^\theta,\theta, B)$ is a reductive subgroup whose Lie algebra is 
  $\g_\iota=\g_+\oplus\z(\g)_-$ (where $\g=\g_+\oplus\g_-$ is the decomposition of $\g$ into the $\pm 1$
  $\iota$-eigenspaces, and likewise for $\z(\g)$). 
 
  3. Let $\Ad_G:G\to \Aut(G)$ be the adjoint representation, and define the action $\iota\laction\Aut(G)$
 by $\phi^\iota(g)=\iota(\phi(\iota(g)))$. Then, $G_\iota$ is the preimage by $\Ad_G$ of $\Ad_G(G)^\iota$.

 4. With the hypothesis of  2., consider $N=N_G(G^\iota)$. If $Z_G(\g^\iota)=Z_G(G^\iota)$, then 
  $(N, N^\theta,\theta, B)$ is a reductive subgroup whose Lie algebra is  also $\g_\iota$.
 
 5. If $Z_{G}(G^\iota)=Z(G)$, then $G_\iota=N$ .
 
 6. We have 
 $$\left(\Ad_\g(N),\Ad_\g(N)^\theta,\ad_\g(\theta),\ad_\g({B})\right)=\left(\Ad_\g(G_\iota),\Ad_\g(G_\iota)^\theta,
 \ad_\g(\theta),\ad_\g({B})\right),
 $$
 where
  $\Ad_\g: G\to\Aut(\g)$ is the adjoint representation.  
 \end{prop}
\begin{proof}
To prove \textit{1.}, we first need to prove a conjugate of $H$ is $\iota$-invariant. The proof is the same as in
Proposition \ref{maximal compact and theta enough} (with the difference that we conjugate the Cartan data rather
than $\iota$). Once this has been done, if we prove \textit{2.}, the remaining part of \textit{1.} follows.

For the proof of  \textit{2.}, note that the fact that $\iota$ be an involution of the whole reductive structure implies that each 
datum is left invariant by  $\iota$. In particular, the maximal compact subgroup of $G_\iota$ is 
$H\cap G_\iota=(G_\iota)^\theta$. Polar decomposition follows from Corollary 7.26 (2) in \cite{K}, just noticing
that its proof does not use (SR) in Definition \ref{K-reductive}. Indeed, according to this result $Z(G)=Z_H(G)e^{\z_\m(\g)}$, so that if $g=he^V$ 
is the polar decomposition of an element $g\in G_\iota$, then
$h^{-1}h^\iota\in Z_H(G)$, $V-\iota V\in\z_\m(\g)$, namely $h\in (G_\iota)^\theta$, $V\in \m\cap\g_\iota$.
Reductivity of $\g_\iota$ will follow once we prove its decomposition, as reductivity is closed by taking 
fixed points of involutions (Proposition \ref{prop red closed by inv}) 
and extensions by central abelian subalgebras.

Now, $X\in\g_\iota\iff X-\iota X\in \z(\g)$. Let $Y\in\z(\g)$ be such that $X=\iota X+Y$(*). Then $\iota X=X+\iota Y$,
which substituting yields $X=X+Y+\iota Y$. Namely, $Y\in\z(\g)\cap\g_-$. Let now $X=X_++X_-$, with $X_\pm\in\g_\pm$.
Then, substituting again in (*), we find $2X_-=Y\in\z(\g)\cap\g_-\iff X_-\in\z(\g)_-$. 

We have proved conditions $(1)$, $(2)$ and $(5)$ in Definition \ref{K-reductive}. The remaining ones follow directly
from the fact that $\sigma$ respects the Cartan involution induced by $\theta$.

As for \textit{3.}, we have that 
$$
\Ad_G(g)\in\Ad_G(G)^\iota\iff \Ad_G(g)=\Ad_G(g^\iota)\iff $$
$$g^{-1}g^\iota\in 
\Ker(\Ad_G)=Z(G)\iff g\in G_\iota.
$$
Point \textit{4.}, we easily check that $Lie(N)=:\lie{n}=\g_\iota$, so conditions (1), (3) and (4) in Definition \ref{K-reductive}
  follow from point \textit{2.} in this proposition. All that's left to check is polar decomposition, as it is clear that
  $N^\theta=N_H(G^\iota)$ is maximally compact. By Lemma 7.22 in \cite{K} applied to the
reductive group $G$ (plus the fact that the proof of the quoted result does not use (SR) in Definition \ref{K-reductive}), since both $N$ and $N^\theta$ normalize the $\theta$-invariant Lie algebra $\g^\iota$, if follows that 
$N_{G}(\g^\iota)=N_U(\g^\iota)\times e^{i\lie{n}_\h(\g^\iota)}$. Now, $n\in N_{G}(\g^\iota)\iff n^{-1}n^\iota\in Z_{G}(\g^\iota)$. 
Likewise, $n\in N\iff n^{-1}n^\iota\in Z_{G}(G^\iota)$. Hence, we have \textit{4.}

Finally, \textit{5.} and \textit{6.} are easy to check from the definitions. In \textit{6.} note that $\Ad_\g(N)$ is always reductive, as 
$\Ad_\g(Z(G))=
\Ad_\g(Z_{G}(\g^\iota))=1$. 
\end{proof}
Now, when  $\iota$ defines a real form of a complex Lie group, Proposition \ref{prop Giota} can be completed as follows: 
\begin{cor}\label{prop Gsigma reductive}
 Let $(G,H,\theta, B)<(G^\C,U,\tau, B_\C)$ be a real form defined by $\sigma$. Then
 \\\\
 1. The tuple $(U_\sigma,\theta, B)$ defines a reductive structure on $(G^\C)_\sigma$.
 \\\\
2. We have $(G^\C)_\sigma=N$ when $Z_{G^\C}(G)=Z(G^\C)$. This is the case, for example, of semisimple groups.
 \\\\
 3. The Lie algebra $\g_\sigma\subset\g^\C$ is a real form of $\g^\C\oplus\z(\g^\C)$.
 \end{cor}
\begin{proof}
Point  \textit{1.} follows from 
the equality $Z_U(G^\C)=Z(U)$, proved just as Lemma \ref{centre ss}.

The first  statement in \textit{2.} follows as in Proposition \ref{prop Giota}, while the second is a consequence of \textit{1.} in Lemma \ref{cor centres}.

Point \textit{3.}, is an easy remark, as from \textit{2.} in Proposition \ref{prop Giota}, we have $\g_\sigma=\g\oplus i\z(\g)$.
\end{proof}
Note that strong reductivity need not be preserved.
\begin{ex}\label{ex SOth}
 We see easily that $N_{\SL(2,\C)}(\SO(2,\C))=\SL(2,\C)_\theta$ which is the extension
 $$
 0\to \SO(2,\C)\to N\to \Z/2\Z\to 0
 $$
 generated by the element $\left(\begin{array}{cc}
                                  0&i\\i&0
                                 \end{array}
\right)$.
\end{ex}
The following proposition points at an important relation between the groups $(G^\C)_\sigma$ and $(G^\C)_\theta$.
\begin{lm}\label{lm Gcs=Gct for ss}
 Let $G<G^\C$ be a real form of a semisimple Lie group whose defining involution we denote by $\sigma$. Then, if
 $\theta$ denotes the holomorphic involution corresponding to $\sigma$ after a choice of a compatible maximal compact subgroup (see Remark \ref{rk invols ss}), we have
 $$
 (G^\C)_\sigma/G=(G^\C)_\theta/H^\C.
 $$
\end{lm}
\begin{proof}
We note that the above groups fit into exact sequences:
$$
0\to G\plonge (G^\C)_\sigma\stackrel{f_1}{\to} Z(G^\C), \qquad f_1(g)=g^{-1}g^\sigma, 
$$
$$
0\to H^\C\plonge (G^\C)_\theta\stackrel{f_2}{\to}Z(G^\C), \qquad f_2(g)=g^{-1}g^\theta.
$$
Thus we just need to prove that $g^{-1}g^\sigma\in Z(G^\C)\iff g^{-1}g^\theta\in Z(G^\C)$.
 By Lemma \ref{centre ss}, $Z(G^\C)=Z(U)$. So let $g=ue^V$ be the polar decomposition of some element of $G^\C$. Then,
 $$
 g^{-1}g^\sigma\in Z(U)\iff u^{-1}u^\sigma=u^{-1}u^\theta\in Z(U)\iff g^{-1}g^\theta\in Z(U).
 $$
\end{proof}
Our interest in groups such as $(G^\C)_\sigma$ is twofold. On the one hand, they produce examples of 
real Lie groups which are not real forms. On the other hand, we will see in Section \ref{section kostant-rallis section},
that the group $\Ad(G^\C)_\theta=\Ad(G^\C)^\theta$ is relevant in the study of the $H^\C$-module $\m^\C$. 
Lemma \ref{lm Gcs=Gct for ss} and Corollary \ref{prop Gsigma reductive}, 
tells  $\Ad(G^\C)_\theta$ determines the real form $\Ad(G^\C)^\sigma=\Ad((G^\C)_\sigma)=\Ad(G^\C)_\sigma$ and viceversa.
\begin{prop}\label{lemma Gtheta}
Let $(G,H,\theta,B)<(G^\C,U,\tau,B)$ be a real form of a complex strongly reductive Lie group. Let 
\begin{equation}\label{eq A}
A=e^\la,
 \end{equation} 
and consider 
\begin{equation}\label{eq F} 
F=\{a\in A:a^2\in Z(G)\}. 
\end{equation}
Then\\
1. We have that $G_\theta=F\cdot H$ and $\Ad(G_\theta)=\Ad(G)_\theta=Q\cdot\Ad(H)$, where $Q=\{a\in \Ad(\la)\ :\ a^2=1\}$.
\\\\
2. There are equalities 
$$
\Ad(G^\C)^\theta=\Ad(G^\C)_\theta=\Ad ((G^\C)_\theta)=Q\cdot\Ad(H^\C)=\Ad(G_\theta)^\C.
$$
3. Let $\Ad_G:G\to\Aut(G)$ be the adjoint representation. Then $G^\theta$ is the preimage of $\Ad(G)^\theta$.
\end{prop}
\begin{proof}
To prove \textit{1.}, consider the decomposition $G=HAH$ (see \cite[VII.3]{K}, noting 
that the arguments leading to Theorem 7.39 do not require Condition (SR) in
Definition \ref{K-reductive}). Now, choose $g\in G_\theta$. By the above, it can be expressed as 
$g=h_1ah_2$ where $h_1$, $h_2\in H$, $a\in A$. Thus, $g(g^\theta)^{-1}=k_1a^2k_1^{-1}\in Z(G)$ 
 if and only if so is $a^2$, whence the result.
 
 As for \textit{2.}, the first equality is a remark, whilst the second follows from 
$$
\Ad(g)^\theta=\Ad(g^\theta)=\Ad(g)\iff g^{-1}g^\theta\in Z(G^\C).
$$
For the third equality, the same proof as in Proposition 1 in \cite{KR71} can be used (note that the proposition itself can only be directly applied if $\Ad(G^\C)$ 
is connected), yielding
$$
\Ad(G^\C)_\theta=Q\cdot \Ad(H^\C),
$$
where $Q=\exp(i\ad(\la))_{[2]}$. But then, $Q\subset \Ad(G)$, as $\sigma(g)=g^{-1}=g$ as $Q$ is two torsion (see Proposition 2 in \cite{KR71}). 
 The proof of \textit{3.} follows from \textit{3.} in Proposition 
\ref{prop Giota}.

Finally, the last 
equality, follows from \textit{1.}, as  $\Ad(F)=Q$. 
\end{proof}
\begin{rk}\label{rk adjoint case}
 When $G^\C$ is the adjoint group of a complex reductive Lie algebra, we obtain that $(G^\C)_\theta=(G^\C)^\theta$, as the centre is trivial. 
 This case is the one considered by Kostant and Rallis, who distinguish between two groups: $K_\theta$, in
 our notation, $(G^\C)_\theta$, and $K$, the identity component of $K_\theta$, in our notation, $(H^\C)^0$.
 This distinction is important for the orbit structure of $\m^\C$ under the action of $(H^\C)^0$ 
 (see \cite{KR71} Theorem 11.) In the real case, if the centre of $G$ is trivial, then $F\subset H$, as in this situation, $a\in F$ if and only if $a^2=1$, so 
 $a^{-1}=a^\theta=a$. Hence $G_\theta=H$.
\end{rk} 

\subsection{Maximal split subgroup}\label{max split subgp}
Just as there is a maximal split subalgebra of a real reductive Lie algebra, we can define the 
maximal connected  split subgroup of a reductive Lie group $(G,H,\theta, B)$.
We introduce the following notions.
\begin{defi}\label{def qsplit gral gp}
 We say that a real reductive Lie group $(G,H,\theta, B)$ is split, quasi-split, etc. if $\g\subset\g^\C$ is split, quasi-split, etc., respectively.
\end{defi}
\begin{defi}\label{defi maximal connected split}
 Let $G$ be a Lie group whose Lie algebra is reductive. The \textbf{maximal connected  split subgroup} is defined
 to be the  analytic subgroup  $\TIG_0\leq G$ with Lie algebra $\tg$.
\end{defi}
Consider the tuple $(\TIG_0,\TIH_0,\tth,\widehat{B})$ where  $\TIH_0:=\exp(\tih)\leq H$, and $\tth$ and $\widehat{B}$
are obtained by restriction. 
\begin{prop}\label{maximal split reductive}
If $(G,H,\theta, B)$ is a reductive Lie group, then, the tuple $(\TIG_0,\TIH_0,\tth,\widehat{B})$ is a strongly reductive Lie group.
\end{prop}
\begin{proof}
By Proposition \ref{redsub} \ref{tg reductive}., conditions (1), (3) and (4) in Definition \ref{K-reductive} hold. Since $\TIG_0$ is connected, we may assume $G$ is connected,
as $\TIG_0\subset G_0$. In this case, writing the polar decomposition of $g\in\TIG_0$, we have, by connectedness of $H$, $g=e^Xe^Y$, for some $X\in\h$, $Y\in \m$.
By construction, $\tg^\C$ is self normalising within $\g^\C$ (as it is the subalgebra generated by a principal normal TDS, $\la^\C$,  and the centre of $\g^\C$),
and the same holds for $\tg$. This implies that, modulo the kernel of the exponential, $X$ and $Y$ can be chosen in $\tih$ and $\tm$. So we may work at the level of the 
universal cover ${G}^u$ of $G$, to which it corresponds a maximal split subgroup $\widehat{G}^u_0$, and then induce the result for $\TIG_0$.

This gives polar decomposition, and maximality of $\TIH_0$ follows from Proposition 7.19 in \cite{K}, just noticing
that its proof does not use (SR) in Definition \ref{K-reductive}, and Remark \ref{rk defi k vs ours}.
 Strong reductivity follows from connectedness, as condition (\ref{diffeo}) in Definition \ref{K-reductive} implies $G=e^\h\cdot e^\m$, 
since $H$ being compact and connected it must be
$H=e^\h$. A simple computation shows that in the case of matrix groups 
$\Ad_{e^X}\circ\Ad_{e^{Y}}\equiv \Ad_{e^{X+Y}}\in\Aut\ \g$. Since $\Ad(G)$ is semisimple, it is a matrix group and 
furthermore $\Ad\left(\Ad(G)\right)\cong \Ad(G)$,
so Condition (SR) in Definition \ref{K-reductive} follows for connected groups.
\end{proof}
If $(G,H,\theta,B)<(G^\C,U,\tau,B)$ 
is a real form of a complex reductive Lie group, there is an alternative natural candidate to a maximal split subgroup.
Note that even in the situation when $G$ has a complexification, $\TIG_0$ need not be a real form of a complex Lie 
group. It is so just up to a finite extension.
\begin{lm}\label{lm max split}
 Let $(G,H,\theta,B)<(G^\C,U,\tau,B)$ be a real form of a complex reductive Lie group, and let $\sigma$ be the corresponding antiholomorphic involution.
 Define $\TIG^\C<G^\C$ to be the analytic subgroup corresponding to $\tg^\C$, where $\tg$ is defined as in Definition
 \ref{defi maximal split}. Then:
 \benum
 \item[1.] The involution $\sigma$ leaves $\TIG^\C$ invariant.
 \item[2.] Let $\TIG=\left(\TIG^\C\right)^\sigma$, and let $\TIH\leq\TIG$ be the maximal compact subgroup. 
 Then $(\TIG,\TIH,\tth,\widehat{B})$, where $\tth$ and $\widehat{B}$ are as in Proposition \ref{maximal split reductive},
 is a reductive Lie group and a real form of $(\TIG^\C,U\cap \TIG^\C,\tau|_{\tg^\C},B|_{\tg^\C})$. 
 \eenum
\end{lm}
\begin{proof}
We first note that $\TIG^\C=(\TIG_0)^\C$, as both are connected complex Lie subgroups of $G^\C$ with
the same Lie algebra. Then. the first statement follows from the following fact: by definition $\sigma$ leaves $G$ pointwise invariant, and so 
does it leave $\TIG_0$. Thus, the complexification $(\TIG_0)^\C=\TIG^\C$ is
$\sigma$-invariant. Indeed,
$\TIG_0\subset\TIG^\C\cap\sigma\left(\TIG^\C\right)$; the intersection is a complex group, so that the 
complexification of  $\TIG_0$ is also contained in the intersection, namely, it is all of the intersection. 

The second assertion follows from  Proposition \ref{maximal split reductive}.
\end{proof}
\begin{defi}\label{defi max split}
Let $(G,H,\theta,B)<(G^\C,U,\tau,B_\C)$ be a real form of a complex reductive Lie group.
Let $(\TIG,\TIH,\tth,\widehat{B})$ be as in Lemma \ref{lm max split}. We call this group the \textbf{maximal split subgroup} of
$(G,H,\theta,B)$.
\end{defi}
Given a reductive Lie group, we would like to determine its maximal connected split subgroup. This is studied 
in work by Borel and Tits \cite{BTComplements} in the case of real forms of complex semisimple algebraic groups. 
It is important to
note that over $\R$, the category of semisimple algebraic groups differs
from the category of semisimple Lie groups. For example, the semisimple algebraic
group $\Sp(2n,\R)$ has a finite cover of any given degree, all of which are semisimple
Lie groups, but none of them is a matrix group. So although their results do not apply to real Lie groups
in general, they do apply to real forms of complex semisimple Lie groups.

In former work \cite{BTGroupes}, the authors build, in the context of reductive algebraic groups (which they consider
functorially), a maximal connected split subgroup, unique up to the choice of a maximal split subtorus $\mc{A}$
and a choice of one unipotent generator of an $\mc{A}$-invariant three dimensional subgroup corresponding to each root $\alpha\in\Delta$ such that
$2\alpha\notin\Delta$. 

Let $\mc{G}$ be a reductive 
algebraic group, and let $\widehat{\mc{G}}_0$ be the maximal connected split subgroup. In case $\widehat{\mc{G}}$
has a complexification $\widehat{\mc{G}}^\C$, it is well known that the map that to a group assigns its complex points
$$
\widehat{\mc{G}}^\C\mapsto\widehat{\mc{G}}^\C(\C)
$$
establishes an equivalence of categories between the categories $\mc{AG}$ of complex semisimple algebraic groups  and
$\mc{LG}$ of (holomorphic) complex semisimple Lie groups (also reductive, but on the holomorphic side we get a subcategory).
This yields:
\begin{prop}\label{correspondence algebraic Lie}
 Let $G^\C$ be a complex semisimple Lie group, and let $\mc{G}^\C$ be the corresponding algebraic group, 
 so that ${G}^\C=\mc{G}^\C(\C)$. Let $G<G^\C$ be a real form. Then,
 there exists a real linear algebraic group $\mc{G}$ such that $\mc{G}(\R)=G$ and
 moreover $\widehat{\mc{G}}_0(\R)=\TIG_0$.
\end{prop}
\begin{proof}
The equivalence between $\mc{AG}$ and $\mc{LG}$ implies that the holomorphic involution $\theta\laction{G}^\C$
corresponding to $G$ via Corollary \ref{complex group as real form} is 
algebraic. Thus, both $\tau$ and $\sigma$ are real algebraic, that it, defined by polynomial equations over the 
real numbers. This implies they induce involutions (that we denote by the same letters) of $\mc{G}^\C$.
Let $\mc{G}=\left(\mc{G}^\C\right)^\sigma$. Then, $\mc{G}(\R)=\left(\mc{G}^\C(\C)\right)^\sigma=G$. 
By construction of $\widehat{G}_0$, the choices required for the uniqueness of  Borel-Tits' 
maximal connected 
split subgroup are met. So there is a unique algebraic group $\widehat{\mc{G}}_0$ such that 
$\widehat{\mc{G}}_0(\R )=\TIG_0$.
\end{proof}
The following lemma gives a necessay condition for a subgroup to be the maximal connected split subgroup.
\begin{lm}\label{coro borel tits}
 Let $\mc{G}$ be a real semisimple algebraic group, $\widehat{\mc{G}}$ a semisimple subgroup such that there
 exist maximal tori $\mc{T},  \widehat{\mc{T}}$ of $\mc{G}$ and
 $\widehat{\mc{G}}$, respectively, with 
 $\widehat{\mc{T}}\subseteq \mc{T}$. Let $\Delta$ be a root system of $\mc{G}$ with respect to $\mc{T}$, and let
 $\widehat{\Delta}$ be the (non-zero) restriction of elements of $\Delta$ to $\widehat{\mc{T}}$. Assume $\widehat{\Delta}$
 is a root system.  If $\mc{G}$ is
simply connected or $\widehat{\Delta}$ is a non reduced root system, then $\widehat{\mc{G}}$ is simply connected.
\end{lm}
\begin{rk}\label{algebraic simply connected}
In the above corollary, simple connectedness is meant in the algebraic sense: namely, the lattice 
of inverse roots is maximal within the lattice of weights of the group. Note that the algebraic 
fundamental group for compact linear algebraic groups and
the topological fundamental group  of their
corresponding groups of matrices of complex points are the same (see \cite{BD} for details). 
The polar decomposition implies the same for the class of reductive Lie groups. However, algebraic simple connectedness does not
mean that the fundamental group be trivial. 
\end{rk}
Lemma \ref{coro borel tits} has the following consequence:
\begin{cor}\label{maximal split as finite cover}
 Let $G^\C$ be  a complex semisimple Lie group, and let $G< G^\C$ be a real form that is either
 simply connected or of type $BC$.
 Then the analytic subgroup $\TIG_0^\C\leq G^\C$ is (topologically) simply
connected.
\end{cor}
\begin{proof}
By Proposition \ref{correspondence algebraic Lie}, we have algebraic groups $\mc{G}^\C$, 
$\widehat{\mc{G}}^\C$  and real forms $\mc{G}$, $\widehat{\mc{G}}$ to which the results of 
Borel and Tits may be applied. In particular $\widehat{\mc{G}}$ is simply connected. Assume 
$\widehat{\mc{G}}^\C$ was not. Then, it would have a finite cover $(\widehat{\mc{G}}^\C)'$, which in turn 
would contain a 
real form $(\widehat{\mc{G}})'$ (defined by a lifting $\sigma$) that would be a finite cover of $\widehat{\mc{G}}$ and an algebraic group.
\end{proof}
\begin{ex}\label{maximal split subgroup supq}
 Take the real form $\SU(p,q)< \SL(p+q,\C)$. Its fundamental group is 
 $$
 \pi_1(\mathrm{S}(\U(p)\times \U(q)))=\Z.
 $$
 We know from \cite{Araki} that the maximal split 
Lie subalgebra of $\lie{su}(p,q)$ $p>q$ is $\lie{so}(q+1,q)$, whereas the maximal split subalgebra
of
 $\lie{su}(p,p)$ is $\lie{sp}(2p,\R)$. In what follows, we analyze what the maximal split subgroup is in the various cases:
 
$\bullet$ $p>q$. Since the root system is non-reduced (see \cite[VI.4]{K}), Lemma \ref{coro borel tits} and 
Corollary \ref{maximal split as finite cover} imply
that the maximal split subgroup is the algebraic universal cover of $\SO(q+1,q)_0$. We have the following
table of fundamental groups of the connected component of $\SO(p+1,p)$:

\begingroup
\leftskip=0cm plus 0.5fil \rightskip=0cm plus -0.5fil
\parfillskip=0cm plus 1fil
\begin{tabular}{|c|c|c|}\hline
&$q=1$&$\Z$\\\cline{2-3}
$\pi_1(\SO(q+1,q)_0)$&$q=2$& $\Z\times\Z/2\Z$\\\cline{2-3}
&$q\geq 3$&$\Z/2\Z\times\Z/2\Z$\\\hline
\end{tabular}\par
\endgroup

For $q=1$, we have the exact sequence
$$
1\to\Z/2\Z\to\Sp(2,\R)\to\SO(2,1)_0\to 1.
$$
Since $\Sp(2,\R)$ is simply connected (for example, since no finite covering of it is a matrix group). In particular $\widehat{\SU(p,1)}=\Spin(2,1)_0\cong\Sp(2,\R)$.

When $q=2$, the maximal split subgroup is again the algebraic universal cover of $\SO(3,2)_0$, which is a two
cover considering the fundamental group. It is well known that
  $\lie{so}(2,3)\cong\lie{sp}(4,\R)$, and $\Sp(4,\R)\cong \Spin(3,2)_0$ is connected, hence $\widehat{\SU(p,2)}=\Spin(3,1)_0$.
  
As for $q\geq 3$, the universal covering group
of $\SO(q,q+1)_0$ is the connected component of $\Spin(q,q+1)$. This group is a $4$ cover 
of $\SO(q,q+1)_0$, which is thus simply connected. 

$\bullet$ $p=q$. Since $\Sp(2n,\R)\subseteq \SU(n,n)$, the candidate to the maximal split subgroup is a finite cover of $\Sp(2n,\R)$
embedding into $\Sp(2n,\C)$ (which is simply connected). Thus $\widehat{\SU(n,n)}=\Sp(2n,\R)$.
\end{ex}
The group $\SU(p,q)$ is a group of Hermitian type, a class of groups which will become relevant in 
Section \ref{section Higgs pairs}. 
\begin{defi}\label{def htype} A reductive group $(G,H,\theta,B)$ is said to be of Hermitian type if the symmetric 
space associated to it admits a complex structure which is invariant by the group of isometries. 
If the group $G$ is simple, this is equivalent to $H$ having non-discrete centre.

The Lie algebras of simple such groups are $\sp(2n,\R)$, $\su(p,q)$, $\so^*(2,n)$, $\so(2,n)$, $\lie{e}_{6(-14)}$ and $\lie{e}_{7(-25)}$.
\end{defi}
\section{The Kostant--Rallis section}\label{section kostant-rallis section}
Let $(G, {H}, \theta, B)$ be a reductive Lie group, and consider the decomposition 
$
\g={\h}\oplus \m
$
induced
by $\theta$. Let $\la\subseteq\m$ be a maximal anisotropic Cartan subalgebra, and let  ${H^\C}$, ${\g^\C}$, etc. 
denote the
complexifications of the respective groups, algebras, etc. Note that we do not assume that  $G^\C$ exists. 
In \cite{KR71}, Kostant and Rallis study the orbit structure of the $H^\C$ module $\m^\C$ in the case when
$G^\C$ is the adjoint group of a complex reductive Lie algebra $\g^\C$ (namely, $G^\C=\Inn(\g^\C)=\Aut(\g^\C)^0$
In this section, we study a generalization of their result to reductive Lie groups in the sense of Definition 
\ref{K-reductive}.

The  first result we will be concerned about is the Chevalley restriction theorem, which is well known for Lie groups of adjoint type. Recall that given a complex reductive Lie algebra
$\g^\C$, its \textbf{adjoint group}, denoted by $\Ad(\g^\C)$, is the connected component of its automorphism group $\Aut(\g^\C)$. It coincides with
the connected component of the image of the adjoint representation of any Lie group $G^\C$ such that 
$\mathrm{Lie}(G^\C)=\g^\C$. 
We need the following.
\begin{defi}
 We define the \textbf{restricted Weyl group} of  $\g$ (resp. ${\g^\C}$) associated to $\la$ (resp. ${\la^\C}$),
 $W({\la})$ (resp. $W({\la^\C})$), to be
 the group of automorphisms of ${\la}$ (resp. ${\la^\C}$) generated by reflections on the hyperplanes defined
 by the restricted roots $\lambda\in\Lambda(\la)$ (resp. $\Lambda(\la^\C)$).
\end{defi}
The Chevalley restriction theorem asserts that, given a group $G$ of adjoint type, the restriction 
$\C[{\m^\C}]\to\C[{\la^\C}]$
induces an isomorphism
$$
\C[{\m^\C}]^{{H^\C}}\stackrel{\cong}{\to}\C[{\la^\C}]^{W({\la^\C})}.
$$ 
See for example \cite{Hel}. 

The restricted Weyl group admits other useful characterizations in the case of 
strongly reductive Lie groups.
\begin{lm}
Let $(G,H,\theta, B)$ be a strongly reductive Lie group. We have

1. $W(\la)=N_{{H}}(\la)/C_{{H}}(\la)$,
where 
 $$
 N_{{H}}(\la)=\{h\in {H}\ :\ \Ad_h(x)\in\la\ \mathrm{ for all }\ x\in\la\},
 $$
 $$
 C_{{H}}(\la)=\{h\in {H}\ :\ \Ad_h(x)=x\ \mathrm{ for all }\ x\in\la\}.
 $$

2.  $W(\la^\C)=N_{{H^\C}}(\la^\C)/C_{{H}}(\la^\C)$,
where $N_{{H^\C}}(\la^\C)$ and $C_{{H}^\C}(\la^\C)$ are defined as above.  

3. Moreover, $W(\la^\C)=W(\la)$ as automorphism groups of $\la^\C$, where the action of $W(\la)$ on $\la^\C$ is defined
by extension by complex linearity.
\end{lm}
\begin{proof}
The first statement follows from Proposition 7.24 in \cite{K}. 

As for \textit{3.}, it follows by definition of restricted roots. 

To prove \textit{2.}, it is therefore enough to prove that $W(\la)=N_{{H^\C}}(\la^\C)/C_{{H^\C}}(\la^\C)$ when acting on
$\la^\C$. Now, if $(G,H,\theta,B)$ is strongly reductive, then $(H^\C,H,\tau,B_\h)$ is also strongly reductive for $\tau$ the
involution defining $\h$ inside its complexification and a suitable
choice of $B_\h$. Hence, by Lemma 7.22 in \cite{K}, if $h=xe^Y$ is the polar decomposition
 of an element in $N_{{H^\C}}(\la^\C)$, we have, by $\tau$-invariance of $\la^\C$, that both $x$ and 
 $Y$ normalise $\la^\C$. This means that $x\in N_H(\la^\C)= N_H(\la)$, and $Y\in\lie{n}_{\h^\C}(\la^\C)$. Now, by 
 Lemma 6.56 in \cite{K}, $\lie{n}_{\h^\C}(\la^\C)=\lie{c}_{\h^\C}(\la^\C)$, so the statement is proved.
\end{proof}

We have the following:
\begin{prop}\label{restriction iso}
Let $(G,{H},\theta,B)$ be a strongly reductive Lie group and 
$(\widehat{G}_0,\widehat{H}_0,\tth,\widehat{B})$
be the maximal connected split subgroup. Then, restriction induces an isomorphism
 $$
 \C[{\m^\C}]^{{H^\C}}\cong\C[{\la^\C}]^{W({\la^\C})}\cong\C[\widehat{{\m}}^\C]^{(\TIH_0)^\C}.
 $$
 
 If moreover $(G,{H},\theta,B)<(G^\C,U,\tau,B_\C)$ is a real form, from Definition \ref{defi max split} one has
 the maximal split subgroup $(\TIG,\TIH,\tth,\widehat{B})<(G,{H},\theta,B)$, and 
 $$
 \C[{\m^\C}]^{{H^\C}}\cong\C[\widehat{{\m}}^\C]^{\TIH^\C}.
 $$
\end{prop}
\begin{proof} 
By Lemma 7.24 in \cite{K},
$$
\Ad(H)\subseteq \Inn ({\h}\oplus i\m).
$$
Then, given that ${H^\C}={H}e^{i{\h}}$, $H^\C$  clearly acts on $\g^\C$ by inner
automorphisms of ${\g^\C}$. So $\Ad({\h^\C})=\Ad(H^\C)\subseteq(\Ad\ {\g^\C})^{\theta}$,
which implies
\begin{equation}\label{eq invars}
 \C[{\m^\C}]^{\Ad\ {\h^\C}}=\C[{\m^\C}]^{\Ad\ {H^\C}}\supseteq \C[{\m^\C}]^{(\Ad\ {\g^\C})^\theta}.
\end{equation}
Now, Proposition 10 in \cite{KR71} implies that 
$$
\C[{\m^\C}]^{(\Ad\ {\g^\C})^\theta}=\C[{\m^\C}]^{\Ad\ {\h^\C}}
$$
and so we obtain equalities in Equation (\ref{eq invars}) above.

Since $W({\la^\C})=N_{\Ad\ \h}(\la)/C_{\Ad\ \h}(\la)$, the isomorphism $\C[{\m^\C}]^{{H^\C}}\cong\C[{\la^\C}]^{W({\la^\C})}$ follows from the
adjoint group case and (\ref{eq invars}).

As for the split subgroup,  by the adjoint case and Proposition \ref{maximal split reductive}, 
we have $\C[{\la^\C}]^{W({\la^\C})}\cong\C[\widehat{{\m}}^\C]^{\widehat{{H}}_0^\C}$.
Also, by the definition of $\widehat{H}$ (see Definition \ref{defi max split}), $\Ad(\TIH_0)\subset\Ad(\TIH)\subset\Ad(\TIG_0)_\theta$, which by Proposition \ref{lemma Gtheta} and Proposition 10 in \cite{KR71}
implies 
that $\C[{\la^\C}]^{W({\la^\C})}\cong\C[\widehat{{\m}}^\C]^{\widehat{{H}}^\C}$.
\end{proof}
\begin{prop}\label{prop exponents}
Let $a=\dim\la^\C$. Then:
\benum
\item[1.] $\C[\la^\C]^{\bweyl}$ is generated by homogeneous polynomials of degrees $m_1,\dots, m_a$, canonically determined
by $(G,\theta)$.  
\item[2.] If $(\widehat{G}_0,\widehat{H}_0,\tth,\widehat{B})<(G, H,\theta,B)$ is the maximal connected split subgroup, 
the exponents are the same for both groups.
 \eenum
\end{prop}
\begin{proof}
Statement \textit{1.} is well known and follows from Proposition \ref{redsub} \ref{tg split}.

\textit{2.} follows by Proposition \ref{restriction iso} and the fact that the exponents of the group relate are $m_k-1$, where $m_k$ are the degrees of the generators of $\C[\m^\C]^{H^\C}$.
 \end{proof}

\begin{rk}\label{rk invariant polynomials} 
 Note that the ring of invariant polynomials depends on the choice of involution $\theta$ for non semisimple groups,
 as the number of degree one generators of $\C[\m^\C]^{H^\C}$ is the dimension $\dim_\R\z(\g)\cap\m$, which
 depends on $\theta$ if $G$ is not semisimple. So two instances of Cartan data
 on the same Lie group will yield different rings of invariants.
\end{rk}
We thus have an algebraic morphism
\begin{equation}\label{Chevalley map}
\chi: {\m^\C}\projects {\m^\C}\sslash{H^\C}\cong {\la^\C}/W({\la^\C})
\end{equation}
where the double quotient sign $\sslash$ stands for the affine GIT quotient.

We build next a section of the above surjective map. This is done by Kostant and Rallis in the
case 
$G^\C=\Ad(\g^\C)$ for a complex reductive Lie algebra $\g^\C$.
Let us start with some preliminary definitions.
\begin{defi}\label{def regular}
An element $x\in{\m^\C}$ is said to be \textbf{regular} is $\dim \lc_{\m^\C}(x)=\dim{\la^\C}$. Here
\begin{equation}\label{eq cm}
\lc_{\m^\C}(x)=\{y\in\m^\C\ :\ [y,x]=0\}. 
\end{equation}
Denote the subset of regular elements of ${\m^\C}$ by $\mr.$
\end{defi}
Regular elements are those whose ${H^\C}$-orbits are maximal dimensional, so this notion generalises the classical
notion of regularity of an element of a complex reductive Lie algebra.
\begin{rk}\label{rk regular m}
Note that the intersection $\m^\C\cap\g^\C_{reg}$ is either empty or the whole of $\m_{reg}$. Here 
$\g_{reg}$ denotes the elements of $\g^\C$ with maximal dimensional $G^\C$-orbit. 
\end{rk}
The following definition follows naturally from the preceeding remark.
\begin{defi}\label{defi qs}
A real form $\g\subset\g^\C$ is
\textbf{quasi-split} if $\m^\C\cap\g^\C_{reg}$.
  These include
split real forms, and the Lie algebras $\su(p,p)$, $\su(p,p+1)$, $\so(p,p+2)$, and $\lie{e}_{6(2)}$. 
Quasi-split real forms admit several equivalent characterizations: $\g$ is quasi-split 
if and only if $\lc_\g(\la)$ is abelian, which holds if and only if $\g^\C$ contains a $\theta$-invariant Borel subalgebra and if and only if $\m^\C\cap\g^\C_{reg}=\m_{reg}$ .
\end{defi}

\begin{thm}\label{theorem KR section}
Let $(G,{H},\theta,B)$ be a strongly reductive Lie group.  
Let  $\lie{s}^\C\subseteq {\g^\C}$ be a principal normal $TDS$ with normal basis $\{x,e,f\}$ (see Definition \ref{defi TDS}). 
Then
\benum
\item[1.] The affine subspace $f+\lie{c}_{\m^\C}(e)$ is isomorphic to ${\la^\C}/W({\la^\C})$ as an affine 
variety.
\item[2.] $f+\lie{c}_{\m^\C}(e)$ is contained in the open subset $\mr$, where $\lie{c}_{\m^\C}(e)$ is defined as in (\ref{eq cm}).
\item[3.] $f+\lie{c}_{\m^\C}(e)$ intersects each $(\Ad(G_\theta))^\C$-orbit at exactly one point. Here $G_\theta$
is given  in Definition \ref{defi Giota}.
\item[4.] $f+\lie{c}_{\m^\C}(e)$ is a section of the Chevalley morphism (\ref{Chevalley map}).
\item[5.] Let $(\widehat{G}_0,\widehat{H}_0,\tth,\widehat{B})<(G,{H},\theta,B)$ be the maximal connected split subgroup. 
Then, $\lie{s}^\C$ can be chosen so that $f+\lc_{\m^\C}(e)\subseteq \tm^\C$. If moreover $G$ is a real form of $G^\C$, say,   
then $f+\lc_{\m^\C}(e)$ is the image
of  Kostant's section for $\TIG^\C$ \cite{Kos}. Here, $\tm^\C$ is defined as in Proposition \ref{redsub} and $\TIG^\C$ as in Lemma \ref{lm max split}.
\eenum  
\end{thm}
\begin{proof}
We follow the proof due to Kostant and Rallis (see Theorems 11, 12 and 13 in \cite{KR71}) adapting their 
arguments to our setting when necessary. 

First note that Proposition \ref{restriction iso} implies the existence of a surjective map
$$
{\m^\C}\to{\la^\C}/W({\la^\C}).
$$
As in \cite{KR71}, consider the element
\begin{equation}\label{eq ec}
e_c=i\sum_jd_jy_j\in i\g, 
\end{equation}
where  $y_i\in\g_{\lambda_i}$ are as in Section \ref{section max split} and
\begin{equation}\label{eq dj}
d_j=\sqrt{\frac{-c_j}{b_j}}. 
\end{equation}
Here the elements $c_j$ are defined so that 
\begin{equation}\label{eq w}
w=\sum_ic_ih_i\in\la 
\end{equation}
is the only element in $\la$ such that $\lambda(w)=2$ for any $\lambda\in\Lambda(\la)$, and 
$h_i$ is the dual of $\lambda_i$ via the bilinear form $B$.  Note that in order for $e_c$ to belong to $i\g$, 
we must prove that $c_i/b_i<0$. Now, 
following the proof of Proposition 18 in \cite{KR71}, for any $y\in{\g}$, we have $2B(y,\theta y)=B(y+\theta y,y+\theta y)<0$
since $y+\theta y\in {\h}$. Hence, if $b_i=B(y_i,\theta y_i)$ it must be a negative real number. Also the fact that
$c_i>0$  follows from general considerations on the representations of three dimensional subalgebras (see 
 Lemma 15 in \cite{KR71}) and so does not depend on the choice of pairing $B$.

Once we have that, taking 
$$
f_c=\theta e_c,
$$
it follows by the same arguments found in \cite{KR71} that $\{e_c,f_c,w\}$ generate a principal normal TDS 
$\lie{s}^\C$ stable by $\sigma$ and $\theta$ (Proposition 22 in \cite{KR71}). In particular, $\lie{s}^\C$ has a
normal basis, say $\{e,f,x\}$.
By construction, it is clear that $f+\lc_{\m^\C}(e)\subseteq \tmr$.
 It is furthermore a section, which is proved as in \cite{KR71}, as groups
act by inner automorphisms of the Lie algebra, together with Lemma \ref{NTDS} following this theorem.
This proves \textit{1.}, \textit{2.} and \textit{4.} 

As for \textit{3.}, it follows directly from Theorem 11 in \cite{KR71}, which asserts that the affine space $f+\lie{c}_{\m^\C}(e)$ hits each $\Ad(G)^\C_\theta$
orbit exactly at one point, taking Remark \ref{rk adjoint case} 
and \textit{2.}  in Proposition \ref{lemma Gtheta} into account.

Statement \textit{5.} follows from the fact that $\TIG_0$ is strongly reductive, hence the statement
follows from Theorem 7 in \cite{Kos}, where a section for the Chevalley morphism for complex groups is defined,
together with Remark 19 in \cite{KR71} and its proof, where it is checked that $f+\lie{c}_{\m^\C}(e)$ defines
a section of the restriction of the Chevalley morphism to $\tmr\plonge\mr$.
\end{proof}
\begin{lm}\label{NTDS}
The Lie algebra $\lie{s}^\C$ is the image of a $\sigma$ and $\theta$-equivariant morphism
$\sll(2,\C)\to \g^\C$ where $\sigma$ on $\sll(2,\C)$ is complex conjugation and $\theta$ on $\sll(2,\C)$
is defined by 
$X\mapsto -\Ad\left(\begin{array}{cc}
             0&1\\
             1&0
            \end{array}
\right)(^tX)$.
\end{lm}
\begin{proof}
Consider the basis of $\sll(2,\R)$
\begin{equation}\label{eq basis sl2}
E=\frac{1}{{2}}\left(
\begin{array}{cc}
     1&-1\\
1&-1
   \end{array}
\right),\qquad 
F=\frac{1}{{2}}\left(
\begin{array}{cc}
     1&1\\
-1&-1
   \end{array}
\right),\qquad
W=\left(
\begin{array}{cc}
     0&1\\
1&0
   \end{array}
\right), 
\end{equation}
and note that $W\in\lie{sym}_0(2,\R)=:\m_{\sll}$, $E=\theta F$, so that $E+F\in\so(2,\R)$.

Consider $e_c,\ f_c,\ w$ as described in the preceeding proposition.
Then the map defined by
\begin{equation}\label{eq rho'}
\rho': E\mapsto ie_c,\ F\mapsto if_c,\ W\mapsto -w
\end{equation}
is the desired morphism. Indeed, it is $\sigma$-invariant by definition. Furthermore,
$\lie{so}(2,\R)\ni E+F\mapsto ie_c+if_c\in \h$ by construction. Finally, 
$\m_{\sll}$ is generated by $W$ and $E-F$, and so is $\s\cap\m$. Indeed, we must only
prove that $ie_c-if_c$ is not a multiple of $w$. But this follows from simplicity of $\sl(2,\C)$, the fact that $\s^\C$ is homomorphic to it and $w\neq 0$, which forces 
$S$-triples to be independent.
\end{proof}
\begin{rk}\label{rk diff sections}
Theorem \ref{theorem KR section} implies that the GIT quotient $\m^\C\sslash H^\C$ does not parameterize $H^\C$ orbits or regular 
elements, but rather $\Ad(H^\C)_\theta$ orbits, each of which contains finitely many $H^\C$-orbits.
This is a consequence of the fact that not all normal principal TDS's are $H^\C$ conjugate, 
which yields different sections for different choices of a TDS. See \cite{KR71} for more details.
\end{rk}
By the above remark, we will need to keep track of conjugacy classes of principal normal TDS's.
\begin{prop}\label{prop normal ppal triple}
 Let $\lie{s}^\C\subseteq \g^\C$ be a normal TDS, and let $(e,f,x)$ be a normal triple
 generating it. Then:
 \benum
 \item[1.] The triple is principal if and only if $e+f=\pm w$, where $w$ is defined by (\ref{eq w}). 
 \item[2.] There exist $e',\ f'$ such that $(e',f',w)$ is a TDS generating $\lie{s}^\C$ and
 $e'=\theta f'$. Under these hypothesis, $e'$ is uniquely defined up to sign. 
 \eenum
\end{prop}
\begin{proof}
 See Lemma 5 and Proposition 13 in \cite{KR71}.  
\end{proof}
In the classical setting of complex reductive Lie algebras, there is also a notion of principal TDS. These are defined to be Lie algebras
homomorphic to $\sl(2,\C)$ generated by regular nilpotents, except that regularity is now taken in the sense of the whole Lie 
algebra $\g^\C$, which need not coincide the notion for a given real form $\g$ (see Remark \ref{rk regular m}).

Let us recall some facts about three dimensional subalgebras. Let $\s^\C$ be a 
normal TDS (cf. Definition \ref{defi TDS}) generated by the normal S-triple $\{e,f,x\}$.
Let $n=\dim{\lc_{\g^\C}}(e)$. The adjoint representation induces a splitting

\begin{equation}\label{eq splitting 2}
 \g^\C\cong_{k=1}^n \oplus M_k
\end{equation}
into irreducible $\sll(2,\C)$-modules $M_k$, generated from the highest weight vector $e_k$ by the action of $f$, possibly
isomorphic to one another. Since
the highest weight vectors are annihilated by the action of $e$, it follows that $\lc_{\g^\C}(e)$ is generated
by the highest weight vectors. Note $\lc_{\g^\C}(e)$ is $\theta$-invariant. Given that $[x,\h^\C]\subset\h^\C$ and
$[x,\m^\C]\subset\m^\C$, then $e_k\in\m^\C$ or $e_k\in\h^\C$.
\begin{lm}\label{lemma centraliser tds}
Let $\g$ be a real reductive Lie algebra, and let $\s\subset\g$ be such that $\s^\C$ is a principal normal TDS of 
$\g^\C$ with generating normal triple  
$\{e,f,x\}$.   Let $\g_\lambda\subset \g^\C$ be the eigenspace of eigenvalue $\lambda$ for the action of $x$. 
Let $e_k$, $k=1,\dots, n$ be highest weight vectors for the action of $x$ with eigenvalues $m_k-1\geq 0$ $k=1,\dots, n$, and assume $m_k<m_{k+1}$, so that
$m_0\geq 1$. Then:  

1. If $m_1=1$, then 
$\g_0=\lc_\g^\C(\s^\C)=\lc_\h^\C(\s^\C)\oplus\z_{\m^\C}(\g^\C)=M_1^{\dim \lc_\g^\C(\s^\C)}$. 

3. Moreover, $\g\subset\g^\C$ is quasi split if and only if $\g_1=\z(\g^\C)$.

2. For all values of $k$, 
\begin{equation}\label{eq mk}
m_k-1:=\frac{\dim M_k-1}{2}.
\end{equation}
\end{lm}
\begin{proof}

To prove \textit{1.}, note it is clear that $\g_1=\lc_\g^\C(\s^\C)$. We need to prove $\g_1\cap \m^\C$ is central. Note that
$\lc_{\g^\C}(\s)=\lc_{\g^\C}(w)\cap \lc_{\g^\C}(ie_c)$, where $e_c$ and $w$ are as in Proposition \ref{restriction iso}. 
By Theorem 3.6 in \cite{KosBetti} $\lc_{\g^\C}(e)$ is fully composed
by nilpotent elements; however, all elements in $\lc_{\g^\C}(w)=\la^\C$ are semisimple, hence 
$$
\lc_{\g^\C}(w)\cap \lc_{\g^\C}(ie_c)=\lc_{\h^\C}(w)\cap \lc_{\h^\C}(ie_c).
$$

For \textit{2.}, by the proof of \textit{1.} above,  $\lc_{\g^\C}(\s)=\z(\g^\C)$ if and only if $\lc_{\h^\C}(w)$ is composed by semisimple
elements, which happens if and only if $\Delta_i=0$, for $\Delta_i$ as in (\ref{eq types roots}). Namely, of  if and only if  $\g$ is quasi-split.

Finally, \textit{3.} follows from \cite{KosBetti}{2.5}(c) and (d) (or simply, by the way the $M_k$'s are generated).
\end{proof}
\begin{rk}
Note that $m_k$ is an exponent of $G$ whenever $e_k\in\m^\C$. 
\end{rk}
\begin{cor}\label{cor TDSgps}
 Let $i:S\plonge G$ be a three dimensional subgroup corresponding to a three dimensional subalgebra $\s\subset\g$.
 Then $i$ is irreducible into the component of the identity $G_0$ 
 (namely, $Z_{G_0}(S)=Z(G_0)$) if and only if $G$ is quasi-split.
\end{cor}
\section{$G$-Higgs bundles}\label{section Higgs pairs}
For this section, we follow \cite{GGMHitchinKobayashi}.
\subsection{Basic theory}\label{basic theory}
Let $X$ be a smooth complex projective curve, and $L\to X$ be a holomorphic line bundle on $X$. Let
$(G,H, B,\theta)$ be a real reductive Lie group as defined in
Section \ref{Reductive groups}, and consider $\h$, $\m$, etc. as defined in Section \ref{liealg}. Note that 
by condition (\ref{diffeo}) in Definition \ref{K-reductive}, we have a representation
\begin{equation}\label{eq isotropy gp}
 \iota:H\to \GL(\m)
\end{equation}
which complexifies to $H^\C\laction\m^\C$. We will refer to both as the isotropy representation.
\begin{defi}\label{defi L twisted Higgs}
 An \textbf{$L$-twisted $G$-Higgs bundle} over $X$ is a pair $(E,\phi)$, where $E$ is a holomorphic principal
 $H^\C$-bundle on $X$ and $\phi\in H^0(X, E(\m^\C)\otimes L)$. Here,  $E(\m^\C)$ is the vector bundle associated to $E$ via the
 isotropy representation. When $L=K$ is the canonical bundle of $X$, these
 pairs are referred to simply as $G$-\textbf{Higgs bundles}.
\end{defi}
\begin{rk}\label{rk complex case included}
1. When $G$ is the real Lie group underlying a complex reductive Lie group, the above definition reduces to
the classical definition for complex groups given by Hitchin \cite{Duke}.
Indeed, if
$U<G$ is the maximal compact subgroup, then $G=(U^\C)_\R$, so $\m^\C=(i\lu)^\C=\g$ and the complexified 
isotropy representation is the adjoint representation.

2. Note that the above definition uses all the ingredients of the Cartan data of $G$ except the bilinear form $B$. Its 
role will become apparent in the definition of stability conditions, as well as the Hitchin 
 equations for $G$-Higgs bundles.
\end{rk}
 Given $s\in i\h$, we define:
\begin{equation}\label{defi parabolic}
\begin{array}{lll}
 \lp_s&=&\{x\in \h^\C\ |\ \Ad(e^{ts})(x)\textrm{ is bounded as } t\to\infty\},
 \\
 P_s&=&\{g\in H^\C\ |\ \Ad(e^{ts})(g)\textrm{ is bounded as } t\to\infty\},
 \\
 \lie{l}_s&=&\{x\in \h^\C\ |\ [x,s]=0\}=\lie{c}_\h(x)\},
 \\
 L_s&=&\{g\in H^\C\ |\ \Ad (e^{ts})(g)=g\}=C_{H^\C}(e^{\R s})\},
 \\
 \m_{s}&=&\{x\in\m^\C\ :\ \lim_{t\to 0}\iota(e^{ts})(x)\textrm{ exists}\},
\\
\m_{s}^0&=&\{x\in\m^\C\ :\ \iota(e^{ts})(x)=x\}.
\end{array} 
\end{equation}
We call $P_s$ and $\lp_s$ (respectively $L_s$ and $\lie{l}_s$) the \textbf{parabolic} (respectively \textbf{Levi}) \textbf{subgroup} 
 and \textbf{subalgebra} associated to $s$. For each $s \in i\h$, we define $\chi_s$, the character of $\lie{p}_s$ dual to $s$ via the bilinear form $B$. We note
 it is a strictly antidominant character of $\lp_s$ (cf. \cite{GGMHitchinKobayashi}).

Consider an $L$-twisted $G$-Higgs bundle $(E,\phi)$. Given a parabolic subgroup $P_s\leq H^\C$ and
$\sigma\in \Gamma(X,E(H^\C/P_s))$ a holomorphic reduction of the 
structure group to $P_s$,
let $E_\sigma$ denote the corresponding principal bundle. The isotropy representation restricts to actions 
$P_s\laction\m_s$, $L_s\laction\m_s^0$, so
it makes sense to
consider $E_\sigma(\m_s)$.
Similarly, any holomorphic reduction of the structure group $\sigma_L\in \Gamma(X,P_s/L_s)$ allows to take $E_{\sigma_L}(\m_s^0)$.

Let $F_h$ be the curvature of the Chern connection of $E$ with respect to a
$C^\infty$ reduction of the structure group $h\in \Omega^0(X,E(H^\C/H))$. Let $s\in i\h$, and let $\sigma\in \Gamma(X, E(H^\C/P_s))$
be holomorphic.
We define the \textbf{degree of $E$ with respect to $s$ 
 and the reduction $\sigma$} as follows:
 \begin{equation}\label{defi deg}
   \deg E(s,\sigma)=\int_X \chi_s(F_h).
 \end{equation}
An alternative definition of the degree when  the character $\chi_s$ lifts to
a character $\delta_s:P_s\to\C^\times$ is given by
\begin{equation}\label{rk alternative defi deg}
 \deg E(s,\sigma)=\deg(E\times_{\delta_s}\C).
\end{equation}
See \cite{GGMHitchinKobayashi} for the equivalence of both definitions.

 We can now define the stability of a $G$-Higgs bundle. This notion naturally depends on an element 
in $i\z$,  which has a special significance when $G$ is a group of Hermitian type (cf. Definition \ref{def htype}). 
\begin{defi}\label{defi polyst}
 Let $\alpha\in i\z$. We say that the pair $(E,\phi)$ is:
 \benum
 \item[1.] \textbf{$\alpha$-semistable} if for any $s\in i\h$ and any holomorphic 
 reduction of the structure group $\sigma\in \Gamma(X,E(H^\C/P_s))$ such that
$\phi\in H^0(X,E(\m_{s})\otimes L)$, then
 $$
 \deg E(s,\sigma)-B(\alpha,s)\geq 0.
 $$
 \item[2.] \textbf{$\alpha$-stable} if it is semistable and for any $s\in i\h\setminus \Ker(d\iota)$, given 
 any  holomorphic reduction $\sigma\in \Gamma(X,E(H^\C/P_s))$ such that
$\phi\in H^0(X,E(\m_{s})\otimes L)$,  then
 $$
 \deg E(s,\sigma)-B(\alpha,s)> 0.
 $$
 \item[3.] \textbf{$\alpha$-polystable} if it is $\alpha$-semistable and whenever
 $$
 \deg E(s,\sigma)-B(\alpha,s)= 0
 $$
 for some $s$ and $\sigma$ as above, there exists a 
 reduction $\sigma'$ to the 
 corresponding Levi subgroup $L_{s}$ such that $\phi$ takes values in 
 $H^0(X, E_{\sigma'}(\m_{s}^0)\otimes L)$.
 \eenum
\end{defi}
The {\bf moduli space of $\alpha$-polystable  $L$-twisted  $G$-Higgs bundles}
is defined as the set $\mc{M}_L^\alpha(G)$ of isomorphism classes
of such objects. It coincides with the moduli space of $S$-equivalence classes 
of $\alpha$-semistable Higgs bundles 
For a more detailed account of these notions, as well as the geometry of $\mc{M}_L^\alpha(G)$,
we refer the reader to \cite{GGMHitchinKobayashi}.

Parameters appear naturally when studying the moduli problem from the gauge-theoretic
point of view.  This relation is established by the Hitchin--Kobayashi correspondence as follows (cf. 
\cite{GGMHitchinKobayashi}).
\begin{thm}\label{theorem Hitchin eqns for alpha moduli}
Let $\alpha\in i\z$. Let $L\to X$ be a line bundle, and let $h_L$ be a
Hermitian metric on  $L$. Fix $\omega$ a K\"ahler form on $X$. An $L$-twisted 
$G$-Higgs bundle $(E, \phi)$ is $\alpha$-polystable
if and only if there exists $h\in \Omega^0(X,E(H^\C/H))$ satisfying:
\begin{equation}\label{HK corresp}
F_h-[\phi, \tau_h(\phi)]\omega=-i\alpha\omega
\end{equation}
where $F_h$ is the curvature of the Chern connection on $E$ corresponding to
$h$, and  $\tau_h:\Omega^0\left(E(\m^\C\otimes L)\right)\to\Omega^{0}\left(E(\m^\C)\otimes L\right)$ 
is the antilinear involution on $\Omega^0(E(\m^\C)\otimes L)$ determined by $h$ and
$h_L$.
\end{thm}
In the above theorem, we fix a $G$-Higgs bundle and look for a solution of equation 
(\ref{HK corresp}). From a different perspective, we can construct the gauge
moduli  space associated  to equation (\ref{HK corresp}) as follows.  
Fix a $C^\infty$ principal $H^\C$-bundle $\E$. Given a reduction $h\in \Omega^0(X,\E(H^\C/H))$,  
let $\E_h$ be the corresponding principal $H$-bundle. Consider pairs $(A,\phi)$ where  $A$ is a connection on 
$\E_h$, and $\phi\in \Omega^0(X,\E_h\otimes L)$  
is holomorphic with respect to the holomorphic structure defined by $A$ and both satisfy (\ref{HK corresp}). The gauge group 
$ \mc{H}= \Omega^0(X,\Ad\ \E_h)$--where $\Ad\ \E_h:=\E_h\times_{\Ad}H$ is the associated bundle of groups--
acts on solutions of (\ref{HK corresp}). 
Let $\mc{M}^{gauge,\alpha}_{L,\E_h}(G)$ be the gauge moduli space obtained by
taking the quotient of the space of solutions to
(\ref{HK corresp}) by this action. In a similar fashion, we can define the moduli space $\mc{M}_{L,\E}^\alpha(G)\subset \mc{M}_{L}^\alpha(G)$  of
$\alpha$-polystable $G$-Higgs bundles with underlying smooth $H^\C$-bundle $\E$. Theorem 
\ref{theorem Hitchin eqns for alpha moduli} defines a homeomorphism
\begin{equation}\label{prop gaue=poly}
\mc{M}_{L,\E_h}^\alpha(G)\cong \mc{M}_{L,\E}^{gauge,\alpha}(G). 
\end{equation}
In the case $L=K$, for $\alpha=0$, there is a third moduli space that can be considered. Let 
$
\mc{R}(G)=\Hom^+(\pi_1(X),G)/G
$
be the quotient of the set of reductive homomorphism $\rho:\pi_1(X)\to G$ by the conjugation action of $G$. Combining
the homeomorphism (\ref{prop gaue=poly}) with Corlette--Donaldson's theorem \cite{Do, Corlette}, from each $\rho\in\mc{R}(G)$ one
obtains  a  polystable Higgs bundle $(E,\phi)\in\mc{M}_{L,\E}^\alpha(G)$.
This induces a homeomorphism 
\begin{equation}\label{thm NAH}
 \mc{R}(G)\cong\mc{M}_K^0(G). 
\end{equation}
This correspondence is the basic  content of \textbf{non-abelian Hodge theory}.
\subsection{Topological type of Higgs bundles}\label{subs topo type Higgs}
Given a $C^\infty$ principal bundle $\E$, its isomorphism class is determined by a topological invariant, 
which in the case when $G$ is connected is given  by  an element $d\in \pi_1(H)$. 
This goes as follows: consider the
short exact sequence
$$
1\to \pi_1(H^\C)\to\widetilde{H^\C}\to H^{\C}\to 1.
$$
Then, since $\dim_\R(X)=2$, and the  fundamental group of a Lie group 
is abelian (see Theorem 7.1 in  \cite{BD}), one has that 
$H^2(X,\pi_1(H^\C))\cong\pi_1(H^\C)\cong\pi_1(H)$, where
the last isomorphism follows from the fact that $H$ is a deformation retract of $H^\C$.
So through the associated long exact sequence in cohomology one associates 
to each class 
$[\E]\in H^1(X,H^\C)$ an element $d(E)\in\pi_1(H)$. 
In particular, given a $G$-Higgs bundle, $(E,\phi)$, one may consider the
class corresponding 
to the differentiable principal bundle underlying
$E$. 
Fixing the topological class $d\in \pi_1(H)$, we can consider 
the  subspace $\mc{M}_{L,d}^\alpha(G)\subset \mc{M}_L^\alpha(G)$ 
consisting of isomorphism classes of 
$\alpha$-polystable $L$-twisted  $G$-Higgs bundles with
class $d$.

In the case of groups of Hermitian type, 
there is an equivalent invariant that one can define called 
the {\bf Toledo invariant}.
The original definition of this invariant in the context of 
representations of the fundamental group 
is due to Toledo \cite{Domingo} when
 $G=\PU(n,1)$, generalised by several authors for the various simple classical 
and  exceptional groups of Hermitian type, and extended to arbitrary groups
of Hermitian type by Burger--Iozzi--Wienhard 
\cite{BIWMaxToledo}.
In the context of $L$-twisted $G$-Higgs bundles the Toledo invariant has been
defined for arbitrary groups of Hermitian type in \cite{BGR}. These two
general definitions naturally coincide when $L=K$.

Let $G$ be a simple Hermitian Lie group such that $G/H$ is irreducible. 
In this situation 
the centre $\liez$ of $\lieh$ is isomorphic to $\R$, and  
the adjoint action of an element $J\in \liez$ defines an almost complex 
structure on $\liem=T_o(G/H)$, where $o\in G/H$ corresponds to the coset $H$, 
making the  symmetric space $G/H$ into a K\"ahler manifold. 
The almost complex structure $\ad(J)$ gives a decomposition 
$\liem^\C=\liem^+ + \liem^-$ in $\pm i$-eigenspaces, which is $H^\C$-invariant.
An immediate consequence of  this decomposition  for an $L$-twisted
$G$-Higgs bundle $(E,\varphi)$ is that it  gives a 
bundle decomposition $E(\liem^\C)=E(\liem^+) \oplus E(\liem^-)$ 
and hence the Higgs field decomposes as $\varphi=(\beta,\gamma)$,
where $\beta\in H^0(X,E(\liem^+)\otimes L)$ and 
$\gamma\in H^0(X,E(\liem^-)\otimes L)$.

There is a character of $\chi_T:\lieh^\C\to \C$ called the Toledo character
and a  rational number $q_T$ such that $q_T\chi_T$ lifts to a character
$\tilde{\chi_T}$ of $H^\C$. We define the Toledo invariant of an $L$-twisted
$G$-Higgs bundle 
$(E,\phi)$ by
\begin{equation}\label{eq toledo inv}
 T(E)=\frac{1}{q_T}\deg(E\times_{\wt{\chi_T}}\C^\times).
\end{equation}

One can define the ranks of $\beta$ and $\gamma$ (see \cite{BGR}). These are
integers bounded by the rank of the symmetric space $G/H)$.
The following can be found in \cite{BGR} (see Theorem 3.18 and the discussion preceeding Theorem 4.14 therein):
\begin{prop}\label{big prop toledo}
Let $G$ be a simple group of Hermitian type with irreducible associated symmetric space, so that
$\z(\h)= i\R$. Let $(E,(\beta,\gamma))$ be an $L$-twisted $G$-Higgs bundle, $\alpha$-semistable for some $\alpha=i\lambda J$.
Then:

1. The Toledo invariant satisfies the Milnor--Wood inequality:
\begin{equation}\label{eq MW}
-\rank \beta\cdot d_L-\lambda\left(\frac{\dim \m}{N}-\rank\beta\right) \leq
T(E)\leq \rank \gamma\cdot d_L-\lambda\left(\frac{\dim
  \m}{N}-\rank\gamma\right), 
\end{equation}
where $N$ is the dual Coxeter number of
$\lieg^\C$, and $d_L$ es the degree of $L$.
Moreover, when $G$ is of tube type (i.e. $G/H$ is biholomorphic to a tube 
domain), $T$ (resp. $-T$) is maximal  if and only if
$\gamma(x)\in \mr^+$ (resp. $\beta(x)\in \mr^-$) for all $x\in X$.

2. There exists a canonical $k>0$ such that 
$$
\ol{d}(E)=k T(E),
$$ 
where $\ol{d}(E)$ denotes the projection of the topological class $d(E)$ to the torsion free part of $\pi_1(H)$.
\end{prop}
Now, the curvature of a principal bundle $E$  determines the torsion free part
of its topological class $d(E)$ via the first Chern class.
This information is partially determined by the parameter and viceversa. Let $\z(\g)^\perp$ be the orthogonal complement of
$\z(\g)$ inside $\g$.
\begin{prop}\label{prop topo type and pars}
 Let $(E,\phi)$ be a an $\alpha$-polystable Higgs bundle. Let 
 $\alpha=\alpha_0+\alpha_1$, where $\alpha_0\in i\z(\h)\cap i\z(\g)$ and
 $\alpha_1\in i\z(\h)\cap i\z(\g)^\perp$ are the projections to $i\z(\g)$ and $i\z(\g)^\perp$.
 Then, $\alpha_0$ is fully determined by and determines $\ol{d}(E)$. 
\end{prop}
\begin{proof}
 In order to see this, we note that $\alpha_0$ is determined by the image $\chi(\alpha)$ for all 
 $\chi\in\Char(\h^\C)\cap\Char(\g^\C)$. Now, $[H^\C,H^\C]$-invariance, implies that it makes sense to evaluate 
 $\chi(F_A-[\phi,\phi^*])$, and moreover, the evaluation of all such characters determines $F_A-[\phi,\phi^*]$. Furthermore,
 for $\chi\in\Char(\h^\C)\cap\Char(\g^\C)$, we have $\chi([\phi,\phi^*])=0$, as $[\phi,\phi^*]$ is a two form with values in 
 $[\g^\C,\g^\C]$. This proves the statement.
 \end{proof}
\begin{rk}{(Topological type and parameters).}\label{rk top pars}
A non zero parameter $\alpha\neq 0$ makes sense only when $\z(\h)\neq 0$. This includes the case of real groups underlying a 
complex non-semisimple reductive Lie group $(G^\C)_\R$ (cf. Remark \ref{rk complex case included}), 
or the case of simple groups of Hermitian type (cf. Definition \ref{def
  htype}).

Proposition \ref{prop topo type and pars} implies that when $G^\C$ has a positive dimensional centre, the topology of the
bundle fully determines the parameter, and conversely, the torsion free piece of the topological type is also
determined by the parameter. On the other hand, the same result implies that for Hermitian groups we are in the opposite
situation, as these are characterised by having large $\z(\h)\cap\z(\g)^\perp$.
\end{rk}
\subsection{Morphisms induced by  group homomorphisms}
Consider a morphism of reductive Lie groups $f:(G',H',\theta',B')\to (G,H,\theta,B)$.
\begin{defi}\label{def extended gp}
Given
 a $G'$-Higgs bundle $(E',\phi')$, we define the \textbf{extended $G$-Higgs 
bundle} (by the morphism $f$) to
 be the pair 
 $(E'(H^\C),df(\phi'))$, where $E'(H^\C)$ is the principal $H^\C$-bundles associated to $E'$ via $f$. Note that
 $df(\phi')$ is well defined as $df$ commutes with the adjoint action,
\end{defi}
 These pairs satisfy the following.
 \begin{prop}\label{prop morphism moduli}
  With  notation as above,  if the $G'$-Higgs bundle $(E',\phi')$ is $\alpha$-polystable,
  and $df(\alpha)\in i\z(\h)$,
  then the corresponding extended $G$-Higgs bundle $(E,\phi)$ is $df(\alpha)$-polystable.
 \end{prop}
 \begin{proof}
  By Theorem \ref{theorem Hitchin eqns for alpha moduli}, polystability of $(E',\phi')$ is equivalent to
  the existence of a solution to the Hitchin equation (\ref{HK corresp}). Let $h'$ be the corresponding
  solution. Now, $h'$ extends to a Hermitian metric on $E$, as $f$ defines a map 
  $$\Omega^0(E((H')^\C/H'))\to\Omega^0(E(H^\C/H)).$$ Let $h\in\Omega^0(E(H^\C/H))$ be the image of $h'$ 
  via that map. Clearly $F_{h'}$ is a two form with values in $\h$. But
  $F_{h}=df(F_{h'})$, where $df$ is evaluated on the coefficients of the 2 form $F_{h'}$, as the canonical 
  connection $\nabla_{h}$ is defined by
  $$
  dh=\langle \nabla_{h}\cdot, \cdot\rangle+\langle\cdot,\nabla_{h}\cdot\rangle.
  $$
  Since $dh=df(dh')$, it follows that $\nabla_h=df(\nabla_{h'})$ solves the modified equations. 
  By Theorem \ref{theorem Hitchin eqns for alpha moduli}, this gives a polystable Higgs bundle, which by construction
   must be $(E,\phi)$.   
 \end{proof}
 As a corollary we have the following.
\begin{cor}\label{prop HKR alpha moduli}
With the above notation, if $\alpha\in i\z'$  is such that
$df(\alpha)\in i\z$, then the map 
$$
(E',\phi')\mapsto\left(E'(H^\C),df(\phi')\right)
$$
induces a morphism
$$
\mc{M}_d^\alpha(G')\to\mc{M}_{f_*d}^{df(\alpha)}(G),
$$
where $f_*d$ is the topological type of $E(H^\C)$. When  $G$ is connected, this corresponds
to the image via the map $f_*:\pi_1(H')\to\pi_1(H)$ induced by $f$.
\end{cor}
\begin{lm}\label{lemma normaliser}
 Let $G'\subseteq G$ be two Lie groups. Let $E,\ \widetilde{E}$ be two  
 principal $G'$-bundles over  $X$, and suppose there exists a morphism 
 $$F:\ E(G)\to\widetilde{E}(G)$$
 of 
 principal $G$-bundles. 
 Then there exists an isomorphism of
 principal $N_G(G')$-bundles $E(N_G(G'))\cong \widetilde{E}(N_G(G'))$.
\end{lm}
\begin{proof}
By Theorem 10.3 in \cite{Steenrod}, $F$ is an isomorphism. Denote $N_G(G')$ by $N$. Choose common trivialising neighbourhoods $U_i\to X$ such that
 $$
 E|_{U_i}\cong U_i\times G'\qquad\widetilde{E}|_{U_i}\cong U_i\times G'.
 $$
 Let $g_{ij},\wt{g_{ij}}$ be the transition functions for $E$ and $\widetilde{E}$ respectively and
 define $F_i:=F|_{E(G)|_{U_i}}$. Then we have the following commutative diagram:
 $$
 \xymatrix{
 U_j\times G\ar[ddd]\ar[rrr]&&&U_j\times G\ar[ddd]\\
&(x,g)\ar@{|->}[r]\ar@{|->}[d]&(x,F_j(g))\ar@{|->}[d]&\\ 
 &(x,gg_{ij})\ar@{|->}[r]&(x,F_i(gg_{ij}))=(x,F_j(g)\wt{g}_{ij})&\\
 U_i\times G \ar[rrr]&&&U_i\times G}.
$$
Now, since for any $n\in N,\ g\in G'$ we have that $ng\in N$, it follows that for all $i, j$'s 
$F_i(N)=F_j(N)\widetilde{g_{ij}}$. Namely, the image bundle of $E(N)$ is isomorphic to
$\widetilde{E}(N)$.
\end{proof}

\subsection{Deformation theory}
The deformation theory of Higgs bundles was studied by several authors, amongst which 
we cite \cite{BisRam} in the setting of arbitrary pairs, and
\cite{GGMHitchinKobayashi} and references therein for $G$-Higgs bundle when
$G$ is a real reductive  Lie groups. 
Let us recall the basics. 

The deformation complex of a $G$-Higgs bundle $(E,\phi)\to X$ is:
\begin{equation}\label{deformation complex}
 C^\bullet: [d\phi,\ \cdot\ ]:E(\h^\C)\to E(\m^\C)\otimes L
\end{equation}
whose hypercohomology sets fit into the exact sequence
\begin{eqnarray}\label{eq hypercohomology}
 0\to \H^0(C^\bullet )\to H^0(X,E(\h^\C))\to H^0(X,E(\m^\C)\otimes L)\to\\\nonumber \H^1(C^\bullet)
 \to H^1(E(\h^\C))\to H^1(X,E(\m^\C)\otimes L)\to\H^2(C^\bullet)\to 0
\end{eqnarray}
In particular, we see that $\H^0(C^\bullet)=\lie{aut}(E,\phi)$, where $\lie{aut}(E,\phi)$ denotes the Lie algebra
of the automorphism group of $(E,\phi)$. 

On the other hand, the space of infinitesimal deformations of a pair $(E,\phi)$
is canonically isomorphic to $\H^1(C^\bullet)$ (Theorem 2.3 \cite{BisRam}). Hence, the expected dimension of
the moduli space is the dimension of $\H^1(C^\bullet(E,\phi))$ at a smooth point $(E,\phi)$.
\begin{defi}
 A $G$-Higgs bundle $(E,\phi)$ is said to be \textbf{simple} if
 $$
 \Aut (E,\phi)=H^0(X, \Ker(\iota)\cap  Z(H^\C)).
 $$
$(E,\phi)$ is said to be \textbf{infinitesimally simple} if
$$
\H^0(X,C^\bullet)\cong H^0(X,(\Ker(d\iota)\cap\z(\h^\C))).
$$
Here $\iota$ is the isotropy representation of $H^\C$ in $\liem^C$.
\end{defi}
These notions are deeply related to smoothness of the points of the moduli space, as the next result shows. For an
alternative proof of the following proposition, see \cite{BGGSOstar}.
\begin{prop}\label{prop simple and complex stable then smooth}
Let $(E,\phi)$ be a stable and simple $G$-Higgs bundle, where $(G,H,\theta, B)$ is a real strongly reductive
Lie group. Let $\z_\m=\z(\g^\C)\cap\m^\C$.
 Assume that $\H^2(C^\bullet)=H^1(X,\z_\m\otimes L)$. 

Then $(E,\phi)$ is a smooth point of the moduli space.
\end{prop}
\begin{proof}
This follows from Theorem 3.1 in \cite{BisRam} applied to the algebraic group $H^\C$ and
the isotropy representation $\iota: H^\C\to \Aut(\m^\C)$.

Indeed, singularities of the moduli space can be of orbifold origin, which are discarded by 
the simplicity assumption, or
caused by the existence of obstructions to deformations, measured by $\H^2(C^\bullet)$. 
Now, although Theorem 3.1 in \cite{BisRam} assumes the
vanishing of the whole hypercohomology
group, a simple argument shows that the centre plays
no role in obstructing infinitesimal deformations.

To understand this, let $\m_{ss}=[\g^\C,\g^\C]\cap\m^\C$, $\h^\C_{ss}=(\h\cap\g_{ss})^\C$, $\z_\m=\z(\g^\C)\cap\m^\C$, and $\z_{\h}=\h^\C\cap\z(\g^\C)$. 
Observe that $\Ad:G\to\Aut(\g)$ factors through $G_{ss}:=[G,G]$, which implies that
 $$
 E(\h^\C)\cong E_{ss}(\h^\C_{ss})\oplus X\times \z_{\h},\ E(\m^\C)\cong E(\m^\C_{ss})\oplus X\times \z_{\m}.
 $$
Moreover, 
$
[\phi,E(\h^\C)]=[\phi_{ss},E_{ss}(\h^\C_{ss})]\subset E_{ss}(\m^\C_{ss}),
$
which implies that the complex $C^\bullet$ splits into a direct sum of complexes $C^\bullet=C^\bullet_{ss}\oplus Z(C^\bullet)$
where 
\begin{equation}\label{eq Css}
C^\bullet_{ss}:E_{ss}(\h^\C_{ss})\to E_{ss}(\m^\C_{ss})\otimes L 
\end{equation}
and
\begin{equation}\label{eq Z(C)}
Z(C^\bullet):=X\times \z_\h\stackrel{0}{\to}
X\times \z_\m\otimes L. 
\end{equation}
Hence 
\begin{equation}\label{hypercohomology splits}
\H^i(C^\bullet)=\H^i(C^\bullet_{ss})\oplus \H^i(Z(C^\bullet)) 
\end{equation}

Now, following the proof of Theorem 3.1 in \cite{BisRam}, we have complexes:
$$
\mc{G}_n^\bullet: p_n^*E(\h^\C)\otimes \C[\epsilon]/\epsilon^n\to
p_n^*E(\m^\C)\otimes L\otimes \C[\epsilon]/\epsilon^n\to 0
$$
where $p_n:X\times\mathrm{Spec}(\C[\epsilon]/\epsilon^n)\to X$ is the projection on the first factor.
 With this we obtain a short exact sequence of complexes
\begin{equation}\label{eq SES deformations}
0\to C^\bullet\otimes \langle\epsilon^n\rangle\to\mc{G}_{n+1}\to \mc{G}_{n}\to0 
\end{equation}
which splits into the direct sum of
$$
0\to C_{ss}^\bullet\otimes \langle\epsilon^n\rangle\to\mc{G}_{n+1,ss}\to \mc{G}_{n,ss}\to0
$$
and 
\begin{equation}\label{eq sequence centres}
0\to Z(C^\bullet)\otimes \langle\epsilon^n\rangle\to Z(\mc{G}_{n+1})\to Z(\mc{G}_{n})\to0
\end{equation}
where $\mc{G}_{n,ss}^\bullet,\ Z(\mc{G}_{n})$ are defined similarly to (\ref{eq Css}), (\ref{eq Z(C)}). Hence, the long exact sequence
 in hypercohomology induced by (\ref{eq SES deformations}), also splits. This, together with (\ref{hypercohomology splits}) and Theorem 3.1 in \cite{BisRam}, 
 implies that the only obstructions to deformation come from the long exact sequence induced by (\ref{eq sequence centres}). We see that this long exact sequence splits into short exact sequences
$$
0\to\H^i(Z(C^\bullet))\to \H^i(Z(\mc{G}_{n+1}^\bullet))\to \H^i(Z(\mc{G}_{n}^\bullet))\to 0
$$
and so we may conclude that no obstruction to deformation lies in $\H^2(Z(C^\bullet))$.
\end{proof}

The above has its counterpart in terms of the gauge moduli space. This is done in full detail in
\cite{GGMHitchinKobayashi} in the case $\alpha=0$, $L=K_X$. We extend it here to 
the deformation complex of an arbitrary
pair. Coming back to the gauge moduli setup developed in Section \ref{basic theory}, Let $(A,\phi)$ be a pair of
a connection on some differentiable principal $H^\C$-bundle $\E$, and $\phi\in\Omega^0(\E(\m^\C)\otimes L)$. Then, 
if $h$ is the solution to (\ref{HK corresp}) corresponding to $(A,\phi)$, we get a deformation complex:
\begin{displaymath}\label{eq complex connections}
C^\bullet(A,\phi):\xymatrix{
\Omega^0(X,E_h(\h))\ar[r]^-{d_0}&\Omega^1(X,E_h(\h))\oplus\Omega^0(X,E_h(\m^\C)\otimes L)\\
\ar[r]_-{d_1}&\Omega^2(X,E_h(\h))\oplus\Omega^{0,1}(X,E_h(\m^\C)\otimes L),
}
\end{displaymath}
where $E_h$ is the reduction of $E$ to an principal $H$-bundle given by $h$, and the maps are defined by
\begin{eqnarray}\label{eqn d's}
 d_0(\psi)=(d_A{\psi},[\phi,\psi])\\\nonumber
 d_1(\stackrel{.}{A},\stackrel{.}{\phi})=
 (d_A(\stackrel{.}{A})-[\stackrel{.}{\phi},\tau\phi]\omega-[\phi,\tau\stackrel{.}{\phi}]\omega,\ol{\partial}_A\stackrel{.}{\phi}+[\stackrel{.}{A}^{0,1},\phi])
\end{eqnarray}
\begin{defi}\label{defi irreducible}
 A pair $(A,\phi)$ is said to be \textbf{irreducible} if its group of automorphisms 
  \begin{equation}\label{eq Aut(A,phi)}
   \Aut(A,\phi):=\{h\in \mc{H}\ :\ h^*A=A,\ \iota(h)(\phi)=\phi\}=Z(H)\cap\Ker(\iota). 
  \end{equation}
  It is said to be \textbf{infinitesimally irreducible} if 
  $$
   \lie{aut}(A,\phi):=\Lie(\Aut(A,\phi))=\z(\h)\cap \Ker d\iota. 
  $$
\end{defi}
The following two propositions are explained in full detail in \cite{GGMHitchinKobayashi}
for moduli spaces of ($0$-polystable) Higgs bundles. For the general case, arguments are also standard
and consist in resolving the hypercohomology complex $\H^1(C^\bullet(E,\phi))$
and choosing harmonic representatives (see for example \cite[VI.8]{Kobayashi}).

\begin{prop}\label{iso 0 hypercohomology}
 Let $(E,\phi)\in \mc{M}^\alpha_{L,d}(G)$, and let $(A,\phi)\in \mc{M}^{gauge}_{L,d}(G)$
 be its corresponding gauge counterpart.
 Assume they are both smooth points of their respective moduli. Then
 $$
 \H^0(C^\bullet(E,\phi))\cong \H^0(C^\bullet(A,\phi))
 $$
\end{prop}
\begin{prop}\label{iso 1 hypercohomology}
 Let $(E,\phi)\in \mc{M}^\alpha_{L,d}(G)$, and let $(A,\phi)\in \mc{M}^{gauge}_{L,d}(G)$
 be its corresponding gauge counterpart. Then
 $$
 \H^1(C^\bullet(E,\phi))\cong \H^1(C^\bullet(A,\phi))
 $$ 
\end{prop}
\begin{prop}\label{prop HKR connections} 
Under the correspondence established by Theorem \ref{theorem Hitchin eqns for alpha moduli},
stable Higgs bundles  correspond to infinitesimally irreducible solutions to (\ref{HK corresp}). On the
other hand, simple and stable bundles correspond to irreducible solutions.
\end{prop}
\section{The Hitchin map and the Hitchin--Kostant--Rallis section}\label{section construction HKR section}
Let $(G,H,\theta,B)$ be a reductive Lie group as in Definition \ref{K-reductive}, and let $\h$, $\m$, $\la$, etc.
be as in Sections \ref{liealg} and \ref{section kostant-rallis section}.

Consider the Chevalley morphism defined in Section \ref{section kostant-rallis section}:
\begin{equation}\label{eq Cheva bis}
\chi:\m^\C\to{\la^\C}/W({\la^\C}).
\end{equation}
This map is  $\C^\times$-equivariant. In particular, it induces a morphism
$$
h_L:\m^\C\otimes L\to{\la^\C}\otimes L/W({\la^\C}).
$$
The map $\chi$ is also $H^\C$-equivariant, thus defining a morphism
\begin{equation}\label{defi stacky Hitchin map}
h_L: \mc{M}^\alpha_L(G)\to B_{L}(G):=H^0(X,{\la^\C}\otimes L/W({\la^\C})).
\end{equation}
\begin{defi}\label{defi hitchin map}
The map $h_L$ in (\ref{defi stacky Hitchin map})
is called the \textbf{Hitchin map}, and the space $B_{L}(G)$ is called the \textbf{Hitchin base}.
\end{defi}
\begin{prop}\label{prop dim BL}
 Let $a=\dim\la^\C$, and let $\z_\m=\z(\g^\C)\cap\m^\C$.
 Let  $\tg^\C\subset \g^\C$ be the complexification of the maximal split subalgebra defined in Section 
 \ref{section max split}. Assume $\deg L\geq g-1$. 
 Then
\begin{equation}\label{eq dim BL}
\dim B_{L}(G)= \frac{d_L}{2}(\dim\tg^\C)+a\left(\frac{d_L}{2}-g+1 \right)+ h^1(L)\dim\z_\m .
\end{equation}
\end{prop}
\begin{proof}
By definition  
\begin{eqnarray}\nonumber
\dim B_{L}(G)&=&\sum_{e_k\in\m^\C}h^0(L^{m_k})=\sum_{e_k\in\m^\C }(m_kd_L-g+1)+\sum_{e_k\in\m^\C } h^1(L^{m_k})
\\\nonumber &=&
\sum_{e_k\in\m^\C}\left((2m_k-1)\frac{d_L}{2}-g+1+\frac{d_L}{2}\right)+h^1(L)\dim \z_\m
\\\nonumber
&=& \frac{d_L}{2}(\dim\tg_{ss}^\C+\dim\z_\m)+a\left(\frac{d_L}{2}-g+1 \right)+h^1(L)\dim \z_\m,
\end{eqnarray}
which yields (\ref{eq dim BL}) since by definition $\z(\tg^\C)=\z_\m$.
\end{proof}
\begin{cor}
If $L=K$ and $G=(U^\C)_\R$ is the real group undelying a complex reductive subgroup, then 
$
\dim B_{L}(G)=(g-1)\dim \u^\C+\dim\z(\u^\C).
$
\end{cor}
\begin{proof}
 We need only note that in this case $\m^\C=\h^\C=\lie{u}^\C$, and $\tg$ is the split real form of $\u^\C$ (cf. 
 Remark \ref{split form complex}). Hence, 
 $\dim\tg^\C=\dim\u^\C$, and $\dim\z_\h=\dim\z_\m=\dim\z(\u^\C)$.
\end{proof}
\begin{rk} We will see later on that the dimension of $B_{L}(G)$ fails to be half the dimension
 of the moduli space unless $L=K$, the case considered by Hitchin \cite{Duke}.
\end{rk}
 In what follows, we proceed to the construction of a section of the Hitchin map (\ref{defi stacky Hitchin map}). This generalizes
 Hitchin's construction \cite{Teich} 
in essentially two ways. 
First of all, Hitchin considers the case  $L=K$, and 
he builds the section into $\mc{M}_K(G^\C)$ for a
complex Lie group $G^\C$ of adjoint type. A consequence of this is that $\alpha=0$, as it happens for all 
semisimple groups (see Remark \ref{rk top pars}). Hitchin then checks that the monodromy of the corresponding 
representations takes values in $G_{split}$, the split real form of $G^\C$, 
so it is implicit in his construction that the section factors
through $\mc{M}_K(G_{split})$. 
In what follows, we consider the existence of the section for  arbitrary real reductive Lie groups, allowing arbitrary 
$\alpha\in i\z(\h)$, and twisting by an arbitrary line bundle $L$;
this requires the implementation
of new techniques to prove stability and smoothness results. Moreover, our
section is directly constructed into the moduli space of $G$-Higgs bundles; in particular, into $\mc{M}_K(G_{split})$
when $G=G_{split}$ is the split
real form of a complex reductive Lie group $G^\C$ and $K=L$; in the latter case, this is precisely a factorization
of Hitchin's section through $\mc{M}_K(G_{split})$. Recall (cf. Remark \ref{rk top pars})
that $\alpha\in i\z(\h)$ decomposes as $\beta+\gamma\in \z(\h)\cap\z(\g)\oplus\z(\h)\cap\z(\g)^\perp$. Then 
$\beta$ is determined by the topology of the bundle and determines its torsion free part. As for $\gamma$, it is
not of topological nature. Amongst groups with $\z(\h)\cap\z(\g)^\perp\neq 0$ we find groups of Hermitian type
(such as $\Sp(2n,\R)$, $\SU(p,q)$, $\SO^*(2n)$ and $\SO(2,n)$), or any group containing one amongst its simple factors.
On the other hand, $\z(g)^\perp\cap\z(\h)=0$ implies that the parameter is purely topological. This includes the case
of complex reductive Lie groups. Indeed, $\z(\g^\C)=\z(\lu)\oplus i\z(\lu)$, and so $\z(\g^\C)^\perp\cap\z(\lu)=0$. 

\subsection{Some representation theory}\label{section reminder}
The content of this section can be found in \cite{KosBetti,KR71}.

Choose $\s^\C\subset\g^\C$ a principal normal TDS (cf. Definition \ref{defi TDS}), defined by the  
homomorphism (\ref{eq rho'}) of Lemma \ref{NTDS} 
\begin{equation}\label{eq rho' bis}
 \rho':\lie{sl}(2,\C)\to \s^\C\subset \g^\C.
\end{equation}
which is $\sigma$ and $\theta$-equivariant for the action of $\sigma$ and $\theta$ on $\sll(2,\C)$ as defined in Proposition \ref{NTDS}. Recall from (\ref{eq cartan dec}) that the 
Cartan decomposition of
$\sll(2,\R)$ under $\theta$ is
\begin{equation}\label{eq cartan sl2}
\sll(2,\R)\cong\so(2)\oplus \lie{sym}_0(2,\R), 
\end{equation}
which identifies $\so(2)$ to trace cero diagonal matrices, and $\lie{sym}_0(2,\R)$ to real antidiagonal matrices.

The image under $\rho'$ of the standard basis
\begin{equation}\label{eq normal triple sl}
\frac{1}{\sqrt{2}}\left(
\begin{array}{cc}
     0&1\\
0&0
   \end{array}
\right)\mapsto e,\quad
\frac{1}{\sqrt{2}}\left(
\begin{array}{cc}
     0&0\\
1&0
   \end{array}
\right)\mapsto f,\quad
\frac{1}{{2}}\left(
\begin{array}{cc}
     1&0\\
0&-1
   \end{array}
\right)\mapsto x  
\end{equation}
is a  normal triple $(e,f,x)$ (cf. Definition \ref{defi TDS}).

By $\theta$-equivariance, $\rho'=\rho'_+\oplus\rho'_-$ where 
\begin{equation}\label{eq rho'+-}
 \rho'_+:\so(2,\C)\cong\C\to\h^\C,\qquad  \rho'_-:\lie{sym}_0(2,\C)\cong\C^2\to\m^\C.
\end{equation}
In particular, $\rho'_+$ fits into  a 
commutative diagram
\begin{equation}\label{eq comm diag}
 \xymatrix{
\C\ar[r]^{\rho'_+}\ar[d]_\iota&{\h^\C}\ar[d]^\iota\\
\lie{gl}(\C^2)\ar[r]&\lie{gl}({\m^\C}).
}
\end{equation}
We claim that the restriction of  $\rho'$ to $\sll(2,\R)$ lifts to a $\theta$-equivariant group homomorphism 
\begin{equation}\label{eq rho}
\rho:\SL(2,\R)\to G.
\end{equation}
taking $\SO(2)$ to $H$. Indeed, by connectedness of $\SL(2,\R)$ and the polar decomposition,
we can define $\rho(e^Ue^V)=e^{\rho' X}e^{\rho' V}$ for given $U\in\so(2,\C)$, $V\in i\so(2,\C)$. We will abuse notation and
use $\rho_+$ both for the restriction  $\rho|_{\SO(2)}$ and its complexification. That is
\begin{equation}\label{eq rho+}
\rho_+:\SO(2,\C)\to H^\C.
\end{equation}
Now, by simple connectedness of $\SL(2,\C)$, $\rho'$ lifts to  
\begin{equation}\label{eq Adrho}
\Ad(\rho):\SL(2,\C)\to \Ad(G)^\C
 \end{equation} 
where $\Ad: G\to \Aut(\g^\C)$ is the adjoint representation and $\Ad(\rho)|_{\SL(2,\R)}=\Ad\circ\rho$.
Note that
\begin{equation}\label{eq Z(G)neq ker Ad}
\Ker(\Ad)=Z_{G}(\g)\supseteq Z(G). 
\end{equation}
\subsection{$\SL(2,\R)$-Higgs bundles}\label{sect SL2 Higgs pair}
Our basic case is $\SL(2,\R)$, which is a group of Hermitian type, as $\SL(2,\R)/\SO(2)$ is the hyperbolic plane.
 Let us start by analysing $\mc{M}^\alpha_d(\SL(2,\R))_L$ for an arbitrary line bundle $L$ of degree $d_L$.

An $L$-twisted $\SL(2,\R)$-Higgs bundle on a curve $X$ is a line bundle $F\to X$ together with morphisms
 $\beta:F^*\to F\otimes L$ and $\gamma:F\to F^*\otimes L.$ 
 \begin{lm}\label{lm moduli SL2}
The moduli space $\mc{M}_{L,d}^\alpha(\SL(2,\R))$:

1. is empty if $d>|d_L/2|$ or $d<\alpha$;

2. consists of all isomorphism classes of semistable $\SL(2,\R)$-Higgs pairs if degree $d>\alpha$;

3. is isomorphic to $\Pic^d(X)$ if $\alpha=d$.

4. Furthermore, if $i\alpha'\leq i\alpha\in\h$, there is an inclusion
\begin{equation}\label{eq inclusion pars} 
\mc{M}^\alpha(\SL(2,\R))\subseteq\mc{M}^{\alpha'}(\SL(2,\R)). 
\end{equation}
 \end{lm}
\begin{proof}
 To prove \textit{1.}, we first observe that the existence of sections  $\beta\in H^0(X,F^{2}\otimes L)$ and 
 $\gamma\in H^0(X, F^{-2}\otimes L)$ implies that $|d_L/2|\geq |\deg F|$ with equality if and only if $F^{\pm2}\cong L$. This
 accounts for the fists condition. 
 
 For the second, since $H^\C\cong\C^\times$ is abelian, for all $s\in i\h$ $P_s=H^\C$, and so the only reduction of 
 the structure group 
is the identity; moreover, the only antidominant character
is the identity (see \cite[Section 2.2]{GGMHitchinKobayashi}), and $B(\alpha,id)=\alpha||id||_B$; hence,
 a Higgs bundle is $\alpha$-semistable if and only if 
 \begin{equation}\label{eq ss SL}
\deg F\geq \alpha||id||_B.  
 \end{equation}
So after normalising $||id||_B=1$, we find that there will be no $\alpha$-semistable bundles for
$\alpha> d_L/2$, and for $\alpha\leq d_L/2$ we get bundles whose degree is at least 
$\lceil\alpha\rceil$ (where $\lceil\alpha\rceil$ is the lowest integer greater that real number 
$\alpha$) and at most $\left[d_L/2\right]$. 

Statements \textit{2.} and \textit{3.} follow from the above dicussion together with the fact that 
conditions for stability are limited to strictness of the inequality (\ref{eq ss SL}). Indeed, 
the Levi is again $H^\C$ itself. As for polystability, all stable bundles are polystable, so the 
only remaining case is when  (\ref{eq ss SL}) is an equality. 
Then, $(F,(\beta,\gamma))$ is polystable if and only if $\beta=\gamma=0$, as for $s\in\z\setminus 0$, 
$\m_s^0=\{0\}$.

Assertion \textit{4.} follows from the definitions. 
\end{proof}
Following \cite{Teich}, fix a holomorphic line bundle $L\to X$ of non-negative even degree, and consider
\begin{equation}\label{eq Higgs partida}
L^{1/2},\qquad\phi=\left(\begin{array}{cc}
                                            0&0\\
1&0
                                           \end{array}
\right)\in H^0(X,\mathrm{Hom}(L^{1/2}, L^{-1/2}\otimes L)). 
\end{equation}
By Lemma \ref{lm moduli SL2}, the $(L^{1/2},(0,1))$ is a 
 stable $L$-twisted $\SL(2,\R)$-Higgs bundle whenever $i\alpha\leq d_L/2$. 
Furthermore, if $i\alpha\geq 0$, we can map 
\begin{equation}\label{eq from alpha to zero}
\mc{M}^\alpha_L(\SL(2,\R))\to\mc{M}^0_L(\SL(2,\C)) 
\end{equation}
by (\ref{eq inclusion pars}), and the associated
$\SL(2,\C)$-Higgs bundle is stable for $\deg L\neq0$ (case in which the pair is strictly polystable whenever $\beta=\gamma=0$).

From now on we will assume that
\begin{equation}\label{hypos on alpha dl}
 d_L>0,\qquad i\alpha\leq d_L/2, \qquad 2|d_LF.
\end{equation}
We analyse the degree zero case in Remark \ref{rk degree 0}.

\begin{prop}\label{prop HKR sl2}
Given $L\to X$ and $i\alpha\in \R=\z(\so(2))$ satisfying (\ref{hypos on alpha dl}), then 
we have two well defined non gauge equivalent sections to the Hitchin
map
$$
h_{L}:\mc{M}^\alpha_L(\SL(2,\R))\to H^0(X,L^2)
$$
given by 
\begin{equation}\label{eq s+}
 s_+:\omega\mapsto (L^{1/2},(1,\omega))
\end{equation}
and
\begin{equation}\label{eq s-}
 s_-:\omega\mapsto (L^{-1/2},(\omega,1)).
\end{equation}
\end{prop}
\begin{proof}
Conditions  (\ref{hypos on alpha dl}) on $\alpha$ ensure polystability of the elements in the image of the section by Lemma
\ref{lm moduli SL2}. The same result ensures it is  enough to consider the 
case $d_L=2\alpha$. 

Non-equivalence of (\ref{eq s+}) and (\ref{eq s-}) follows from the fact that both sections
are conjugate via the complex gauge transformation $\Ad \left(\begin{array}{cc}
                                                          0&i\\
                                                          i&0
                                                         \end{array}
\right)$ of $\moduli{\SL(2,\C)}$ which is not in the image of $\SO(2,\C)$ under $\Ad:\SL(2,\C)\to \Aut(\SL(2,\C))$.
\end{proof}
\begin{rk}\label{rk regular sl2}
1. By Remark \ref{rk regular m}, $\lie{sym}_0(2,\C)^{reg}\subset\sl(2,\C)^{reg}$, and since 
$\sl(2,\C)^{reg}=\sll(2,\C)\setminus\{0\}$,
the Higgs field of every element in the image of the section is trivially everywhere regular 
 (cf. Definition \ref{def regular}).

 2. Note that for $i\alpha\geq 0$, the images of $s_+$ and $s_-$ are identified in $\mc{M}^0_L(\SL(2,\C))$ 
 under the morphism (\ref{eq from alpha to zero}).
 \end{rk}
\subsection{The induced basic $G$-Higgs bundle}\label{sect basic pair}
We are interested in a section of (\ref{defi stacky Hitchin map}) for arbitrary reductive groups $(G,H,\theta,B)$. It turns out that the $\SL(2,\R)$-Higgs bundle
$(L^{1/2},\phi)$ defined in (\ref{eq Higgs partida}) induces a $G$-Higgs bundle as follows.

Let $V$ be the principal bundle of frames of $L^{1/2}$. This has a 
structure group equal to $\C^\times$, which is isomorphic to $\SO(2,\C)$. 
Let $\rho_+$ be as in (\ref{eq rho}), and consider the corresponding associated bundle 
\begin{equation}\label{eq E}
E=V(H^\C).
\end{equation}
Letting $\rho'_-$ be as in (\ref{eq rho'+-}), we obtain a Higgs field
\begin{equation}\label{eq Phi base}
\Phi:=\rho'_-(\phi)\in H^0(X,E(\m^\C)\otimes L) 
\end{equation}
 where $\phi$ is as in (\ref{eq Higgs partida}) and $E(\m^\C)$ is the bundle associated to $E$ via the isotropy representation.

Since $E$ is extends a principal $\C^\times$-bundle, the structure of $E(\m^\C)\otimes L$ is determined by the action of 
$\ad(x)$, where $x$ is defined in (\ref{eq normal triple sl}). Furthermore, Proposition \ref{NTDS} implies that $e$ is a principal nilpotent element of $\m^\C$.

Note that $V(\lie{sym}_0(2,\C))\cong E(M_{\s}\cap {\m^\C})$ (where $M_{\s}$ is the module as defined in (\ref{eq splitting 2}) corresponding to the irreducible representation $\s^\C$) is the bundle of
symmetric endomorphisms of 
$L^{1/2}\oplus L^{-1/2}$, so we can identify it with
$L \oplus L^{-1}$, as  $L\cong\Hom(L^{-1/2}, L^{1/2})$. It follows that
\begin{equation}\label{decomp E(m)}
 E(M_k\cap{\m^\C})\cong \bigoplus_{i=0}^{\lfloor m_k-1/2\rfloor}L^{m_k-1-2i}\quad\textrm{if}\
 e_k\in{\m^\C}.
\end{equation}
In particular, $\Phi$ can be identified with the element
$f\in\m^\C$ considered  as a section of ${\m^\C}_{-1}\otimes L^{-1}\otimes L {\subset}E(\m^\C)\otimes L$,
where $\m^\C_\lambda$ is the eigenspace of $\ad(x)$ with eigenvalue $\lambda$. 
More generally 
\begin{equation}\label{decomp2}
e_k\in \m^\C_{m_k-1}\otimes L^{m_k-1}\otimes L^{-m_k+1} \subset E(M_k\cap{\m^\C}\otimes L^{-m_k+1}).
\end{equation}
since $m_k$ is odd whenever $e_k\in\m_\C$ by (\ref{eq mk}). 
\begin{defi}
 We call the pair $(E,\Phi)$ the \textbf{basic $G$-Higgs bundle}.
\end{defi}
In what follows, we study stability and smoothness properties of the basic $G$-Higgs bundle.
\begin{lm}
Let $(E,\Phi)$ be defined by (\ref{eq E}) and (\ref{eq Phi base}). Then $(E,\Phi)\in \mc{M}_L^0(G)$.
\end{lm}
\begin{proof}
By $\theta$-equivariance of (\ref{eq rho}),
we obtain a principal $H^\C$-bundle and a Higgs field taking values in ${\m^\C}$. 
Corollary \ref{prop HKR alpha moduli} 
gives the rest.
\end{proof}
Moreover, we have the following.
\begin{prop}\label{basic pair stable}
If $G$ is quasi-split, the pair $(E,\Phi)$ defined by (\ref{eq E}) and (\ref{eq Phi base}) is stable. Moreover, if $G$ is strongly reductive
and $Z(G)=Z_G(\g)$, then it is also simple.
\end{prop}
Before we prove Proposition \ref{basic pair stable}, we need a Lemma.
\begin{lm}\label{lm irred rho}
 Let $G$ be a strongly reductive quasi-split group (cf. Definition \ref{def qsplit gral gp}). Then, the map $\rho$ (\ref{eq rho}) satisfies that
 $Z_G(\Ima(\rho))=Z_G(\g)$.  
\end{lm}
\begin{proof}
Let $S=\Ima(\rho)$.
Under the hypothesis on the group, by Lemma \ref{lemma centraliser tds} \textit{2.} we have that $\lc_{\g}(\s)=0$. 
Thus, by definition, $\Ad(\rho)$ (\ref{eq Adrho}) is irreducible, 
so we have a three dimensional subgroup $\Ad(S)^\C=\Ad(\rho)(\SL(2,\C))<\Ad(G)^\C$. In particular
$Z_{\Ad(G)^\C}(\Ad(S)^\C)=1$.

Now, let $g\in Z_G(S)$. Since 
$\Ad(G)^\C$ is a group of matrices, we have that 
 $\Ad(S)^\C\subset\C\otimes\Ad(S)\subset \mbox{End}(\g^\C)$, so  
 $Z_{\Ad(G)^\C}(\Ad(S))\subset Z_{\Ad(G)^\C}(\Ad(S)^\C)$. This implies $g\in \Ker(\Ad)=Z_{G}(\g)$.
 \end{proof}
\begin{proof}{(\textit{Proposition \ref{basic pair stable}})}
Assume first $G$ is connected.

Note that  $(E,\Phi)$ is obtained by extending the stable $\SL(2,\R)$-Higgs pair $(V,\phi)$ via the morphism
$\rho$ defined in (\ref{eq rho}); by Proposition \ref{prop morphism moduli}, $(E,\Phi)$ is polystable.
By Theorem \ref{theorem Hitchin eqns for alpha moduli}, there exists  a solution
$h\in \Omega^0(X,V(\SO(2,\C)/\SO(2)))$ (resp. $h'\in \Omega^0(X,E(H^\C/H))$)
to the Hitchin equations (\ref{HK corresp}) for $\alpha=0$ and group  $\SL(2,\R)$ (resp. $G$). Let  $A$ (resp. $A'$) 
be the corresponding Chern connection for the given
holomorphic structure of $V$ (resp. $E$). From the proof of Proposition \ref{prop morphism moduli}, we may assume 
that $A'=\rho'(A)$. Locally, write
$$
A=d+M_A
$$
where $M_A\in \Omega^1(X,\so(2))$. Then $M_A$ is generically non zero, as otherwise $L^{1/2}$ would be flat, which
by assumptions \ref{hypos on alpha dl} is not the case. Now, an automorphism $g$ of $(A',\Phi)$ satisfies that for each $x\in X$ 
$$
\Ad_{g_x}\rho'(M_{A,x})=\rho'(M_{A,x})
$$
and 
$$\Ad_{g_x}\Phi_x=\Phi_x.
$$
Since for generic $x$, $M_{A,x}$ and $\phi_x$ generate  $\sl(2,\C)$, it follows that $g_x$ must centralise 
$\rho'(\sl(2,\C))=\s^\C$. In particular, $g_x$ centralises the subgroup 
$S=\rho(e^{\so(2)}e^{\lie{sym}_0(2,\R)})$. By Lemma \ref{lm irred rho}, 
 we have that  $g_x\in H\cap Z_G(\g)=Z_H(\g)$. Now, by closedness of 
 $Z_H(\g)$ inside of $H$, it follows that 
$g_x\in Z_H(\g)\cap \Ker(\iota)$ for arbitrary $x\in X$. 
Thus
$$
\lie{aut}(A',\Phi)\subseteq \z_\h(\g)=\z(\h)\cap\Ker(d\iota)\subset\lie{aut}(A',\Phi),
$$
so $(A',\Phi)$ is infinitesimally irreducible, and by Proposition \ref{prop HKR connections} $(E,\Phi)$ is stable.

When $Z(G)=Z_{G}(\g)$, then $g_x\in H\cap Z(G)=Z_H(G)=Z(H)\cap \Ker(\iota)$, and we also have 
$\Aut(A',\Phi)=Z(H)\cap \Ker(\iota)$. That is, $(A',\Phi)$ is irreducible and so 
$(E,\Phi)$ is stable and simple by Proposition \ref{prop HKR connections}.

As for disconnected groups, we note that the basic $G$-Higgs bundle $(E,\Phi)$ reduces its structure group
to $G^0$, the component of the identity in $G$. let $(E^0,\Phi^0)$ be the $G^0$-Higgs bundle whose extension is $(E,\Phi)$. By
the previous discussion, $(E^0,\Phi^0)$ is stable, and by Proposition \ref{prop morphism moduli}, $(E,\Phi)$ is polystable. 
Assume $\sigma\in\Gamma(X,E(H^\C/P_s))$ is a reduction of the structure group to a parabolic subgroup $P_s\subset H$  
violating the stability condition, namely, $\deg E(s,\sigma)>B(\alpha,s)$. We claim that $\sigma$ induces
a reduction $\sigma'\in\Gamma(X,E^0(H^\C/P_s\cap H^\C_0))$. Indeed, let $\sigma_\alpha(x)=(x,h_\alpha(x)P_s)$ be the expression of $\sigma$ on a trivialising 
neighbourhood $U_\alpha$. Then, on $U\alpha\cap U_\beta$, $\sigma_\beta(x)=(x,g_{\alpha\beta}(x)h_\alpha(x) P_s)$, where 
$g_{\alpha\beta}:U\alpha\cap U_\beta\to H^\C_0$ are the
transition functions of $E=E^0(H^\C)$. Then, we readily check that $\sigma'_\alpha(x)=(x,h_\alpha(x) P_s\cap H^\C_0)$ is well defined, as 
$h_\alpha^{-1}g_{\alpha\beta}h_\alpha\in H^\C_0$. So we obtain a 
principal $P_s\cap H^\C_0$-bundle $E_s$ such that $E_s(H^\C)=E$. Since $P_s\cap H^\C_0\subset H^\C_0$, also $E'=E_s(H^\C_0)$. 
Let $\sigma^0\in \Gamma(X,E^0(H^\C_0/P_s\cap H^\C_0))$
 be the corresponding reduction of the structure group. We need to check that $\deg E(s,\sigma)=\deg E^0(s,\sigma^0)$, which is easily seen using the definition
 of the degree given in  (\ref{rk alternative defi deg}).  This contradicts stability of $(E,\Phi)$. 
 
Concerning simplicity, Lemma \ref{lm irred rho} applies just as in the connected case.
\end{proof}
\begin{prop}\label{prop basic pair smooth} If $G$ is a strongly reductive Lie group and $(E,\phi)$
 is the basic $G$-Higgs bundle as defined in (\ref{eq E}) and (\ref{eq Phi base}), then 
$\H^2(C^\bullet(E,\Phi))=H^1(X,\z_\m\otimes L).$
\end{prop}
\begin{proof}
  First note that $S\plonge G$ factors through $S\plonge G_{ss}$. Let $(E_{ss},\phi_{ss})$ be the corresponding
 $G_{ss}$ bundle.
 Then 
 $$
 E(\h^\C)\cong E_{ss}(\h^\C_{ss})\oplus X\times \z_{\h},
 $$
where $\z_{\h}=\h^\C\cap\z(\g^\C)$, and $\h^\C_{ss}=(\h\cap\g_{ss})^\C$.
Likewise, $E(\m^\C)\cong E_{ss}(\m^\C_{ss})\oplus X\times \z_{\m}$. 

So the exact sequence (\ref{eq hypercohomology}) has the form:
$$
\begin{array}{l}
 \H^0(C^\bullet)\plonge H^0(E_{ss}(\h^\C_{ss}))\oplus H^0(\z_{\h})\to 
 H^0(E_{ss}(\m^\C_{ss})\otimes L)\oplus H^0(\z_{\m}\otimes L)\to \\
 \H^1(C^\bullet)
 \to H^1(E_{ss}(\h^\C_{ss}))\oplus H^1(\z_{\h})\to H^1(E_{ss}(\m^\C_{ss})\otimes L)\oplus H^1(\z_{\m}\otimes L)
 \to
 \\\H^2(C^\bullet)\to 0 
\end{array}
$$
Moreover, 
$$[\phi,E(\h^\C)]=[\phi_{ss},E_{ss}(\h^\C_{ss})]\subset E_{ss}(\m^\C_{ss})
$$
which implies $H^{i-1}(\z_{\m}\otimes L)\plonge \H^i(C^\bullet)$, and thus 
$$
\H^2(C^\bullet)=\H^2(\Ad(C^\bullet))\oplus H^{1}(X,\z_{\m}\otimes L).
$$
With the notation of Proposition \ref{prop simple and complex stable then smooth} we just need to prove that if 
$d_L\geq 2(g-1)$, then
$$
[\phi_{ss},H^1(X,E_{ss}(\h^\C_{ss}))]=H^1(X, E_{ss}(\m^\C_{ss})\otimes L).
$$
By (\ref{decomp E(m)}), we have: 
\begin{equation}\label{eq dec lb h}
E_{ss}(\h^\C_{ss}\cap M_k)=
\left\{\begin{array}{ll}
        L^{m_k-1}\oplus L^{m_k-3}\oplus\cdots\oplus L^{j_k}\oplus\cdots L^{-m_k+1}&\textrm{ if }e_k\in\h^\C,\\
          L^{m_k-2}\oplus L^{m_k-4}\oplus\cdots \oplus L^{l_k} \oplus\cdots  \oplus L^{-m_k+2}&\textrm{ if }e_k\in\m^\C,\\
          \end{array}
\right. 
\end{equation}
where $j_k=0$ if $m_k-1\equiv 0(2)$, $j_k=1$ if $m_k-1\equiv 1(2)$, and $l_k=j_k+1(2)$.
In a similar way, we see
\begin{equation}\label{eq dec lb m}
E_{ss}(\m^\C_{ss}\cap M_k)=
\left\{\begin{array}{ll}
                L^{m_k-1}\oplus L^{m_k-3}\oplus\cdots \oplus L^{j_k}\oplus\cdots L^{-m_k+1}&\textrm{ if }e_k\in\m^\C,\\
          L^{m_k-2}\oplus L^{m_k-4}\oplus\cdots \oplus L^{j_k+1(2)} \oplus\cdots  \oplus L^{-m_k+2}&\textrm{ if }e_k\in\h^\C.\\
  \end{array}
\right.
\end{equation}
Now, by definition 
\begin{equation}\label{eq action phi}
 \Phi: L^{j}\mapsto \left\{\begin{array}{ll}
			      L^{j-1}\otimes L=L^{j}&\ \textrm{ if }\ j>-m_k+1,\ m_k\neq1. \\
			      0&\ \textrm{ otherwise. }
                           \end{array}\right.
 \end{equation}
Hence
$$
[\Phi, H^1(E_{ss}(\h^\C_{ss})\cap M_k)]=
\left\{\begin{array}{ll}
        H^1(L^{m_k-1}\oplus\cdots L^{j_k}\oplus\cdots L^{-m_k+3})&\textrm{ if }e_k\in\h^\C,\\
          H^1(L^{m_k-2}\oplus L^{m_k-4}\oplus\cdots  \oplus L^{-m_k+2})&\textrm{ if }e_k\in\m^\C.
          \end{array}
          \right.
$$
We thus have
$$
\mathrm{Ker}(\ad(\Phi))=\left\{\begin{array}{ll}
H^1(L^{-m_k+1})& \textrm{ if }e_k\in\h^\C,\\
      0&\textrm{ if }e_k\in\m^\C,
          \end{array}
\right.
$$
which implies
$$
\mathrm{Coker}(\ad(\Phi))=\left\{\begin{array}{ll}
H^1(L^{m_k-1})& \textrm{ if }e_k\in\m^\C,\\
          0&\textrm{ if }e_k\in\h^\C.
          \end{array}
\right.
$$
If $G$ is quasi-split, given that $m_k>1$ 
(as we are only considering the semisimple part), we have
$h^1(L^{m_k-1})=h^0(L^{-m_k+1}K)=0$, and thus $\H^2(\Ad(C^\bullet))=0$, which proves the statement. 

If $G$ is not quasi-split, the only thing that is different is the fact that the trivial representation
 $\lc_{\g^\C}(\s^\C)$ has $m_1=1$--and positive
multiplicity  $n_1$ by Lemma \ref{lemma centraliser tds}. Therefore, 
$H^1(X,L\otimes \lc_{\m^\C}(\s^\C))\plonge \H^2(\Ad(C^\bullet))$. But $\lc_{\m^\C}(\s^\C)=0$ by \textit{1.}
 in Lemma \ref{lemma centraliser tds}.
\end{proof}
\subsection{Construction of the section.}\label{section non Hermitian}
We have now all the ingredients yielding to the Hitchin--Kostant--Rallis section. Let us recall some of 
the notation before 
stating the theorem. 
Let $\rho':\sll(2,\R)\to \g$ be the homomorphism defining the principal normal TDS $\s\subset\g$ (see (\ref{eq rho'})). Consider
the group $Q$, satisfying $(\Ad(G)^\C)^\theta=Q\Ad(H^\C)$ (see Proposition \ref{lemma Gtheta} for other characterizations).
It is a finite group whose cardinality we denote by $N$.
\begin{thm}\label{thm HKR Hermitian}
Let $(G,H,\theta, B)$ be a strongly reductive Lie group, and let  $(\TIG_0,\TIH_0,\tth,\widehat{B})$ be its maximal connected split subgroup. 
Let $L\to X$ be a line bundle with degree $d_L\geq 2g-2$.
Let  $\alpha\in i\z(\so(2))$ be such that $\rho'(\alpha)\in\z(\h)$. 
 Then, the choice of a square root of $L$ determines $N$ non equivalent sections of the map
$$
h_L: \mc{M}_L^{\rho'(\alpha)}(G)\to B_{L}(G).
$$
Each such section $s_G$ satisfies
\benum
\item[1.] If $G$ is quasi-split, $s_G(B_L(G))$ is contained in the stable locus of $\mc{M}_L^{\rho'(\alpha)}(G)$, and in the smooth locus
if $Z(G)=Z_G(\g)$ and $d_L\geq 2g-2$.
\item[2.] If $G$ is not quasi-split, the image of the section is contained in the strictly polystable locus.
\item[3.] For arbitrary groups, the Higgs field is everywhere regular.
\item[4.] If $\rho'(\alpha)\in i\z\left({\tih}\right)$, the section factors through $\mc{M}_L^{\rho'(\alpha)}(\TIG_0)$. This is in particular
the case if $\alpha=0$.
\item[5.] If $G_{split}<G^\C$ is the split real form satisfying $\z(\g^\C)\cap i\lu\subset\m$, $K=L$ and $\alpha=0$, 
$s_G$ is the factorization of the Hitchin section through $\moduli{G_{split}}$.
\eenum
\end{thm}
\begin{proof}
The proof consists of three parts: first, we construct a section into $\mc{M}_L^0(G)$ for quasi-split real forms. This in particular includes the split group case. 
Secondly,  using the maximal split subgroup, we are able to extend the section to $\mc{M}_L^0(G)$ for all groups. A third part deals with stability for
other values of the parameter.

\textit{1.} Quasi-split groups. To start with, we note that the deformation argument used by Hitchin in \cite{Teich} adapts in the case of quasi-split groups: 
for each $\ol{\gamma}\in \oplus_{i=1}^a H^0(X, L^{m_i})$, define the field
$$
\Phi_{\ol{\gamma}}=f+\sum_{i=1}^a\gamma_ie_i,
$$
where $e_i, i=1,\dots, a$ generate $\lc_{\m^\C}(e)$ and $e_1=e$. Note that this is a well defined section of $E(\m^\C\otimes L)$ 
by (\ref{decomp2}). 

Now, any family of Higgs bundles containing a stable point automatically
contains a dense open set of stable points. 
  In particular, by Proposition \ref{basic pair stable}, $(E,\Phi)$ is $0$-stable, so for sufficiently small
$\gamma_i$'s, we have that
$(E,\Phi_\gamma)$ is $0$-stable. Namely, the basic solution $(E,\Phi)$ can be deformed to a section from an open neighbourhood 
of $0\in B_{G,L}$ into $\mc{M}_L^0(G)^{stable}$. Next, note that exponentiation of $x$ produces an automorphism of 
$E$ and $E(\m^\C)\otimes L$ sending $\Phi_\gamma$ to
$$
\Psi_\gamma=\mu^{-1}f_1+\gamma_1\mu^{m_1}e_1+\dots+\gamma_a\mu^{m_a}e_a.
$$
That is, the automorphism transforms the family corresponding to $(E,\Phi)$ into the family corresponding to $(E,\mu^{-1}\Phi)$. The same
arguments apply to the latter bundle, so that for sufficiently small $\mu^{m_i}\gamma_i$, $\Psi_\gamma$ is stable.  
So every element of the family can be identified to one with small $\gamma_i$, as $m_i>0$ by (\ref{eq mk}).
Since gauge transformations preserve stability, we are done. Furthermore, by Propositions \ref{prop basic pair smooth} and \ref{prop simple and complex stable then smooth},
if $Z(G)=Z_G(\g)$, the points in the image of the section are smooth.

For moduli spaces depending on an arbitrary parameter, 
we note that the hypotheses on the parameter and Equation (\ref{eq inclusion pars})
imply that for $\alpha\neq 0$, $\mc{M}^0_L(\SL(2,\R))\subset \mc{M}^\alpha_L(\SL(2,\R))$, and stability is preserved.
Since $(E,\Phi)$ is the extended $G$-Higgs bundle of $(V,\phi)$ via $\rho$ (cf. Definition \ref{def extended gp} and 
Equation (\ref{eq rho})),
polystability is automatic for any $\rho'(\alpha)$ such that $(V,\phi)$ is $\alpha'$ stable, where
$\rho'=d\rho$ is as in (\ref{eq rho'}). Hence, we
have $(E,\Phi)\in\mc{M}^{\rho'(\alpha)}_L(G)$ for all $i\alpha\leq 0$. Namely, for 
all $s\in i\h$ and all $\sigma\in \Omega^0(X,E(H^\C/P_s))$ satisfying conditions in Definition 
\ref{defi polyst}, we have 
$$
\deg(E(s,\sigma))\geq B(\rho'(\alpha),s).
$$
Now, $B(\rho'(\alpha),s)=i\alpha B(\rho'(i),s)$, which given that $B$ is definite positive on $i\h$, means
that $i\alpha B(\rho'(i),s)\leq 0$. But $0$-stability of $(E,\Phi)$ implies 
$$\deg(E(s,\sigma))>0\geq B(\rho'(\alpha),s),$$
whence stability follows.

\textit{2.} Non quasi-split groups. By \textit{1.}, the elements in the image of the 
Hitchin--Kostant--Rallis section for the split subgroup are 
$0$-stable, as split groups are quasi-split.
 So Corollary \ref{prop HKR alpha moduli} and Theorem \ref{theorem KR section} imply the existence of 
  a  $0$-polystable section for any group. Strict polystability follows from Proposition 2.14 in 
  \cite{GGMHitchinKobayashi} and Corollary \ref{cor TDSgps}. 
  
Points \textit{3.}, and \textit{4.} follow  by construction. For \textit{5.}, we just note that from Definition \ref{defi qs},
the principal normal TDS is in particular a TDS in the usual sense \cite{Kos}, so
the construction matches Hitchin's as long as the rings of invariants $\C[\m^\C]^{H^\C}$ and $\C[\g^\C]^{G^\C}$
match. This is guaranteed for a split subgroup as in the statement (see Remark \ref{rk invariant polynomials}). Such a
split form always exists as the maximal connected split subgroup for a choice of Cartan data
$(G^\C, U, \tau, B))$ satisfies the conditions.

Concerning the number of sections, the construction depends on a choice of principal normal TDS. 
By Theorem 6 in \cite{KR71}, all such are $(\Ad(H)_\theta)^\C$ conjugate, and by Proposition \ref{lemma Gtheta},
the number of non conjugate $H^\C$-orbits is determined by $\#Q$.

Finally, regularity follows from Theorem \ref{theorem KR section}.
\end{proof}
\begin{rk}
 The Hitchin--Kostant--Rallis section is a section in the sense that for a given choice of homogeneous generators
 $\{p_1,\dots,p_a\}\subset\C[\m^\C]^{H^\C}$, the map 
 $$(p_1,\dots,p_a)\circ s_G:B_L(G)\to B_L(G)$$
 is the identity. This follows from Theorem 7 in \cite{Kos}.
\end{rk}

\begin{rk}[Degree zero twisting]\label{rk degree 0}
 When $d_L=0$, there are two cases to consider:

1. Trivial bundle: if $L=\mc{O}_X$, the existence and construction of the section amounts to the results  in \cite{KR71}.
Indeed, the Hitchin base $B_{G,\mc{O}}=H^0(X,\mc{O}\otimes \la^\C/\bweyl)\cong \la^\C/\bweyl$. On the other hand, by (\ref{decomp2}),
$e_i\in H^{0}(X,E(\m^\C))$. Thus everything follows from \cite{KR71}, modulo the choice of a square root of $\mc{O}$, i.e., an order
two point of $Jac(X)$.

2. Non-trivial bundle: this is a trivial case, as $B_{G,L}=0$.
 \end{rk}
 Propositions \ref{prop simple and complex stable then smooth},  \ref{basic pair stable},
 \ref{prop basic pair smooth}
 and Theorem \ref{thm HKR Hermitian} yield the following.
\begin{cor}\label{cor basic pair smooth}
 Let $G$ be a strongly reductive quasi-split group. 
 Assume that $Z(G)=Z_G(\g)$ and $d_L\geq2g-2$. Then
 the $G$-Higgs bundle $(E,\Phi)$ defined by (\ref{eq E}) and (\ref{eq Phi base}) is a smooth 
 point of $\mc{M}^{0}_L(G)$.
 In particular, for any quasi-split group, 
 the associated $\Ad(G)$-Higgs bundle $(E([H]^\C),[\phi]))$ is a smooth point of $\mc{M}_L(\Ad(G))$,
 where $[H]=H/Z(G)\cap H$, $[\phi]=\phi/\z_\m$.
\end{cor}
\begin{prop}\label{prop dimension moduli}
Let $\alpha\in i\z(\h)$ be such that the basic $G$-Higgs bundle $(E,\Phi)$ is $\alpha$-stable. 
 Let  $a=\dim\la^\C$, $b=\dim\lc_{\h^\C}(\la^\C)$ and $c=\dim\lc_{\h^\C}(\s^\C)$, where $\s^\C$ is
 a normal principal TDS. Let $G$ 
be a strongly reductive Lie group. Then the expected dimension of the irreducible component of the moduli space
$\mc{M}^\alpha_L(G)$ containing the image of the HKR section is
\begin{equation}\label{eq dim moduli}
\mathrm{exp.}\dim\left(\mc{M}^\alpha_L(G)\right)=c+h^1(\z_\m\otimes L)+\frac{d_L}{2}\dim \g^\C+ (a-b)\left(\frac{d_L}{2}-g+1\right).
 \end{equation}
 In particular, if $G$ is quasi-split and $\deg L\geq 2g-2$, the expected dimension is the actual dimension of the moduli space.
\end{prop}
\begin{proof}
Letting $\mathbf{h}^i=\dim \H^i(X, C^\bullet)$, the expected dimension is $\mathbf{h}^1$. 
From the long exact sequence (\ref{eq hypercohomology}), we have 
\begin{equation}\label{eq exp dim}
\mathbf{h}^1=\chi(E(\m^\C)\otimes L)-\chi(E(\h^\C))+\mathbf{h}^0+\mathbf{h}^2. 
\end{equation}
By (\ref{eq dec lb h}) 
$$
\chi(E(\h^\C))=\sum_{e_k\in\h^\C}\sum_{0\leq j\leq m_k-2}
\chi\left(L^{m_k-2j-1}\right)+\sum_{e_k\in\m^\C}\sum_{1\leq j\leq m_k-1}\chi\left(L^{m_k-2j}\right).
$$
Similarly
$$
\chi(E(\m^\C)\otimes L)=\sum_{e_k\in\h^\C}\sum_{0\leq i\leq {m_k-2}}
\chi\left(L^{m_k-2i-1}\right)+\sum_{e_k\in\m^\C}\sum_{0\leq i\leq {m_k-1}}\chi\left(L^{m_k-2i}\right).
$$
Also, by Proposition \ref{prop basic pair smooth} $\mathbf{h}^2=h^1(\z_\m\otimes L)$. 
On the other hand, we easily deduce from (\ref{eq action phi}) 
that $\mathbf{h}^0=\dim\lc_{\h^\C}(\s^\C)$. Substituting it all 
 into (\ref{eq exp dim}), and applying Riemann--Roch yields 
 $$
 \mathbf{h}^1=c+h^1(\z_\m\otimes L)+\sum_{e_k\in\m^\C}\left(d_L m_k-g+1\right)+
 \sum_{e_k\in\h^\C}\left((m_k-1)d_L+g-1\right).
$$
Using (\ref{eq mk}), we obtain 
\begin{eqnarray} \nonumber
 \mathbf{h}^1&=&c+h^1(\z_\m\otimes L)+
 \sum_{e_k\in\m^\C}\left(\frac{d_L}{2} (2m_k-1)+\frac{d_L}{2}-g+1\right)\\\nonumber
&+& \sum_{e_k\in\h^\C}\left((2m_k-1)\frac{d_L}{2}-\frac{d_L}{2}+g-1\right)\\ \nonumber
 &=&c+h^1(\z_\m\otimes L)+\frac{d_L}{2}\dim \g^\C+ (a-b)\left(\frac{d_L}{2}-g+1\right).
 \end{eqnarray}
This yields the result about the expected dimension.

The last assertion follows fom Corollary \ref{cor basic pair smooth}.
\end{proof}
\begin{rk}
 We can give the expected dimension of the moduli space $\mc{M}_L(G)$. Indeed, let $(E,\phi)$ be a smooth point. This implies
 $\mathbf{h}^0=\dim\z_\h$, $\mathbf{h}^2=h^1(\z_\m\otimes L)$. On the other hand, by Hirzebruch--Riemann--Roch, we have
 $$
 \chi(E(\h^\C))=c_1(E(\h^\C))+\dim(\h^\C)(1-g).
 $$
 Now, $c_1(\h^\C)=0$, as $\h^\C\cong_{B_\C} (\h^\C)^\wedge$. Likewise,
 $
 \chi(E(\m^\C)\otimes L)=c_1(E(\m^\C)\otimes L)+\dim(\m^\C)(1-g)$. Since $c_1(E(\h^\C))+c_1(E(\m^\C))=0=c_1(E(\g^\C))$, then 
 $\chi(E(\m^\C)\otimes L)=\dim(\m^\C)(d_L+1-g)$. So altogether
 $$
 \mathbf{h}^1=\dim\z_\h+h^1(\z_\m\otimes L)+\dim(\m^\C)(d_L+1-g)-\dim(\h^\C)(1-g).
 $$
 Expressing $\dim(\m^\C)(d_L+1-g)=(\frac{d_L}{2}+\frac{d_L}{2}+1-g)$ and $\dim(\h^\C)(1-g)=\dim(\h^\C)(-\frac{d_L}{2}\frac{d_L}{2}+1-g)$
 we have
 $$
\mathbf{h}^1=\dim\z_\h+h^1(\z_\m\otimes L)+\dim(\g^\C)\frac{d_L}{2}+(\dim\m^\C-\dim\h^\C)(\frac{d_L}{2}+1-g).
 $$
 By Proposition 5 in \cite{KR71}, $\dim\m^\C-\dim\h^\C=a-b$, where $a$ and $b$ are as in Proposition \ref{prop dimension moduli}.
 So in the particular example of quasi-split groups, the expected dimension of the component containing the image of the HKR section 
 is the expected dimension of the moduli space (as should be by smoothness). On the other hand, we find that for non quasi-split groups, the singularities of the image of the HKR section 
 are of orbifold origin. 
\end{rk}

\begin{cor}
Assume $L=K$. Then 
 
 1.  If $G$ is the real group underlying a complex reductive Lie group $U^\C$, then
 $$\mathrm{exp}.\dim \mc{M}_K(U^\C)=\dim \mc{M}_K(U^\C)=2\left(\dim (U^\C)(g-1)+\dim Z(U^\C)\right).$$

 2. If $G<G^\C$ is a real form of a complex reductive Lie group, then 
 $$\mathrm{exp}.\dim \mc{M}_K(G)=\frac{1}{2}\dim \mc{M}_K(U^\C)+\dim\lc_{\h^\C_{ss}}(\s^\C).$$ 
 Therefore, it matches the expected dimension of the moduli space if and only if $G$ is quasi-split.
 \end{cor}
 \begin{proof} 
 To see \textit{1.} first remark that $G<G^\C\times G^\C$ is quasi-split, and Proposition \ref{prop dimension moduli} the expected dimension at any element the HKR section is the actual
 dimension. So under the given hypotheses, $c=\dim Z(U^\C)=\dim\z_\m$, where the first equality follows from Lemma \ref{lemma centraliser tds}. 

 For \textit{2.}, we note that $c=\dim\z_\h+\dim\lc_{\h^\C_{ss}}(\s^\C)$, and that $\z(\g^\C)=\z_\h\oplus\z_\m$. 
 \end{proof}
\begin{prop}\label{prop component HKR}
Let $G$ be a quasi-split Lie group. Let $L\to X$ be a holomorphic line bundle such that 
$d_L:=\deg L\geq2g-2$. Then
the HKR section covers a connected component of the moduli space of $L$-twisted
Higgs bundles if and only if $G$ is split.

Under the above possible hypothesis, our construction yields $N\cdot2^{2g}$ Hitchin components, where $N$ is defined as in 
Theorem \ref{thm HKR Hermitian}.
\end{prop}
\begin{proof}
Since $G$ is quasi-split, by Theorem \ref{thm HKR Hermitian}, the image of the 
section defines a closed subspace contained in the smooth locus of the moduli space. Moreover, by construction,
the image of the section is open whenever 
\begin{equation}\label{eq ineq}
\dim B_L(G)\geq\dim \mc{M}_L(G), 
\end{equation}
as it is an affine subset of a manifold of the right dimension 
(cf. Theorem 3.4 in \cite{GGMHitchinKobayashi}), and it is open as it is a family of stable elements
parameterised by the Hitchin base. 
 So we apply Propositions \ref{prop dimension moduli} and \ref{prop dim BL}, noting that by quasi-splitness $c=\dim\z_\h$ (cf. Lemma \ref{lemma centraliser tds}\textit{3.}). 
Comparing dimensions, we obtain that (\ref{eq ineq}) holds if and only if
\begin{equation}\label{eq ineq ss}
\frac{d_L}{2}\left(\dim\tg^\C-\dim\g^\C\right)+b\left(\frac{d_L}{2}-g+1 \right)-\dim\z_\h\geq 0.  
\end{equation}
Note that $\dim\tg^\C-\dim\g^\C=-b-2\cdot(\# \Delta-\# \widehat{\Lambda}(\la))$, where 
$\Delta$ denotes the set of roots and $\widehat{\Lambda}(\la))$ the set of reduced restricted roots defined in 
(\ref{eq TLambda}). It follows that
(\ref{eq ineq ss}) is equivalent to 
$$
-{d_L}\left(\# \Delta-\# \widehat{\Lambda}(\la)\right)-b(g-1)-\dim\z_\h\geq 0,
$$
which is possible if and only if each of the (negative) terms vanishes. But this implies in
particular that $b=0$, so that $\g$ must be split, and consequently all other terms vanish. 
%

As for the statement concerning the number of sections, the factor $N$ is the one appearing in Theorem \ref{thm HKR Hermitian}.
The remaining choices correspond to taking a square root 
of $L$. We could have also chosen such for $L^{-1}$, but the sections obtained this way are identified with 
the ones resulting from using $L^{1/2}$ by the action of $\Ad(H_\theta)$. A way to see this is by considering
the section into $\moduli{\Ad(G)}$ and complexifying them. Remark \ref{rk regular sl2}, together with Proposition 
\ref{lemma Gtheta} and Lemma \ref{lemma normaliser} allow to conclude. The same reasoning implies
inequivalence of the $N\cdot 2^{2g}$ sections.
\end{proof}
\subsection{Regularity}\label{section regularity}
Regularity of the Higgs field is directly related to smoothness of points in the Hitchin fiber. 
This essentially goes back to Kostant's \cite{Kos}, as it is proved by Biswas and 
Ramanan for complex Lie groups(\cite{BisRam},Theorem 5.9). Their proof applies to the real case, so we have:
\begin{prop}\label{Biswas Ramanan adapts}
Let $\omega\in B_L(G)$, and assume $(E,\phi)\in\moduli{G}\cap h_G^{-1}(\omega)^{smooth}$ is a smooth point of $h_G^{-1}(\omega)$, then $\phi(x)\in\mr$ for 
all $x\in X$.
\end{prop}
\begin{proof}
 Fixing $x\in X$, we have
 that $\mathrm{ev}_x\circ h(E,\phi)=\chi\phi_x $, where $\chi:\m^\C\to\lie{a}^\C/W(\la)$ is the Chevalley
 map. At a smooth point of the fiber, $dh$ is surjective, and since $\mathrm{ev}_x$ is surjective
 too, it follows that $d(\chi\circ \mathrm{ev}_x)$ is itself surjective. Since 
 $d\mathrm{ev}_x:H^0(X,E(\m^\C\otimes K))\to \m^\C\otimes K_x$ is surjective, and is itself evaluation
 at $x$, this implies that $d_{\phi_x}\chi$ is surjective.
 But Kostant--Rallis' work \cite{KR71}, citing Kostant \cite{Kos}, implies this happens if and only if $\phi_x$ is regular.
\end{proof}
 \section{Topological type of the elements in the image of the HKR section}\label{section topo type}
 Recall from Section \ref{subs topo type Higgs} that to a Higgs bundle we can assign a topological invariant.
 We now come to the problem of determining the topological invariant of the component of the moduli space
 where the image of the HKR section falls in.

We remark that given a $G$-Higgs bundle $(E,\Phi)$, the topological type depends uniquely on $E$,  
so it is enough to compute the 
invariants for the principal bundle defined in (\ref{eq E}). Moreover,
by construction of the section, the type of $E$ is independent of the value of $\alpha=0$, as it is the principal 
bundle associated to some fixed $\SO(2,\C)$ bundle. 
\begin{prop}\label{prop top invariant non hermitian}
Let $\mc{G}$ be a connected simply connected simple algebraic group over $\R$, $\widehat{\mc{G}}$ its maximal split subgroup, 
and  $G:=\mc{G}(\R)$, $\TIG=\wh{\mc{G}}(\R)$ be the groups of their respective real points. Assume $G$ is not of
Hermitian type. Let $E$ and $\wh{E}$ be the principal bundles defined by (\ref{eq E}) for the groups $G$ and $\TIG$ respectively.
 Then either $\pi_1(G)=1$ (and then $d(E)=0$) or $d(E)=d(\wh{E})\mod 2$.
\end{prop}
\begin{proof}
We observe that by simplicity, $\pi_1(G)=1$, $\Z/2\Z$, or $\Z$, but the last option corresponds to Hermitian groups, so either $\pi_1(G)=1$ or $\pi_1(G)=\Z/2\Z$. Likewise,
$\pi_1(\TIG)$ is either $\Z/2\Z$ or $\Z$, as on the one hand $\TIG$ is simple by construction, and on the other, split groups are never simply connected
by Corollary 1.2 in \cite{Adams}. We will prove that the map $i_*:\pi_1(\TIG)\projects \pi_1(G)$ (induced by the inclusion $i$) is surjective, 
which implies the statement. Indeed, the only homomorphisms 
$\Z\to\Z/2\Z$ are constant or reduction modulo two, and similarly for $\Z/2\Z\to\Z/2\Z$.

 By Proposition 2.10 in \cite{Adams}, $\pi_1(G)=\Z\langle\Delta^{\vee}_{nc}(\g^\C,\ld_c^\C)\rangle/\Delta^{\vee}_{nc}(\g^\C,\ld_c^\C)\cap \Delta^{\vee}(\h^\C,\ld_c^\C\cap\h^\C)$, 
where $\ld_c\subset\g$ is a maximally compact Cartan subalgebra.
Let $\beta^\vee\in\Delta^\vee(\g^\C,\ld_c^\C)$ be a generator. 
Using the Cayley transform, we may identify it with a real coroot $\alpha^{\vee} \in\Delta^{\vee}(\g^\C,\ld^\C)$,
where $\ld\subset\g$ is a maximally split Cartan subalgebra. Since $\alpha\in\Delta_r$, 
$(\alpha|_\la)^{\vee}=\alpha^{\vee}$, so 
$i_*(\alpha|_\la)^{\vee}=\alpha^{\vee}$.
In particular, the image of $(\alpha|_\la)^{\vee}$ in $\pi_1(\TIG)$ is non trivial, and so $i_*$ is surjective.
\end{proof}
\begin{prop}\label{prop toledo}
Let $G$ be a connected simple real Lie group of Hermitian type. Then,
 the topological invariant $d(E)$ corresponding to the Hitchin--Kostant--Rallis section for the moduli space of Higgs bundles
 is maximal if $G$ is of tube type, and zero if it is of non-tube type.  
\end{prop}
\begin{proof}
First of all, by Proposition \ref{big prop toledo} \textit{2.}, maximality or vanishing are equivalent whether we consider $T$ or $d$,
so we will use them indistincly. As discussed above, it is enough to determine the degree of $E$.

Let $G$ be of tube type. Then, by Theorem \ref{thm HKR Hermitian}, the Higgs field is
regular at every point, and thus
Proposition \ref{big prop toledo} \textit{1.} implies maximality of the Toledo invariant. 

Now, if $G$ is of non-tube type, $\TIG_0$ is not of Hermitian type unless
its split rank is one or two. Indeed, the simple Lie algebras of  
Hermitian non-tube type are $\su(p,q)$ with $p\neq q$, $\so^*(4p+2)$ and $e_{6(-14)}$. 
The maximal split subalgebra of all of them is 
$\so(rk_\R(\g),rk_\R(\g)+1)$, which is not of Hermitian type whenever the real rank is 
higher than two (see Table \ref{tab}). 

Now, the basic $G$-Higgs bundle  $E$ is associated to the basic $\TIG_0$-Higgs 
bundle  by extension of the structure group. 
By Corollary \ref{prop HKR alpha moduli}, if $G$ has rank at least three, the topological type is zero, as it is the image of
a torsion group inside $\pi_1(G)=\Z$. 

As for ranks $1$ and $2$, for Lie groups with Lie algebra $\su(n,1)$ with $n>1$, $\su(n,2)$ with $n>2$ and $\lie{e}_{6(-14)}$, as well as simply 
connected Lie groups, the result follows from  Corollary \ref{maximal split as finite cover}.

The only remaining groups are $\so^*(6)$ and $\so^*(10)$, of ranks $1$ and $2$ respectively, which are covered by 
Lemma \ref{lm toston} below. 
\end{proof}
\begin{lm}\label{lm toston}
Let $\g$ be the Lie algebra $\so^*(6)$ or $\so^*(10)$. Then, if $\{e,f,x\}$ is a normal triple generating a principal normal TDS, then
the semisimple element $x$ decomposes as $x=\left(\begin{array}{c|c}
   A&0\\\hline
   0&B
  \end{array}\right)
$ where $tr(A)=tr(B)=0$.
\end{lm}
\begin{proof}
 Following \cite{K}, we realise the Lie algebra $\so^*(2n)$ as the subalgebra of $\sl(2n,\C)$ whose elements satisfy:
 $$
 -\Ad(I_{n,n})^t\ol{A}=A,\quad -\Ad(J_{n,n})^tA=A,
 $$
 where
 $$
 I_{n,n}=\left(\begin{array}{cc}
                I_n&0\\0&-I_n
               \end{array}
\right), \quad 
 J_{n,n}=\left(\begin{array}{cc}
                0&I_n\\I_n&0
               \end{array}
\right).
 $$
 We have also:
 $$
 \h^\C=\left\{\left(\begin{array}{cc}A&0\\
 0&-^tA\end{array}\right)\ :\ A\in \gl(n,\C)\right\},
 $$
 $$
 \m^\C=\left\{\left(\begin{array}{cc}0&B\\
 C&0\end{array}\right)\ :\ B,C\in \gl(n,\C), B+^tB=0=C+^tC\right\}.
 $$
 In particular, 
 \begin{equation}\label{action th so*}
  \theta\left(\begin{array}{cc}A&B\\
 C&-^tA\end{array}\right)=\left(\begin{array}{cc}A&-B\\
 -C&-^tA\end{array}\right)
 \end{equation}
Now, with the same notation of Theorem \ref{theorem KR section}, we can easily compute generators $e_c$, $f_c$, $w$ for a principal normal TDS. From these, a normal triple is given by:
$e=\frac{-e_c+f_c+w}{2}$, $e=\frac{e_c-f_c+w}{2}$, $x={e_c+f_c}$. So to have $x$, it is enough to compute $e_c$, as $f_c=\theta e_c$.

We start by $\so^*(6)$. In this case $e_c$ is a multiple of an eigenvector $y\in\so^*(6)$ for 
$$
w=\left(
\begin{array}{ccc|ccc}
 \phantom{-}0&\phantom{-}0&\phantom{-}0&\phantom{-}0&\phantom{-}0&\phantom{-}1\\
 \phantom{-}0&\phantom{-}0&\phantom{-}0&\phantom{-}0&\phantom{-}0&\phantom{-}0\\
 \phantom{-}0&\phantom{-}0&\phantom{-}0&-1&\phantom{-}0&\phantom{-}0\\\hline
 \phantom{-}0&\phantom{-}0&-1&\phantom{-}0&\phantom{-}0&\phantom{-}0\\
 \phantom{-}0&\phantom{-}0&\phantom{-}0&\phantom{-}0&\phantom{-}0&\phantom{-}0\\
 \phantom{-}1&\phantom{-}0&\phantom{-}0&\phantom{-}0&\phantom{-}0&\phantom{-}0
\end{array}
\right).
$$
By setting $[w,y]=y$, we obtain 
$$
y=\left(
\begin{array}{ccc|ccc}
 \phantom{-}0&\phantom{-}1&\phantom{-}0&\phantom{-}0&\phantom{-}0&\phantom{-}0\\
 -1&\phantom{-}0&\phantom{-}0&\phantom{-}0&\phantom{-}0&\phantom{-}1\\
 \phantom{-}0&\phantom{-}0&\phantom{-}0&\phantom{-}0&-1&\phantom{-}0\\\hline
 \phantom{-}0&\phantom{-}0&\phantom{-}0&\phantom{-}0&\phantom{-}1&\phantom{-}0\\
 \phantom{-}0&\phantom{-}0&-1&-1&\phantom{-}0&\phantom{-}0\\
 \phantom{-}1&\phantom{-}0&\phantom{-}0&\phantom{-}0&\phantom{-}0&-1
\end{array}
\right).
$$
Then, since both diagonal blocks of $y$ have zero trace, so do the ones of $e_c=\lambda y$, and $f_c$ by 
(\ref{action th so*}), hence the same holds for $x$.

As for $\so^*(10)$, an element of the maximal anisotropic Cartan subalgebra has the form:
$$
w=\left(
\begin{array}{ccccc|ccccc}
 \phantom{-}0&\phantom{-}0&\phantom{-}0&\phantom{-}0&\phantom{-}0&\phantom{-}0&\phantom{-}0&\phantom{-}0&\phantom{-}0&a\\
 \phantom{-}0&\phantom{-}0&\phantom{-}0&\phantom{-}0&\phantom{-}0&\phantom{-}0&\phantom{-}0&\phantom{-}0&b&\phantom{-}0\\
 \phantom{-}0&\phantom{-}0&\phantom{-}0&\phantom{-}0&\phantom{-}0&\phantom{-}0&\phantom{-}0&\phantom{-}0&\phantom{-}0&\phantom{-}0\\
 \phantom{-}0&\phantom{-}0&\phantom{-}0&\phantom{-}0&\phantom{-}0&\phantom{-}0&-b&\phantom{-}0&\phantom{-}0&\phantom{-}0\\
 \phantom{-}0&\phantom{-}0&\phantom{-}0&\phantom{-}0&\phantom{-}0&-a&\phantom{-}0&\phantom{-}0&\phantom{-}0&\phantom{-}0\\\hline
 \phantom{-}0&\phantom{-}0&\phantom{-}0&\phantom{-}0&-a&\phantom{-}0&\phantom{-}0&\phantom{-}0&\phantom{-}0&\phantom{-}0\\
 \phantom{-}0&\phantom{-}0&\phantom{-}0&-b&\phantom{-}0&\phantom{-}0&\phantom{-}0&\phantom{-}0&\phantom{-}0&\phantom{-}0\\
 \phantom{-}0&\phantom{-}0&\phantom{-}0&\phantom{-}0&\phantom{-}0&\phantom{-}0&\phantom{-}0&\phantom{-}0&\phantom{-}0&\phantom{-}0\\
 \phantom{-}0&b&\phantom{-}0&\phantom{-}0&\phantom{-}0&\phantom{-}0&\phantom{-}0&\phantom{-}0&\phantom{-}0&\phantom{-}0\\
 a&\phantom{-}0&\phantom{-}0&\phantom{-}0&\phantom{-}0&\phantom{-}0&\phantom{-}0&\phantom{-}0&\phantom{-}0&\phantom{-}0\\
\end{array}
\right)=a h_1+b h_2.
$$
We compute $y_i$ to be the an eigenmatrix of $y_i$ within $\so^*(10)$. We see
$$
y_1=\left(
\begin{array}{ccccc|ccccc}
  \phantom{-}0&\phantom{-}1&\phantom{-}1&\phantom{-}1&\phantom{-}0&\phantom{-}0&-1&-1&-1&\phantom{-}0\\
 -1&\phantom{-}0&\phantom{-}0&\phantom{-}0&\phantom{-}1&\phantom{-}1&\phantom{-}0&\phantom{-}0&\phantom{-}0&\phantom{-}1\\
 -1&\phantom{-}0&\phantom{-}0&\phantom{-}0&\phantom{-}1&\phantom{-}1&\phantom{-}0&\phantom{-}0&\phantom{-}0&\phantom{-}1\\
 -1&\phantom{-}0&\phantom{-}0&\phantom{-}0&\phantom{-}1&\phantom{-}1&\phantom{-}0&\phantom{-}0&\phantom{-}0&\phantom{-}1\\
  \phantom{-}0&-1&-1&-1&\phantom{-}0&\phantom{-}0&-1&-1&-1&\phantom{-}0\\\hline
  \phantom{-}0&\phantom{-}1&\phantom{-}1&\phantom{-}1&\phantom{-}0&\phantom{-}0&\phantom{-}1&\phantom{-}1&\phantom{-}1&\phantom{-}0\\
 -1&\phantom{-}0&\phantom{-}0&\phantom{-}0&-1&-1&\phantom{-}0&\phantom{-}0&\phantom{-}0&\phantom{-}1\\
 -1&\phantom{-}0&\phantom{-}0&\phantom{-}0&-1&-1&\phantom{-}0&\phantom{-}0&\phantom{-}0&\phantom{-}1\\
 -1&\phantom{-}0&\phantom{-}0&\phantom{-}0&-1&-1&\phantom{-}0&\phantom{-}0&\phantom{-}0&\phantom{-}1\\
  \phantom{-}0&\phantom{-}1&\phantom{-}1&\phantom{-}1&\phantom{-}0&\phantom{-}0&-1&-1&-1&\phantom{-}0\\
\end{array}
\right).
$$
As for $y_2$, $h_2=Ah_1A^{-1}$ is obtained from $h_1$ by exchange of columns and rows 
$1\leftrightarrow 2$, $4\leftrightarrow 5$, $6\leftrightarrow 7$, $9\leftrightarrow 10$,
so  $y_2$ can be obtained from 
$y_1$ in the same way. We readily check that
$$
y_2=\left(
\begin{array}{ccccc|ccccc}
  \phantom{-}0&-1&\phantom{-}0&\phantom{-}1&\phantom{-}0&\phantom{-}0&\phantom{-}1&\phantom{-}0&\phantom{-}1&\phantom{-}0\\
 \phantom{-}1&\phantom{-}0&\phantom{-}1&\phantom{-}0&\phantom{-}1&-1&\phantom{-}0&-1&\phantom{-}0&-1\\
 \phantom{-}0&-1&\phantom{-}0&\phantom{-}1&\phantom{-}0&\phantom{-}0&\phantom{-}1&\phantom{-}0&\phantom{-}1&\phantom{-}0\\
 -1&\phantom{-}0&-1&\phantom{-}0&-1&-1&\phantom{-}0&-1&\phantom{-}0&-1\\
  \phantom{-}0&-1&\phantom{-}0&\phantom{-}1&\phantom{-}0&\phantom{-}0&\phantom{-}1&\phantom{-}0&\phantom{-}1&\phantom{-}0\\\hline
  \phantom{-}0&-1&\phantom{-}0&-1&\phantom{-}0&\phantom{-}0&-1&\phantom{-}0&\phantom{-}1&\phantom{-}0\\
 \phantom{-}1&\phantom{-}0&\phantom{-}1&\phantom{-}0&\phantom{-}1&\phantom{-}1&\phantom{-}0&\phantom{-}1&\phantom{-}0&\phantom{-}1\\
 \phantom{-}0&-1&\phantom{-}0&-1&\phantom{-}0&\phantom{-}0&-1&\phantom{-}0&\phantom{-}1&\phantom{-}0\\
 \phantom{-}1&\phantom{-}0&\phantom{-}1&\phantom{-}0&\phantom{-}1&-1&\phantom{-}0&-1&\phantom{-}0&-1\\
  \phantom{-}0&-1&\phantom{-}0&-1&\phantom{-}0&\phantom{-}0&-1&\phantom{-}0&\phantom{-}1&\phantom{-}0\\
\end{array}
\right)
$$
belongs to $\so^*(10)$ and so we are done, as $e_c=l_1y_1+l_2y_2$, and the arguments used for the rank one case apply. 
\end{proof}

\providecommand{\bysame}{\leavevmode\hbox to3em{\hrulefill}\thinspace}

 \end{document}